\documentclass[12pt,letter]{amsart}

\usepackage{amssymb,latexsym, amsmath, amsxtra}
\usepackage[dvips]{graphics}
\usepackage{xypic}
\usepackage{verbatim}
\usepackage[abs]{overpic}

\theoremstyle{plain}
        \newtheorem{theorem}{Theorem}[section]
        \newtheorem*{theorem*}{Theorem}
        \newtheorem*{conj*}{Conjecture}
        \newtheorem{lemma}[theorem]{Lemma}
        \newtheorem{prop}[theorem]{Proposition}
        \newtheorem{cor}[theorem]{Corollary}

\theoremstyle{definition}
        \newtheorem{definition}[theorem]{Definition}
        \newtheorem{rem}[theorem]{Remark}
         \newtheorem{rems}[theorem]{Remarks}

\theoremstyle{remark}

\numberwithin{equation}{section}
\numberwithin{theorem}{section}
\numberwithin{table}{section}
\numberwithin{figure}{section}

\providecommand{\defn}[1]{\emph{#1}}

\newcommand{\diam}  {\operatorname{diam}}

\newcommand{\inte}  {\operatorname{int}}

\newcommand{\id} {\operatorname{id}}

\newcommand{\card} {\operatorname{card}}
\newcommand{\Ind}{\operatorname{Ind}}

\newcommand{\R}{\mathbb{R}}
\newcommand{\B}{\mathbb{B}}    
\newcommand{\C}{\mathbb{C}}      

\newcommand{\N}{\mathbb{N}}      
\newcommand{\Z}{\mathbb{Z}}      

\newcommand{\D}{\mathbb{D}}      
\providecommand{\abs}[1]{\lvert#1\rvert}
\providecommand{\Abs}[1]{\left|#1\right|}
\providecommand{\norm}[1]{\bigg\|#1\bigg\|}
\providecommand{\Norm}[1]{\left\|#1\right\|}
\renewcommand{\:}{\colon}

\def\({\left(}
\def\){\right)}
\def\[{\left[}
\def\]{\right]}

\newcommand{\crit}{\operatorname{crit}}
\newcommand{\post}{\operatorname{post}}

\newcommand{\CC}{\mathcal{C}}

 \newcommand{\DD}{\mathbf{D}}

\newcommand{\PP}{\mathbf{P}}

\newcommand{\X} {\mathbf{X}}

\newcommand{\E} {\mathbf{E}}

\newcommand{\V} {\mathbf{V}}

\newcommand{\CCC}{C}

\newcommand{\PPP}{\mathcal{P}}

\newcommand{\MMM}{\mathcal{M}}

\begin{document}
\title[Periodic Points and the Measure of Maximal Entropy]{Periodic Points and the Measure of Maximal Entropy of an Expanding Thurston Map}
\author{Zhiqiang~Li}
\address{Department of Mathematics, UCLA, Los Angeles CA 90095-1555}
\email{lizq@math.ucla.edu}

\date{\today}



\begin{abstract}
In this paper, we show that each expanding Thurston map $f\: S^2\rightarrow S^2$ has $1+\deg f$ fixed points, counted with appropriate weight, where $\deg f$ denotes the topological degree of the map $f$. We then prove the equidistribution of preimages and of (pre)periodic points with respect to the unique measure of maximal entropy $\mu_f$ for $f$. We also show that $(S^2,f,\mu_f)$ is a factor of the left shift on the set of one-sided infinite sequences with its measure of maximal entropy, in the category of measure-preserving dynamical systems. Finally, we prove that $\mu_f$ is almost surely the weak$^*$ limit of atomic probability measures supported on a random backward orbit of an arbitrary point.
\end{abstract}

\maketitle

\tableofcontents

\section{Introduction}

A basic theme in dynamics is the investigation of the measure-theoretic entropy and its maximizing measures known as the measures of maximal entropy. By the pioneering work of R.~Bowen, D.~Ruelle, P.~Walters, Ya.~Sinai, M.~Lyubich, R.~Ma\~n\'e and many others, existence and uniqueness results of the measure of maximal entropy are known for uniformly expansive continuous dynamical systems, distance expanding continuous dynamical systems, uniformly hyperbolic smooth dynamical systems and rational maps on the Riemann sphere. In many cases, the measure of maximal entropy is also the asymptotic distribution of the period points (see \cite{Pa64, Si72,Bo75,Ly83,FLM83,Ru89,PU10}).

In this paper, we discuss a class of dynamical systems that are not among the classical dynamical systems mentioned above, namely, expanding Thurston maps on a topological $2$-sphere $S^2$. Thurston maps are branched covering maps on a sphere $S^2$ that generalize rational maps with finitely many post-critical points on the Riemann sphere. More precisely, a branched covering map $f\: S^2 \rightarrow S^2$ is a \emph{Thurston map} if its topological degree is at least 2 and if each of its finitely many critical points is preperiodic. These maps arose in W.P.~Thurston's study of a characterization of rational maps on the Riemann sphere in a general topological context (see \cite{DH93}). We will give a more detailed introduction to Thurston maps in Section~\ref{sctThurstonMap}.

In order to obtain the existence and uniqueness of the measure of maximal entropy of an Thurston map, one has to impose a condition of expansion for the map. More generally, P.~Ha\"issinsky and K.~Pilgrim introduced a notion of expansion for any finite branched coverings between two suitable topological spaces (see \cite[Section~2.1 and Section~2.2]{HP09}). We will use an equivalent definition in the context of Thurston maps formulated by M.~Bonk and D.~Meyer in \cite{BM10}. We will discuss the precise definition in Section~\ref{sctThurstonMap}. We call Thurston maps with this property \emph{expanding Thurston maps}. For a list of equivalent definitions of expanding Thurston maps, we refer to \cite[Proposition~8.2]{BM10}.

As mentioned earlier, words like ``expanding'' and ``expansive'' have been used in different contexts to describe various notions of expansion. Our notion of expansion differs from all of the classical notions (except that of \cite{HP09}), with the closest notion being that of \emph{piecewise expanding maps} from \cite{BS03}. Even though an expanding Thurston map $f$ is expanding, in the sense of \cite{BS03}, when restricted to any $1$-tile $X$ in the cell decompositions that we will discuss in Section~\ref{sctThurstonMap}, it is still not clear why $f$ is piecewise expanding in their sense. 

As a consequence of their general results in \cite{HP09},  P.~Ha\"issinsky and K.~Pilgrim proved that for each expanding Thurston map, there exists a measure of maximal entropy and that the measure of maximal entropy is unique for an expanding Thurston map without periodic critical points. M.~Bonk and D.~Meyer then proved the existence and uniqueness of the measure of maximal entropy for all expanding Thurston maps using an explicit combinatorial construction \cite{BM10}.

In \cite{BM10}, M.~Bonk and D.~Meyer studied various properties of expanding Thurston maps and gave a wealth of combinatorial and analytical tools for such maps. Using the framework set in \cite{BM10}, we investigate in this paper the numbers and locations of the fixed points, periodic points, and preperiodic points of such maps. We establish equidistribution results of preimages of any point, of preperiodic points, and of periodic points, with respect to the measure of maximal entropy. We also generalize some of the results from \cite{BM10} in the development of this paper.

\smallskip

We will now give a brief description of the structure and main results of this paper.

After fixing notation in Section~\ref{sctNotation}, we introduce Thurston maps in Section~\ref{sctThurstonMap} and record, in some cases generalize, a few key concepts and results from \cite{BM10}.

In Section~\ref{sctFixedPts}, we study the fixed points, periodic points, and preperiodic points of the expanding Thurston maps. For the convenience of the reader, we first provide a direct proof in Proposition~\ref{propNoFixedPtsRational}, using knowledge from complex dynamics, of the fact that a rational expanding Thurston map $R$ on the Riemann sphere has exactly $1+\deg R$ fixed points. Then we set out to generalize this result to the class of expanding Thurston maps, and get our first main theorem.

\begin{theorem}  \label{thmNoFixedPts}
Every expanding Thurston map $f\: S^2 \rightarrow S^2$ has $1+\deg f$ fixed points, counted with weight given by the local degree of the map at each fixed point.
\end{theorem}

Here $\deg f$ denotes the topological degree of the map $f$. The local degree is a natural weight for points on $S^2$ for expanding Thurston maps. P.~Ha\"issinsky and K.~Pilgrim also used the same weight in the general context they considered in \cite{HP09}. For a more detailed discussion on the local degree, we refer to Section~\ref{sctThurstonMap}.  

We first observe that the statement of Theorem~\ref{thmNoFixedPts} agrees with what can be concluded from the Lefschetz fixed-point theorem (see for example, \cite[Chapter~3]{GP10})  if the map $f$ is smooth and the graph of $f$ intersects the diagonal of $S^2\times S^2$ transversely at each fixed point of $f$. However, an expanding Thurston map may not satisfy either of these conditions. It is not clear how to give a proof by using the Lefschetz fixed-point theorem. 

The proof of Theorem~\ref{thmNoFixedPts} that we adopt here is quite different from that of the rational case. It uses the correspondence between the fixed points of $f$ and the $1$-tiles in some cell decomposition of $S^2$ induced by $f$ and its invariant Jordan curve $\CC\subseteq S^2$, for the special case when $f$ has a special invariant Jordan curve $\CC$. In fact, $f$ may not have such a Jordan curve, but by a main result of \cite{BM10}, for each $n$ large enough there exists an $f^n$-invariant Jordan curve. We will need a slightly stronger result as formulated in Lemma~\ref{lmCexists}. Then the general case follows from an elementary number-theoretic argument. One of the advantages of this proof is that we also exhibit an almost one-to-one correspondence between the fixed points and the $1$-tiles in the cell decomposition of $S^2$, which leads to precise information on the location of each fixed point. This information is essential later in the proof of the equidistribution of preperiodic and periodic points of expanding Thurston maps in Section~\ref{sctEquidistribution}. As a corollary of Theorem~\ref{thmNoFixedPts}, we give a formula in Corollary~\ref{corNoPrePeriodicPts} for the number of preperiodic points when counted with the corresponding weight.

In Section~\ref{sctEquidistribution}, the concepts of topological entropy, measure-theoretic entropy, and the measure of maximal entropy are reviewed. Then a number of equidistribution results are proved. More precisely, we first prove in Theorem~\ref{thmWeakConv} the equidistribution of the $n$-tiles in the tile decompositions discussed in Section~\ref{sctThurstonMap}  with respect to the measure of maximal entropy $\mu_f$ of an expanding Thurston map $f$. The proof uses a combinatorial characterization of $\mu_f$ due to M.~Bonk and D.~Meyer \cite{BM10} that we will state explicitly in Theorem~\ref{thmBMCharactMOME}.

We then formulate the equidistribution of preimages with respect to the measure of maximal entropy $\mu_f$ in Theorem~\ref{thmWeakConvPreImg} below. Here we denote by $\delta_x$ the Dirac measure supported on a point $x$ in $S^2$. 

\begin{theorem}[Equidistribution of preimages]  \label{thmWeakConvPreImg}
Let $f\: S^2 \rightarrow S^2$ be an expanding Thurston map with its measure of maximal entropy $\mu_f$. Fix $p\in S^2$ and define the Borel probability measures
\begin{equation}  \label{eqDistrPreImg}
\nu_i=\frac{1}{(\deg f)^i}\sum\limits_{q\in f^{-i}(p)} \deg_{f^i}(q) \delta_q, \qquad \widetilde{\nu}_i=\frac{1}{Z_i}\sum\limits_{q\in f^{-i}(p)}\delta_q,
\end{equation}
for each $i\in\N_0$, where $Z_i=\card\(f^{-i}(p)\)$. Then
\begin{equation} \label{eqWeakConvPreImgWithWeight}
\nu_i \stackrel{w^*}{\longrightarrow} \mu_f \text{ as } i\longrightarrow +\infty,
\end{equation}
\begin{equation} \label{eqWeakConvPreImgWoWeight}
\widetilde{\nu}_i \stackrel{w^*}{\longrightarrow} \mu_f \text{ as } i\longrightarrow +\infty.
\end{equation}
\end{theorem}
Here $\deg_{f^i}(x)$ denotes the local degree of the map $f^i$ at a point $x\in S^2$. In (\ref{eqWeakConvPreImgWithWeight}), (\ref{eqWeakConvPreImgWoWeight}), and similar statements below, the convergence of Borel measures is in the weak$^*$ topology, and we use $w^*$ to denote it. Note that the difference of $\nu_i$ and $\widetilde\nu_i$ is the weight at each preimage of $p$ under $f^i$. As mentioned earlier, the local degree is a natural weight for a point in $S^2$ in the context of Thurston maps. On the other hand, it is also natural to assign the same weight for each preimage.

After generalizing Lemma~\ref{lmCoverEdgesBM}, which is due to M.~Bonk and D.~Meyer \cite[Lemma~20.2]{BM10}, in Lemma~\ref{lmCoverEdges}, we prove the equidistribution of preperiodic points with respect to $\mu_f$.

\begin{theorem}[Equidistribution of preperiodic points]  \label{thmWeakConvPrePerPts}
Let $f\: S^2 \rightarrow S^2$ be an expanding Thurston map with its measure of maximal entropy $\mu_f$. For each $m\in\N_0$ and each $n\in\N$ with $m<n$, we define the Borel probability measures
\begin{equation}   \label{eqDistrPrePerPts}
\xi_n^m = \frac{1}{s_n^m} \sum\limits_{f^m(x)=f^n(x)}  \deg_{f^n}(x) \delta_x,   \qquad \widetilde\xi_n^m = \frac{1}{\widetilde s_n^m} \sum\limits_{f^m(x)=f^n(x)} \delta_x,
\end{equation}
where $s_n^m,\widetilde s_n^m$ are the normalizing factors defined in (\ref{eqSetPrePeriodicPts}) and (\ref{eqNoPrePeriodicPts}). If $\{m_n\}_{n\in\N}$ is a sequence in $\N_0$ such that $m_n <n$ for each $n\in\N$, then
\begin{equation}  \label{eqWeakConvPrePerPtsWithWeight}
\xi_n^{m_n}   \stackrel{w^*}{\longrightarrow} \mu_f \text{ as } n\longrightarrow +\infty,
\end{equation}
\begin{equation}    \label{eqWeakConvPrePerPtsWoWeight}
\widetilde \xi_n^{m_n}   \stackrel{w^*}{\longrightarrow} \mu_f \text{ as } n\longrightarrow +\infty.
\end{equation}
\end{theorem}

We prove in Corollary~\ref{corNoPrePeriodicPts} that $s_n^m = (\deg f)^n + (\deg f)^m$ for $m\in\N_0$, $n\in\N$ with $m<n$.

As a special case of Theorem~\ref{thmWeakConvPrePerPts}, we get the equidistribution of periodic points with respect to $\mu_f$.

\begin{cor}[Equidistribution of periodic points]   \label{corWeakConvPerPts}
Let $f\: S^2 \rightarrow S^2$ be an expanding Thurston map with its measure of maximal entropy $\mu_f$. Then
\begin{equation}  \label{eqWeakConvPerPts1}
\frac{1}{1+(\deg f)^n} \sum\limits_{x=f^n(x)}  \deg_{f^n} (x) \delta_x \stackrel{w^*}{\longrightarrow} \mu_f \text{ as } n\longrightarrow +\infty,
\end{equation}
\begin{equation}  \label{eqWeakConvPerPts2}
\frac{1}{\card \{x\in S^2 \,|\, x=f^n(x)\}} \sum\limits_{x=f^n(x)}  \delta_x \stackrel{w^*}{\longrightarrow} \mu_f \text{ as } n\longrightarrow +\infty,
\end{equation}
\begin{equation}  \label{eqWeakConvPerPts3}
\frac{1}{(\deg f)^n} \sum\limits_{x=f^n(x)}  \delta_x \stackrel{w^*}{\longrightarrow} \mu_f \text{ as } n\longrightarrow +\infty.
\end{equation}
\end{cor}

The equidistribution (\ref{eqWeakConvPreImgWithWeight}), (\ref{eqWeakConvPreImgWoWeight}), (\ref{eqWeakConvPerPts1}), and (\ref{eqWeakConvPerPts2}) are analogs of corresponding results for rational maps on the Riemann sphere by M.~Lyubich \cite{Ly83}. Some ideas from \cite{Ly83} are used in the proofs of Theorem~\ref{thmWeakConvPreImg} and Theorem~\ref{thmWeakConvPrePerPts} as well. P.~Ha\"issinsky and K.~Pilgrim also proved (\ref{eqWeakConvPreImgWithWeight}) and (\ref{eqWeakConvPerPts1}) in their general context \cite{HP09}, which includes expanding Thurston maps. 

The equidistribution (\ref{eqWeakConvPrePerPtsWithWeight}) and (\ref{eqWeakConvPrePerPtsWoWeight}) are inspired by the recent work of M.~Baker and L.~DeMarco \cite{BD11}. They used some equidistribution result of preperiodic points of rational maps on the Riemann sphere in the context of arithmetic dynamics. 

We show in Corollary~\ref{corAsympRatioNoPerPts} that for each expanding Thurston map $f$, the exponential growth rate of the cardinality of the set of fixed points of $f^n$ is equal to the topological entropy $h_{\operatorname{top}}(f)$ of $f$, which is known to be equal to $\log(\deg f)$ (see for example, \cite[Corollary~20.8]{BM10}). This is analogous to the corresponding result for expansive homeomorphisms on compact metric spaces with the \emph{specification property} (see for example, \cite[Theorem~18.5.5]{KH95}).

In Section~\ref{sctFactor}, we prove in Theorem~\ref{thmLfactor} that for each expanding Thurston map $f$ with its measure of maximal entropy $\mu_f$, the measure-preserving dynamical system $(S^2,f,\mu_f)$ is a factor, in the category of measure-preserving dynamical systems, of the measure-preserving dynamical system of the left-shift operator on the one-sided infinite sequences of $\deg f$ symbols together with its measure of maximal entropy. This generalizes the corresponding result in \cite{BM10} in the category of topological dynamical systems, reformulated in Theorem~\ref{thmBMfactor}.

Finally, in Section~\ref{sctIteration}, we follow the idea of J.~Hawkins and M.~Taylor \cite{HT03} to prove in Theorem~\ref{thmRandomIntConv} that for each $p\in S^2$, the measure of maximal entropy $\mu_f$ of an expanding Thurston map $f$ is almost surely the limit of
$$
\frac1n \sum\limits_{i=0}^{n-1} \delta_{q_i}
$$
as $n\longrightarrow+\infty$ in the weak* topology, where $q_i$ is one of the points in $f^{-1}(q_{i-1})$, chosen with probability proportional to the weight given by the local degree of $f$ at each point in $f^{-1}(q_{i-1})$, for all $i\in\N$, and $q_0=p$. A similar result for certain hyperbolic rational maps on $S^2$ was proved by M.~Barnsley \cite{Bar88}. J.~Hawkins and M.~Taylor generalized it to any rational map on the Riemann sphere of degree $d\geq 2$ \cite{HT03}.

\bigskip
\noindent
\textbf{Acknowledgments.} The author wants to express his gratitude to M.~Bonk for introducing him to the subject of expanding Thurston maps and patiently teaching and guiding him as an advisor.

\section{Notation} \label{sctNotation}
Let $\C$ be the complex plane and $\widehat{\C}$ be the Riemann sphere. Let $\D$ be the open unit disk on $\C$. We use the convention that $\N=\{1,2,3,\dots\}$ and $\N_0 = \{0\} \cup \N$. We always use base $e$ for logarithm unless otherwise specified.

The cardinality of a set $A$ is denoted by $\card{A}$. For each $x\in\R$, we define $\lfloor x\rfloor$ as the greatest integer smaller than or equal to $x$, and $\lceil x \rceil$ the smallest integer greater than or equal to $x$.

Let $(X,d)$ be a metric space. For subsets $A,B\subseteq X$, we set $d(A,B)=\sup \{d(x,y)\,|\, x\in A,y\in B\}$, and $d(A,x)=d(x,A)=d(A,\{x\})$ for $x\in X$. For each subset $S\subseteq X$, we denote the diameter of $S$ by $\diam_d(S)=\sup\{d(x,y)\,|\,x,y\in S\}$. For each $r>0$, we set $N^r_d(A)$ to be the open $r$-neighborhood $\{y\in X \,|\, d(y,A)<r\}$ of $A$, and $\overline{N^r_d}(A)$ the closed $r$-neighborhood $\{y\in X \,|\, d(y,A)\leq r\}$ of $A$. The identity map $\id_X\: X\rightarrow X$ maps each $x\in X$ to $x$ itself. We denote by $\CCC(X)$ the space of continuous functions from $X$ to $\R$, by $\MMM(X)$ the set of finite signed Borel measures, and $\PPP(X)$ the set of Borel probability measures on $X$. We use $\Norm{\cdot}$ to denote the total variation norm on $\MMM(X)$. For a point $x\in X$, we define $\delta_x$ as the Dirac measure supported on $\{x\}$. For $g\in\CCC(X)$ we set $\MMM(X,g)$ to be the set of $g$-invariant Borel probability measures on $X$.

\section{Thurston maps} \label{sctThurstonMap}

In this section, we briefly review some key concepts and results on Thurston maps, and expanding Thurston maps in particular. For a more thorough treatment of the subject, we refer to \cite{BM10}. Towards the end of this section, we state and prove a slightly stronger version of one of the main theorems in \cite{BM10}, which we will repeatedly use in the following sections.

Let $S^2$ denote an oriented topological $2$-sphere. A continuous map $f\:S^2\rightarrow S^2$ is called a \defn{branched covering map} on $S^2$ if for each point $x\in S^2$, there exists $d\in \N$, open neighborhoods $U$ of $x$ and $V$ of $y=f(x)$, $U'$ and $V'$ of $0$ in $\widehat{\C}$, and orientation-preserving homeomorphisms $\varphi\:U\rightarrow U'$ and $\eta\:V\rightarrow V'$ such that $\varphi(x)=0$, $\eta(y)=0$ and
$$
(\eta\circ f\circ\varphi^{-1})(z)=z^d
$$
for each $z\in U'$. The positive integer $d$ above is called the \defn{local degree} of $f$ at $p$ and is denoted by $\deg_f (p)$. The \defn{(global) degree} of $f$ is defined as
\begin{equation}   \label{eqDeg=SumLocalDegree}
\deg f=\sum\limits_{x\in f^{-1}(y)} \deg_f (x)
\end{equation}
for each $y\in S^2$. It is independent of $y\in S^2$. It is true that if $f\:S^2\rightarrow S^2$ and $g\:S^2\rightarrow S^2$ are two branched covering maps on $S^2$, then
\begin{equation} \label{eqLocalDegreeProduct}
 \deg_{f\circ g}(x) = \deg_g(x)\deg_f(g(x)), \qquad \text{for each } x\in S^2,
\end{equation}   
and moreover, 
\begin{equation}  \label{eqDegreeProduct}
\deg(f\circ g) =  (\deg f)( \deg g).
\end{equation}

A point $x\in S^2$ is a \defn{critical point} of $f$ if $\deg_f(x) \geq 2$. The set of critical points of $f$ is denoted by $\crit f$. A point $y\in S^2$ is a \defn{postcritical point} of $f$ if $y \in \bigcup\limits_{n\in\N}\{f^n(x)\,|\,x\in\crit f\}$. The set of postcritical points of $f$ is denoted by $\post f$. Note that $\post f=\post f^n$ for all $n\in\N$.

\begin{definition} [Thurston maps] \label{defThurstonMap}
A Thurston map is a branched covering map $f\:S^2\rightarrow S^2$ on $S^2$ with $\deg f\geq 2$ and $\card(\post f)<+\infty$.
\end{definition}

We define two notions of equivalence for Thurston maps. The first one is the usual topological conjugation. We call the maps $f$ and $g$ \defn{topologically conjugate} if there exists a homeomorphism $h\: S^2\rightarrow S^2 $ such that $h\circ f = g\circ h$. The second one is a weaker notion due to W.P.~Thurston \cite{DH93}.
  
\begin{definition}[Thurston equivalence]\label{defThequiv}
 Two Thurston maps $f\: S^2\rightarrow S^2$ and $g\: S^2\rightarrow  S^2$ are called \defn{Thurston equivalent} if there exist homeomorphisms 
  $h_0,h_1 \:S^2\rightarrow  S^2 $ that are  isotopic rel.\ $\post f$  and satisfy  $    h_0\circ f = g\circ h_1$.
\end{definition} 
For the usual definition of an isotopy, we refer to \cite[Section~3]{BM10}.

We now set up the notation for cell decompositions of $S^2$. A \defn{cell of dimension $n$} in $S^2$, $n \in \{1,2\}$, is a subset $c\subseteq S^2$ that is homeomorphic to the closed unit ball $\overline{\B^n}$ in $\R^n$. We define the \defn{boundary of $c$}, denoted by $\partial c$, to be the set of points corresponding to $\partial\B^n$ under such a homeomorphism between $c$ and $\overline{\B^n}$. The \defn{interior of $c$} is defined to be $\inte c = c \setminus \partial c$. A cell $c$ of dimension 0 is a singleton set $\{x\}$ for some point $x\in S^2$. For cells $c$ with dimension $0$, we adopt the convention that $\partial c=\emptyset$ and $\inte c =c$. 

The following three definitions are from \cite{BM10}.

\begin{definition}[Cell decompositions]\label{defcelldecomp}
Let $\DD$ be a collection of cells in $S^2$.  We say that $\DD$ is a \defn{cell decomposition of $S^2$} if the following conditions are satisfied:

\begin{itemize}

\smallskip
\item[(i)]
the union of all cells in $\DD$ is equal to $S^2$,

\smallskip
\item[(ii)] for $c_1,c_2 \in \DD$ with $c_1 \ne c_1$, we have $\inte c_1 \cap \inte c_2= \emptyset$,  

\smallskip
\item[(iii)] if $c\in \DD$, then $\partial c$ is a union of cells in $\DD$,

\smallskip
\item[(iv)] every point in $S^2$ has a neighborhood that meets only finitely many cells in $\DD$.

\end{itemize}
\end{definition}

\begin{definition}[Refinements]\label{defrefine}
Let $\DD'$ and $\DD$ be two cell decompositions of $S^2$. We
say that $\DD'$ is a \defn{refinement} of $\DD$ if the following conditions are satisfied:
\begin{itemize}

\smallskip
\item[(i)] for every cell $c'\in \DD'$ there exits  a cell $c\in \DD$ with $c'\subseteq c$,

\smallskip
\item[(ii)] every cell $c\in \DD$ is the union of all cells $c'\in \DD'$ with $c'\subseteq c$.

\end{itemize}
\end{definition}

\begin{definition}[Cellular maps and cellular Markov partitions]\label{defcellular}
Let $\DD'$ and $\DD$ be two cell decompositions of  $S^2$. We say that a continuous function $f \: S^2 \rightarrow S^2$ is \defn{cellular} for  $(\DD', \DD)$ if for every cell $c\in \DD'$, the restriction $f|_c$ is a homeomorphism of $c$ onto a cell in $\DD$. We say that $(\DD',\DD)$ is a \defn{cellular Markov partition} for $f$ if $f$ is cellular for $(\DD',\DD)$ and $\DD'$ is a refinement of $\DD$.
\end{definition}

Let $f\:S^2 \rightarrow S^2$ be a Thurston map, and $\CC\subseteq S^2$ be a Jordan curve such that $\post f\subseteq \CC$. Then the pair $(f,\CC)$ induces natural cell decompositions $\DD^n(f,\CC)$ of $S^2$, for $n\in\N_0$, in the following way:

By the Jordan curve theorem, the set $S^2\setminus\CC$ has two connected components. We call the closure of one of them the \defn{white $0$-tile} for $(f,\CC)$, denoted by $X^0_w$, and the closure of the other one the \defn{black $0$-tile} for $(f,\CC)$, denoted by $X^0_b$. The set of \defn{$0$-tiles} is $\X^0(f,\CC)=\{X_b^0,X_w^0\}$. The set of \defn{$0$-vertices} is $\V^0(f,\CC)=\post f$. We define $\overline\V^0(f,\CC)$ to be $\{ \{x\} \,|\, x\in \V^0(f,\CC) \}$. The set of \defn{$0$-edges} $\E^0(f,\CC)$ is the set of connected components of $\CC \setminus  \post f$. Then we get a cell decomposition 
$$
\DD^0(f,\CC)=\X^0(f,\CC) \cup \E^0(f,\CC) \cup \overline\V^0(f,\CC)
$$
of $S^2$ consisting of \defn{$0$-cells}.

One can recursively define, for each $n\in\N$, the unique cell decomposition $\DD^n(f,\CC)$ consisting of \defn{$n$-cells} such that $f$ is cellular for $(\DD^{n+1}(f,\CC),\DD^n(f,\CC))$. For details, we refer to \cite[Lemma~5.4]{BM10}. We denote by $\X^n(f,\CC)$ the set of $n$-cells of dimension 2, called \defn{$n$-tiles}; by $\E^n(f,\CC)$ the set of $n$-cells of dimension 1, called \defn{$n$-edges}; by $\overline\V^n(f,\CC)$ the set of $n$-cells of dimension 0; and by $\V^n(f,\CC)$ the set $\{x\,|\, {x}\in \overline\V^n(f,\CC)\}$, called \defn{$n$-vertices}.

For the convenience of the reader, we record Proposition~6.1 of \cite{BM10} here in order to summarize properties of the cell decompositions $\DD^n(f,\CC)$.

\begin{prop} \label{propCellDecomp}
 Let $k,n\in \N_0$, let   $f\: S^2\rightarrow S^2$ be a Thurston map,  $\CC\subseteq S^2$ be a Jordan curve with $\post f \subseteq \CC$, and   $m=\card(\post f)$. 
 
\smallskip
\begin{itemize}

\smallskip
\item[(i)] The map  $f^k$ is cellular for $(\DD^{n+k}(f,\CC), \DD^n(f,\CC))$. In particular, if  $c$ is any $(n+k)$-cell, then $f^k(c)$ is an $n$-cell, and $f^k|_c$ is a homeomorphism of $c$ onto $f(c)$.

\smallskip
\item[(ii)]  Let  $c$ be  an $n$-cell.  Then $f^{-k}(c)$ is equal to the union of all 
$(n+k)$-cells $c'$ with $f^k(c')=c$.

\smallskip
\item[(iii)] The  $0$-skeleton of $\DD^n(f,\CC)$ is the set $\V^n(f,\CC)=f^{-n}(\post f )$, and we have $\V^n(f,\CC) \subseteq \V^{n+k}(f,\CC)$.  The $1$-skeleton of $\DD^n(f,\CC)$ is  equal to  $f^{-n}(\CC)$. 

\smallskip
\item[(iv)] $\card (\V^n(f,\CC)) \leq m (\deg f)^n$,  $\card(\E^n(f,\CC))=m(\deg f)^n$,  and 
$\card(\X^n(f,\CC))=2(\deg f)^n$.

\smallskip
\item[(v)] The $n$-edges are precisely the closures of the connected components of $f^{-n}(\CC)\setminus f^{-n}(\post f )$. The $n$-tiles are precisely the closures of the connected components of $S^2\setminus f^{-n}(\CC)$.

\smallskip
\item[(vi)] Every $n$-tile  is an $m$-gon, i.e., the number of $n$-edges and the number of $n$-vertices contained in its boundary are equal to $m$.  

\end{itemize}
\end{prop}

Here the \defn{$n$-skeleton}, for $n\in\{0,1,2\}$, of a cell decomposition of $S^2$ is the union of all $n$-cells in this cell decomposition.

For $n\in \N_0$, we define \defn{the set of black $n$-tiles} as
$$
\X_b^n(f,\CC)=\{X\in\X^n(f,\CC) \, |\,  f^n(X)=X_b^0\},
$$
and the \defn{set of white $n$-tiles} as
$$
\X_w^n(f,\CC)=\{X\in\X^n(f,\CC) \, |\, f^n(X)=X_w^0\}.
$$
Moreover, for $n\in\N$, we define \defn{the set of black $n$-tiles contained in a white $(n-1)$-tile} as
$$
\X_{bw}^n(f,\CC) = \{ X\in \X_b^n(f,\CC) \, |\, \exists X'\in \X_w^{n-1}(f,\CC), \, X\subseteq X' \},
$$
\defn{the set of black $n$-tiles contained in a black $(n-1)$-tile} as
$$
\X_{bb}^n(f,\CC) = \{ X\in \X_b^n(f,\CC) \, |\, \exists X'\in \X_b^{n-1}(f,\CC),\, X\subseteq X' \},
$$
\defn{the set of white $n$-tiles contained in a black $(n-1)$-tile} as
$$
\X_{wb}^n(f,\CC) = \{ X\in \X_w^n(f,\CC) \, |\, \exists X'\in \X_b^{n-1}(f,\CC),\, X\subseteq X' \},
$$
\defn{and the set of white $n$-tiles contained in a while $(n-1)$-tile} as
$$
\X_{ww}^n(f,\CC) = \{ X\in \X_w^n(f,\CC) \, |\, \exists X'\in \X_w^{n-1}(f,\CC),\, X\subseteq X' \}.
$$
In other words, for example, a black $n$-tile is an $n$-tile that is mapped by $f^n$ to the black $0$-tile, and a black $n$-tile contained in a white $(n-1)$-tile is an $n$-tile that is contained in some white $(n-1)$-tile as a set, and is mapped by $f^n$ to the black $0$-tile.

From now on, we will say the cell decompositions induced by the pair $(f,\CC)$ and induced by $f$ and $\CC$ interchangeably. If the pair $(f,\CC)$ is clear from the context, we will sometimes omit $(f,\CC)$ in the notation above.

We now define two notions of expansion by M.~Bonk and D.~Meyer \cite{BM10}.

It is proved in \cite[Corollary~6.4]{BM10} that for each expanding Thurs\-ton map $f$ (see Definition \ref{defExpanding} below), we have $\card(\post f) \geq 3$.

\begin{definition}[Joining opposite sides]  \label{defConnectop} 
Fix a Thurston map $f$ with $\card(\post f) \geq 3$ and an $f$-invariant Jordan curve $\CC$ containing $\post f$.  A set $K\subseteq S^2$ \defn{joins opposite sides} of $\CC$ if $K$ meets two disjoint $0$-edges when $\card( \post f)\geq 4$, or $K$ meets  all  three $0$-edges when $\card(\post f)=3$. 
 \end{definition}

\begin{definition}[Combinatorial expansion]\label{defCombExpanding}
Let $f$ be a Thurston map.  We say that $f$ is \defn{combinatorially expanding} if $\card (\post f)\geq 3$, and there exists an $f$-invariant Jordan curve $\CC\subseteq S^2$ (i.e., $f(\CC)\subseteq \CC$) with $\post f \subseteq \CC$, and there exists a number $n_0\in \N$ such that none of the $n_0$-tiles in $\X^{n_0}(f,\CC)$ joins opposite sides of $\CC$. 
\end{definition} 
 
\begin{definition} [Expansion] \label{defExpanding}
A Thurston map $f\:S^2\rightarrow S^2$ is called \defn{expanding} if there exist a metric $d$ on $S^2$ that induces the standard topology on $S^2$ and a Jordan curve $\CC\subseteq S^2$ containing $\post f$ such that $\lim\limits_{n\to+\infty}\max \{\diam_d(X) \,|\, X\in \X^n(f,\CC)\}=0$.
\end{definition}

\begin{rems}  \label{rmExpanding}
We observe that being expanding is a purely topological property of a Thurston map and independent of the choice of the metric $d$ that generates the standard topology on $S^2$. By Lemma~8.1 in \cite{BM10}, it is also independent of the choice of the Jordan curve $\CC$ containing $\post f$. More precisely, if $f$ is an expanding Thurston map, then
$$
\lim\limits_{n\to+\infty}\max \{\diam_{\widetilde{d}}(X) \,|\, X\in \X^n(f,\widetilde\CC)\}=0,
$$
for each metric $\widetilde{d}$ that generates the standard topology on $S^2$ and each Jordan curve $\widetilde\CC\subseteq S^2$ that contains $\post f$. From the definition, it is also clear that if $f$ is an expanding Thurston map, so is $f^n$ for each $n\in\N$.
\end{rems}

P. Ha\"{\i}ssinsky and K. Pilgrim developed a more general notion of expansion for finite branched coverings between two Hausdorff, locally compact, locally connected topological spaces (see \cite[Section~2.1 and Section~2.2]{HP09}). When restricted to Thurston maps, their notion of expansion is equivalent to our notion defined above (see \cite[Proposition~8.2]{BM10}). Such notions of expansion are the natural analogs in the context of finite branched coverings and Thurston maps to some of the more classical notions of expansion, such as expansive homeomorphisms and forward-expansive continuous maps between compact metric spaces (see for example, \cite[Definition~3.2.11]{KH95}), and distance-expanding maps between compact metric spaces (see for example, \cite[Chapter~4]{PU10}). Our notion of expansion is not equivalent to any of such classical notions in the context of Thurston maps.

M.~Bonk and D.~Mayer proved that if two expanding Thurston maps are Thurston equivalent, then they are topologically conjugate (see \cite[Theorem~10.4]{BM10}).

For an expanding Thurston map $f$, we can fix a metric $d$ for $f$ on $S^2$ called a visual metric. For the existence and properties of such metrics, see \cite[Chapter~8]{BM10}. In particular, we will need the fact that $d$ induces the standard topology on $S^2$ (\cite[Proposition~8.9]{BM10}). One major advantage of visual metrics $d$ is that in $(S^2,d)$ we have good quantitative control over the sizes of the cells in the cell decompositions discussed above, see \cite[Lemma~8.10]{BM10}. More precisely,

\begin{lemma}[M.~Bonk \& D.~Meyer, 2010]  \label{lmBMCellSizeBounds}
Let $f\: S^2\rightarrow S^2$ be an expanding Thurston map, $\CC\subseteq S^2$ be a Jordan curve with $\post f\subseteq \CC$, and $d$ a visual metric for $f$. Then there exists a constant $\Lambda>1$ called the expansion factor, and a constant $C\geq 1$ such that for each $n\in\N_0$,

\begin{itemize}

\smallskip

\item[(i)] $d(\delta,\tau) \geq \frac{1}{C} \Lambda^{-n}$ whenever $\delta$ and $\tau$ are disjoint $n$-cells,

\smallskip

\item[(ii)] $\frac{1}{C} \Lambda^{-n} \leq \diam_d(\tau) \leq C\Lambda^{-n}$ for all $n$-edges and all $n$-tiles $\tau$.

\end{itemize}
\end{lemma}

A Jordan curve $\CC\subseteq S^2$ is $f$-invariant if $f(\CC)\subseteq \CC$. We are interested in $f$-invariant Jordan curves that contain $\post f$, since for such a curve $\CC$, the partition $(\DD^1(f,\CC),\DD^0(f,\CC))$ is then a cellular Markov partition for $f$. According to Example~15.5 in \cite{BM10}, $f$-invariant Jordan curves containing $\post{f}$ need not exist. However, M.~Bonk and D.~Meyer proved in \cite[Theorem~1.2]{BM10} that for each sufficiently large $n$ depending on $f$, an $f^n$-invariant Jordan curve $\CC$ containing $\post{f}$ always exists. We will need a slightly stronger version in this paper. Its proof is almost the same as that of \cite[Theorem~1.2]{BM10}. For the convenience of the reader, we include the proof here.

\begin{lemma}  \label{lmCexists}
Let $f\:S^2\rightarrow S^2$ be an expanding Thurston map, and $\widetilde{\CC}\subseteq S^2$ be a Jordan curve with $\post f\subseteq \widetilde{\CC}$. Then there exists an integer $N(f,\widetilde{\CC}) \in \N$ such that for each $n\geq N(f,\widetilde{\CC})$ there exists an $f^n$-invariant Jordan curve $\CC$ isotopic to $\widetilde{\CC}$ rel.\ $\post f$ such that no $n$-tile in $\DD^n(f,\CC)$ joins opposite sides of $\CC$.
\end{lemma}

\begin{proof}
By \cite[Lemma~15.9]{BM10}, there exists an integer $N(f,\widetilde{\CC})\in\N$ such that for each $n \geq N(f,\widetilde{\CC})$, there exists a Jordan curve $\CC' \subseteq f^{-n} (\widetilde{\CC})$ that is isotopic to $\widetilde{\CC}$ rel.\ $\post f$, and no $n$-tile for $(f,\widetilde{\CC})$ joins opposite sides of $\CC'$. Let $H\:S^2 \times [0,1] \rightarrow S^2$ be this isotopy rel.\ $\post f$. We set $H_t(x)=H(x,t)$ for $x\in S^2, t\in [0,1]$. We have $H_0=\id_{S^2}$ and $\CC'=H_1(\widetilde{\CC}) \subseteq f^{-n}(\widetilde{\CC})$.

If we set $F=f^n$, then $\post F=\post f$ and $F$ is also an expanding Thurston map (\cite[Lemma~8.4]{BM10}). Note that $F$ is cellular for $(\DD^n(f,\widetilde{\CC}), \DD^0(f,\widetilde{\CC}))$. So $\DD^1(F,\widetilde{\CC})=\DD^n(f,\widetilde{\CC})$ (see \cite[Lemma~5.4]{BM10}). Thus no $1$-cell for $(H_1 \circ F, \CC')$ joins opposite sides of $\CC'$, and thus $H_1 \circ F$ is combinatorially expanding for $\CC'$. Note that $\CC'$ contains $\post(H_1\circ F)=\post F=\post f$. By Corollary~13.18 in \cite{BM10}, there exists a homeomorphism $\phi\:S^2\rightarrow S^2$ that is isotopic to the identity rel.\ $\post{(H_1\circ F)}$ such that $\phi(\CC')=\CC'$ and $G=\phi \circ H_1 \circ F$ is an expanding Thurston map. Since $\phi\circ H_1$ is isotopic to the identity on $S^2$ rel.\ $\post F$, the pair $F$ and $G$ are Thurston equivalent. By Theorem~10.4 in \cite{BM10}, there exists a homeomorphism $h\:S^2\rightarrow S^2$ that is isotopic to the identity on $S^2$ rel.\ $F^{-1}(\post F)$ with $F\circ h=h\circ G$. Set $\CC = h(\CC')$. Then $\CC$ is a Jordan curve in $S^2$ that is isotopic to $\CC'$ rel.\ $F^{-1}(\post F)$ and thus isotopic to $\widetilde{\CC}$ rel.\ $\post F$. Since $F(\CC)=F(h(\CC'))=h(G(\CC'))=h(\phi(H_1(F(\CC')))) \subseteq h(\phi(\CC'))=h(\CC')=\CC$, we get that $\CC$ is $F$-invariant.

Moreover, since no $1$-cell for $(H_1 \circ F, \CC')$ joins opposite sides of $\CC'$, $H_1\circ F (\CC')\subseteq H_1(\widetilde{\CC})=\CC'$, $\phi\:S^2 \rightarrow S^2$ is a homeomorphism with $\phi(\CC')=\CC'$, $G=\phi \circ H_1 \circ F$, we can conclude that $G(\CC') \subseteq \CC'$ and no $1$-cell for $(G,\CC')$ joins opposite sides of $\CC'$.  Since $h\:S^2\rightarrow S^2$ is a homeomorphism, $\CC=h(\CC')$, and $F\circ h=h \circ G$, we can finally conclude that no $1$-cell for $(F,\CC)$ joins opposite sides of $\CC$. Therefore no $n$-cell for $(f,\CC)$ joins opposite sides of $\CC$.
\end{proof}

In fact, we will only need the following corollary of Lemma~\ref{lmCexists} in the following sections.

\begin{cor}  \label{corCexists}
Let $f\:S^2\rightarrow S^2$ be an expanding Thurston map. Then there exists a constant $N(f)>0$ such that for each $n \geq N(f)$, there exists an $f^n$-invariant Jordan curve $\CC$ containing $\post f$ such that no $n$-tile in $\DD^n(f,\CC)$ joins opposite sides of $\CC$.
\end{cor}

\begin{proof}
We can choose an arbitrary Jordan curve $\widetilde\CC\subseteq S^2$ containing $\post f$ and set $N(f)= N(f,\widetilde\CC)$, and $\CC$ an $f^n$-invariant Jordan curve containing $\post f$ as in Lemma~\ref{lmCexists}.
\end{proof}

\begin{lemma}   \label{lmPreImageDense}
Let $f\:S^2\rightarrow S^2$ be an expanding Thurston map. Then for each $p\in S^2$, the set $\bigcup\limits_{n=1}^{+\infty}  f^{-n}(p)$ is dense in $S^2$, and
\begin{equation}   \label{eqCardn-preimgGoToInfty}
\lim\limits_{n\to +\infty}  \card(f^{-n}(p))  = +\infty.
\end{equation}
\end{lemma}

\begin{proof}
Let $\CC \subseteq S^2$ be a Jordan curve containing $\post f$. Let $d$ be any metric on $S^2$ that generates the standard topology on $S^2$.

Without loss of generality, we assume that $p\in X^0_w$ where $X^0_w \in \X^0_w(f,\CC)$ is the white $0$-tile in the cell decompositions induced by $(f,\CC)$. The proof for the case when $p\in X^0_b$ where $X^0_b \in \X^0_b(f,\CC)$ is the black $0$-tile is similar.

By Proposition~\ref{propCellDecomp}(ii), for each $n\in\N$ and each white $n$-tile $X^n_w\in\X^n_w(f,\CC)$, there is a point $q\in X^n_w$ with $f^n(q)=p$. Since $f$ is an expanding Thurston map,
\begin{equation}  \label{eqMeshGoTo0}
\lim\limits_{n\to+\infty}\max \{\diam_d(X) \,|\, X\in \X^n(f,\CC)\}=0.
\end{equation}
Then the density of the set $\bigcup\limits_{n=1}^{+\infty}  f^{-n}(p)$ follows from the observation that for each $n\in\N$, each black $n$-tile $X^n_b\in\X^n_b(f,\CC)$ intersects nontrivially with some white $n$-tile $X^n_w \in \X^n_w(f,\CC)$.

By the above observation, the triangular inequality, and the fact that $\diam_d(S^2) > 0$ and $S^2$ is connected in the standard topology, the equation (\ref{eqCardn-preimgGoToInfty}) follows from (\ref{eqMeshGoTo0}).
\end{proof}

\section{Fixed points of expanding Thurston maps}  \label{sctFixedPts}

The main goal of this section is to prove Theorem~\ref{thmNoFixedPts}; namely, that the number of fixed points, counted with an appropriate weight, of an expanding Thurston map $f$ is exactly $1+\deg f$. In order to prove Theorem~\ref{thmNoFixedPts}, we first establish in Lemma~\ref{lmAtLeast1} and Lemma~\ref{lmAtMost1} an almost one-to-one correspondence between fixed points and $1$-tiles in the cell decomposition $\DD^1(f,\CC)$ for an expanding Thurston map $f$ with an $f$-invariant Jordan curve $\CC$ containing $\post f$. As a consequence, we establish in Corollary~\ref{corNoPrePeriodicPts} an exact formula for the number of preperiodic points, counted with appropriate weight. We end this section by establishing a formula for the exact number of periodic points with period $n$, $n\in\N$, for expanding Thurston maps without periodic critical points.

Let $f$ be a Thurston map and $p\in S^2$ a periodic point of $f$ of period $n\in\N$, we define \defn{the weight of $p$ (with respect to $f$)} as the local degree $\deg_{f^n} (p)$ of $f^n$ at $p$. When $f$ is understood from the context and $p$ is a fixed point of $f$, we abbreviate it as  \defn{the weight of $p$}. We will prove in this section that each expanding Thurston map $f$ has exactly $1+\deg f$ fixed points, counted with weight.

Note the difference between the weight and \defn{the multiplicity} of a fixed point of a rational map (see \cite[Chapter 12]{Mi06}). In comparison, the multiplicity of a fixed point $p\in\C$ of a rational map $R\:\widehat\C\rightarrow\widehat\C$ is $\deg_{\widetilde R}(p)$, where $\widetilde R(z) = R(z)-z$. For every expanding rational Thurston map $R$, M.~Bonk and D.~Meyer proved that $R$ has no periodic critical points (see \cite[Proposition~19.1]{BM10}). So the weight of every fixed point of $R$ is 1. We can prove that $R$ has exactly $1+\deg R$ fixed points by using basic facts in complex dynamics, even though it will follow as a special case of our general result in Theorem~\ref{thmNoFixedPts}. For the relevant definitions and general background of complex dynamics, see \cite{CG93} and \cite{Mi06}.

\begin{prop}   \label{propNoFixedPtsRational}
Let $R\:\widehat\C\rightarrow\widehat\C$ be a expanding rational Thurston map, then $R$ has exactly $1+\deg R$ fixed points. Moreover, the weight $\deg_R(q)$ of each fixed point $q$ of $R$ is equal to 1.
\end{prop}

\begin{proof}
Conjugating $R$ by a fractional linear automorphism of the Riemann sphere if necessary, we may assume that the point at infinity is not a fixed point of $R$.

Since $R$ is expanding, $R$ is not the identity map. By Lemma~12.1 in \cite{Mi06}, which is basically an application of the fundamental theorem of algebra, we can conclude that $R$ has $1+\deg R$ fixed points, counted with multiplicity. For rational Thurston maps, being expanding is equivalent to having no periodic critical points (see \cite[Proposition~19.1]{BM10}). So the weight $\deg_R(q)$ of every fixed point $q$ of $R$ is exactly 1. Thus it suffices now to prove that each fixed point $q$ of $R$ has multiplicity 1.

Suppose a fixed point $p$ of $R$ has multiplicity $m>1$. In the terminology of complex dynamics, $q$ is then a parabolic fixed point with multiplier $1$ and multiplicity $m$. Then by Leau-Fatou flower theorem (see for example, \cite[Chapter~10]{Mi06} or \cite[Theorem~2.12]{Br10}), there exists an open set $U\subseteq S^2$ such that $f(U)\subseteq U$ and $U\neq S^2$ (by letting $U$ be one of the attracting petals, for example). This contradicts the fact that the function $R$, as an expanding Thurston map, is \emph{eventually onto}, i.e., for each nonempty open set $V\subseteq S^2$, there exists a number $m\in\N$ such that $R^m(V)=S^2$.

In order to see that $R$ is eventually onto, let $d$ be a metric on $S^2$ and $\CC\subseteq S^2$ be a Jordan curve, as given in Definition~\ref{defExpanding}. Since $V$ is open, it contains some open ball in the metric space $(S^2,d)$. Then since $R$ is expanding, by Definition~\ref{defExpanding}, we can conclude that there exists a constant $m\in\N$, a black $m$-tile $X^m_b\in\X^m_b(R,\CC)$ and a white $m$-tile $X^m_w\in\X^m_w(R,\CC)$ such that $X^m_b\cup X^m_w \subseteq V$. Thus $R^m(V)\supseteq R^m(X^m_b\cup X^m_w) = S^2$. Therefore, $R$ is eventually onto.
\end{proof}

For general expanding Thurston maps, we need to use the combinatorial information from \cite{BM10}. Recall that cells in the cell decompositions are by definition closed sets.

\begin{lemma}  \label{lmAtLeast1}
Let $f$ be an expanding Thurston map with an $f$-invariant Jordan curve  $\CC$ containing $\post f$. If $X\in \X^1_{ww}(f,\CC) \cup \X^1_{bb}(f,\CC)$ is a white $1$-tile contained in the while $0$-tile $X^0_w$ or a black $1$-tile contained in the black $0$-tile $X^0_b$, then $X$ contains at least one fixed point of $f$. If $X\in \X^1_{wb}(f,\CC) \cup \X^1_{bw}(f,\CC)$ is a white $1$-tile contained in the black $0$-tile $X^0_b$ or a black $1$-tile contained in the white $0$-tile $X^0_w$, then $\inte X$ contains no fixed points of $f$.
\end{lemma}

Recall the set of $0$-tiles $\X^0(f,\CC)$ consists of the white $0$-tile $X^0_w$ and the black $0$-tile $X^0_b$.

\begin{proof}
If $X\in \X^1_{ww}(f,\CC) \cup \X^1_{bb}(f,\CC)$, then $X\subseteq f(X)$. By Proposition~\ref{propCellDecomp}(i), $f|_X$ is a homeomorphism from $X$ to $f(X)$, which is one of the two $0$-tiles. Hence, $f(X)$ is homeomorphic to the closed unit disk. So by Brouwer's fixed point theorem, $(f|_X)^{-1}$ has a fixed point $p$. Thus $p$ is also a fixed point of $f$.

If $X\in \X^1_{wb}(f,\CC)$, then $\inte X \subseteq \inte X_b^0$ and $f(X)=X_w^0$. Since $ X_w^0 \cap \inte X_b^0 = \emptyset$, the map $f$ has no fixed points in $\inte X$. The case when $X\in \X^1_{bw}(f,\CC)$ is similar.
\end{proof}

\begin{lemma}   \label{lmAtMost1}
Let $f$ be an expanding Thurston map with an $f$-invariant Jordan curve  $\CC$ containing $\post f$ such that no $1$-tile in $\DD^1(f,\CC)$ joins opposite sides of $\CC$. Then for every $n\in\N$, each $n$-tile $X^n \in\X^n(f,\CC)$ contains at most one fixed point of $f^n$.
\end{lemma}

\begin{proof}
Fix an arbitrary $n\in\N$. We denote $F=f^n$ and consider the cell decompositions induced by $F$ and $\CC$ in this proof. Note that $F$ is also an expanding Thurston map and there is no $1$-tile in $\DD^1(F,\CC)$ joining opposite sides of $\CC$.

It suffices to prove that each $1$-tile $X^1\in\X^1$ contains at most one fixed point of $F$.

Suppose that there are two distinct fixed points $p,q$ of $F$ in a $1$-tile $X^1$. We prove that there is a contradiction in each of the following cases.

\smallskip

Case 1: one of the fixed points, say $p$, is in $\inte X^1 $.  Then $X^1 \in \X^1_{ww}\cup \X^1_{bb}$ by Lemma \ref{lmAtLeast1}. Since $p$ is contained in the interior of $X_1\cap F(X_1)$, we get that $X_1\subset F(X_1)$. Since $F|_{X^1}$ is a homeomorphism from $X^1$ to $F(X^1)$ (see Proposition~\ref{propCellDecomp}(i)), we define a $2$-tile $X^2= (F|_{X^1})^{-1}(X^1) \subseteq X^1$. Then we get that $p\in \inte X^2$ and $F(X^2)=X^1$. On the other hand, the point $q$ must be in $X^2$ as well for otherwise there exists $q'\neq q$ such that $q'\in X^2$ and $F(q')=q$, thus $q'$ and $q$ are two distinct points in $X^1$ whose images under $F$ are $q$, contradicting the fact that $F|_{X^1}$ is a homeomorphism from $X^1$ to $F(X^1)$  and $X^1\subseteq F(X^1)$. Thus we can inductively construct an $(n+1)$-cell $X^{n+1}\subseteq X^n$ such that $F(X^{n+1})=X^n$, $p \in \inte(X^{n+1})$, and $q\in X^{n+1}$, for each $n\in\N$. This contradicts the fact that $F$ is an expanding Thurston map, see Remark~\ref{rmExpanding}.

\smallskip

Case 2: there exists a $1$-edge $e\in\E^1$ such that $p,q\in e$. Note that $e\subseteq X^1$. Then one of the fixed points $p$ and $q$, say $p$, must be contained in the interior of $e$, for otherwise $p$, $q$ are distinct $1$-vertices that are fixed by $F$, thus they are both 0-vertices, hence $X^1$ joins opposite sides, a contradiction. Since $F(e)$ is a $0$-edge by Proposition~\ref{propCellDecomp}, and $p\in F(e)$, there exists a $1$-edge $e'\subseteq F(e)$ with $p\in e'$. Thus $e'$ intersects with $e$ at the point $p$, which is an interior point of $e$. So $e'=e$, and $e\subseteq F(e)$. Then by the same argument as when $p \in \inte X^1 $ in Case 1, we can get a contradiction to the fact that $F$ is an expanding Thurston map.

\smallskip

Case 3: the points $p$, $q$ are contained in two distinct $1$-edges $e_1,e_2$ of $X^1$, respectively, and $e_1\cap e_2 \neq \emptyset$. Since $F$ is an expanding Thurston map, we have $m=\card(\post F)\geq 3$ (see \cite[Corollary 6.4]{BM10}). So $X^1$ is an $m$-gon (see Proposition \ref{propCellDecomp}(vi)). Since $e_1\cap e_2 \neq \emptyset$, we get $\card(e_1\cap e_2)=1$, say $e_1\cap e_2 =\{v\}$. By Case 2, we get that $v\neq p$ and $v\neq q$. Note that $p\in F(e_1)$, $q\in F(e_2)$, and $F(e_1),F(e_2)$ are $0$-edges. If at least one of $p$ and $q$ is a $1$-vertex, thus a 0-vertex as well, then since Proposition~\ref{propCellDecomp}(i) implies that  $F(e_1)\neq F(e_2)$, we can conclude that $X^1$ touches at least three $0$-edges, thus joins opposite sides of $\CC$, a contradiction. Hence $p\in\inte e_1$ and $q\in\inte e_2$. So $e_1\subseteq F(e_1)$, $e_2\subseteq F(e_2)$, and
$$
\{v\} = e_1 \cap e_2 \subseteq F(e_1)\cap F(e_2) = F(e_1\cap e_2) = F(\{v\}),
$$
by Proposition~\ref{propCellDecomp}(i). Thus $F(v)=v$. Then $e_1$ contains two distinct fixed points $p$ and $v$ of $F$, which was already proven to be impossible in Case 2.

\smallskip

Case 4: the points $p$, $q$ are contained in two distinct $1$-edges $e_1,e_2$ of $X^1$, respectively, and $e_1\cap e_2 = \emptyset$. Thus $F(e_1)$ and $F(e_2)$ are a pair of disjoint edges of $F(X^1)$. But $p=F(p)\in F(e_1)$, $q=F(q)\in F(e_2)$, so $X^1$ joins opposite sides of $\CC$, a contradiction.

\smallskip

Combining all cases above, we can conclude, therefore, that each $1$-tile $X^1\in\X^1$ contains at most one fixed point of $F$.
\end{proof}

We can immediately get an upper bound of the number of periodic points of an expanding Thurston map from Lemma~\ref{lmAtMost1}.

\begin{cor}   \label{corNoFixedPtsUpperBound}
Let $f$ be an expanding Thurston map. Then for each $n\in\N$ sufficiently large, the number of fixed points of $f^n$ is $\leq 2(\deg f)^n$. In particular, the number of fixed points of $f$ is finite.
\end{cor}

\begin{proof}
By Corollary~\ref{corCexists}, for each $n\geq N(f)$, where $N(f)\in\N$ is a constant as given in Corollary~\ref{corCexists}, there exists an $f^n$-invariant Jordan curve $\CC$ containing $\post f$ such that no $n$-tile in $\DD^n(f,\CC)$ joins opposite sides of $\CC$. Let $F=f^n$. So $F$ is an expanding Thurston map, and $\CC$ is an $F$-invariant Jordan curve containing $\post F$ such that no $1$-tile in $\DD^1(F,\CC)$ joins opposite sides of $\CC$. By Proposition~\ref{propCellDecomp}(iv), the number of $1$-tiles in $\X^1(F,\CC)$ is exactly $2\deg F= 2 (\deg f)^n$. By Lemma~\ref{lmAtMost1}, we can conclude that there are at most $2(\deg f)^n$ fixed points of $F=f^n$.

Since each fixed point of $f$ is also a fixed point of $f^n$, for each $n\in\N$, the number of fixed points of $f$ is finite.
\end{proof}

\begin{lemma}  \label{lmDeg_f_C}
Let $f$ be an expanding Thurston map with an $f$-invariant Jordan curve $\CC$ containing $\post f$. Then
\begin{align}
 \deg(f|_{\CC}) &= \card (\X_{ww}^1(f,\CC))  - \card (\X_{bw}^1(f,\CC)) \label{eqDeg_f_C} \\   
                &= \card (\X_{bb}^1(f,\CC))  - \card (\X_{wb}^1(f,\CC)).    \notag
\end{align}
\end{lemma}

Here $\deg(f|_\CC)$ is the \emph{degree} of the map $f|_\CC\: \CC\rightarrow \CC$ (see for example, \cite[Section~2.2]{Ha02}).

Note that the first equality in (\ref{eqDeg_f_C}), for example, says that the degree of $f$ restricted to $\CC$ is equal to the number of white $1$-tiles contained in the white $0$-tile minus the number of black $1$-tiles contained in the white $0$-tile.

Recall that for each continuous path $\gamma\: [a,b]\rightarrow \C\setminus \{0\}$ on the Riemann sphere $\widehat\C$, with $a,b\in\R$ and $a<b$, we can define the \emph{variation of the argument along $\gamma$}, denoted by $V(\gamma)$, as the change of the imaginary part of the logarithm along $\gamma$. Note that $V(\gamma)$ is invariant under an orientation-preserving reparametrization of $\gamma$ and if $\widetilde\gamma\: [a,b]\rightarrow \widehat\C$ reverses the orientation of $\gamma$, i.e., $\widetilde\gamma(t)=\gamma(b-t)$, then $V(\widetilde\gamma)=-V(\gamma)$. We also note that if $\gamma$ is a loop, then $V(\gamma)= 2\pi \Ind_\gamma(0)$, where $\Ind_\gamma(0)$ is the \emph{winding number of $\gamma$ with respect to $0$} \cite[Chapter~IV]{Bu79}. 

\begin{proof}
Consider the cell decompositions induced by $(f,\CC)$. Let $X_w^0$ be the white $0$-tile.

We start with proving the first equality in (\ref{eqDeg_f_C}).

By the Schoenflies theorem (see, for example, \cite[Theorem~10.4]{Mo77}), we can assume that $S^2$ is the Riemann sphere $\widehat{\C}$, and $X_w^0$ is the unit disk with the center $0$ disjoint from $f^{-1}(\CC)$.

For each $1$-edge $e\in\E^1$, we choose a parametrization $\gamma^+_e\: [0,1]\rightarrow \C\setminus\{0\}$ of $e$ with positive orientation (i.e., with the white $1$-tile on the left), and a parametrization $\gamma^-_e\: [0,1]\rightarrow \C\setminus\{0\}$ of $e$ with negative orientation. Then $f\circ\gamma^+_e$ and $f\circ\gamma^-_e$ are parametrizations of one of the $0$-edges on the unit circle $\CC$, with positive orientation and negative orientation, respectively.

We claim that
\begin{align}   \label{eqSumVarArg}
     & \sum\limits_{X\in\X^1_{ww}} \sum\limits_{e\in\E^1,e\subseteq\partial X} V(f\circ\gamma^+_e) - \sum\limits_{X\in\X^1_{bw}} \sum\limits_{e\in\E^1,e\subseteq\partial X} V(f\circ\gamma^+_e) \\
  =  & \sum\limits_{e\in\E^1, e\subseteq\CC} V(f\circ\gamma_e),   \notag
\end{align}
where on the right-hand side, $\gamma_e=\gamma^+_e$ if $e\subseteq\CC\cap X$ for some $X\in\X^1_{ww}$ and $\gamma_e=\gamma^-_e$ if $e\subseteq\CC\cap X$ for some $X\in\X^1_{bw}$, or equivalently, $\gamma_e$ parametrizes $e$ in such a way that $X_w^0$ is always on the left of $e$ for each $e\in \E^1$ with $e\subseteq\CC$.

We observe that the left-hand side of (\ref{eqSumVarArg}) is the sum of $V(f\circ\gamma^+_e)$ over all $1$-edges $e$ in the boundary of a white $1$-tile $X\subseteq X_w^0$ plus the sum of $V(f\circ\gamma^-_e)$ over all $1$-edges $e$ in the boundary of a black $1$-tile $X\subseteq X_w^0$. Since each $1$-edge $e$ with $\inte e\subseteq X_w^0$ is the intersection of exactly one $1$-tile in $\X^1_{ww}$ and one $1$-tile in $\X^1_{bw}$, the two terms corresponding to a $1$-edge $e$ that is not contained in $\CC$ cancel each other. Moreover, there is exactly one term for each $1$-edge $e\subseteq X^0_w$ that is contained in $\CC$, and $e$ that corresponds to such a term is parametrized in such a way that $X_w^0$ is on the left of $e$. The claim now follows.

We then note that by Proposition~\ref{propCellDecomp}(i) and the definition of branched covering maps on $S^2$ in the beginning of Section~\ref{sctThurstonMap}, the map $f$ is an orientation-preserving local homeomorphism. Thus the left-hand side of (\ref{eqSumVarArg}) is equal to
\begin{equation*}
\sum\limits_{X\in\X^1_{ww}} 2 \pi - \sum\limits_{X\in\X^1_{bw}} 2 \pi  = 2\pi  \(\card (\X_{ww}^1)  - \card (\X_{bw}^1) \),
\end{equation*}
and the right-hand side of (\ref{eqSumVarArg}) is equal to
\begin{equation*}
2\pi \Ind_{f\circ\gamma_\CC}(0) = 2\pi \deg (f|_\CC),
\end{equation*}
where $\gamma_\CC$ is a parametrization of $\CC$ with positive orientation. Hence the first equality in (\ref{eqDeg_f_C}) follows.

The second equality in (\ref{eqDeg_f_C}) follows from the fact that 
$$
\card (\X_{ww}^1)  + \card (\X_{wb}^1)  = \deg f = \card (\X_{bb}^1) + \card (\X_{bw}^1). 
$$
\end{proof}

Let $f$ be an expanding Thurston map with an $f$-invariant Jordan curve $\CC$ containing $\post f$. We orient $\CC$ in such a way that the white $0$-tile lies on the left of $\CC$. Let $p\in \CC$ be a fixed point of $f$. We say that $f|_\CC$ \defn{preserves the orientation at $p$} (resp.\ \defn{reverses the orientation at $p$}) if there exists an open arc $l\subseteq\CC$ with $p\in l$ such that $f$ maps $l$ homeomorphically to $f(l)$ and $f|_\CC$ preserves (resp.\ reverses) the orientation on $l$. More concretely, when $p$ is a $1$-vertex, let $l_1,l_2\subseteq \CC$ be the two distinct $1$-edges on $\CC$ containing $p$; when $p\in\inte e$ for some $1$-edge $e\subseteq\CC$, let $l_1,l_2$ be the two connected components of $e \setminus \{p\}$. Then $f|_\CC$ preserves the orientation at $p$ if $l_1\subseteq f(l_1)$ and $l_2\subseteq f(l_2)$, and reverses the orientation at $p$ if $l_2\subseteq f(l_1)$ and $l_1\subseteq f(l_2)$. Note that it may happen that $f|_\CC$ neither preserves nor reverses the orientation at $p$, because $f|_\CC$ need not be a local homeomorphism near $p$, where it may behave like a ``folding map''. 

\begin{lemma}  \label{lmNoFixedPts_f_C}
Let $f$ be an expanding Thurston map with an $f$-invariant Jordan curve $\CC$ containing $\post f$. Then the number of fixed points of $f|_\CC$ where $f|_\CC$ preserves the orientation minus the number of fixed points of $f|_\CC$ where $f|_\CC$ reverses the orientation is equal to $\deg (f|_\CC) -1$.
\end{lemma}

\begin{proof}
Let $\psi\:[0,1]\rightarrow\CC$ be a continuous map such that $\psi|_{(0,1)}\:(0,1)\rightarrow \CC\setminus\{x_0\}$ is an orientation-preserving homeomorphism, and $\psi(0)=\psi(1)=x_0$ for some $x_0\in\CC$ that is not a fixed point of $f|_\CC$. Note that for each $x\in\CC$ with $x\neq x_0$, $\psi^{-1}(x)$ is a well-defined number in $(0,1)$. In particular, $\psi^{-1}(y)$ is a well-defined number in $(0,1)$ for each fixed point $y$ of $f|_\CC$. Define $\pi\:\R\rightarrow\CC$ by $\pi(x)=\psi(x-\lfloor x \rfloor)$. Then $\pi$ is a covering map. We lift $f|_\CC\circ \psi$ to $G\:[0,1]\rightarrow\R$ such that $\pi\circ G = f|_\CC\circ\psi$ and $G(0)=\psi^{-1}(f(x_0))\in(0,1)$. So we get the following commutative diagram:
\begin{equation*}
\xymatrix{ & \R \ar[d]^\pi \\ 
[0,1] \ar[r]_{f|_\CC \circ \psi} \ar[ur]^G & \CC.}
\end{equation*}
Then $G(1)-G(0)\in\Z$ and
\begin{equation}
\deg(f|_\CC)=G(1)-G(0).
\end{equation}

Observe that $y\in\CC$ is a fixed point of $f|_\CC$ if and only if $G(\psi^{-1}(y))-\psi^{-1}(y)\in\Z$. Indeed, if $y\in\CC$ is a fixed point of $f|_\CC$, then $\pi\circ G \circ \psi^{-1}(y)=f|_\CC(y)=y$. Thus $G \circ \psi^{-1}(y)- \psi^{-1}(y) \in \Z$. Conversely, if $G \circ \psi^{-1}(y)- \psi^{-1}(y) \in \Z$, then $y \neq x_0$, thus
$$
f|_\CC(y)=f|_\CC \circ \psi \circ \psi^{-1} (y) = \pi \circ G \circ \psi^{-1} (y) = \pi \circ \psi^{-1} (y) = y.
$$

For each $m\in\Z$, we define the line $l_m$ to be the graph of the function $x\mapsto x+m$ from $\R$ to $\R$.

Let $y\in\CC$ be any fixed point of $f|_\CC$. Since by Corollary~\ref{corNoFixedPtsUpperBound} fixed points of $f$ are isolated, there exists a neighborhood $(s,t)\subseteq (0,1)$ such that $\psi^{-1}(y)\in(s,t)$ and for each fixed point $z\in\CC \setminus\{y\}$ of $f|_\CC$, $\psi^{-1}(z)\notin(s,t)$. Define $k=G(\psi^{-1}(y))-\psi^{-1}(y)$; then $k\in\Z$. Moreover, $z\in\CC$ is a fixed point of $f|_\CC$ if and only if the graph of $G$ intersects with $l_m$ at the point $(\psi^{-1}(z), G\(\psi^{-1}(z)\)$ for some $m\in\Z$.

\begin{figure}
    \centering
    \begin{overpic}
    [width=10cm, 
    tics=20]{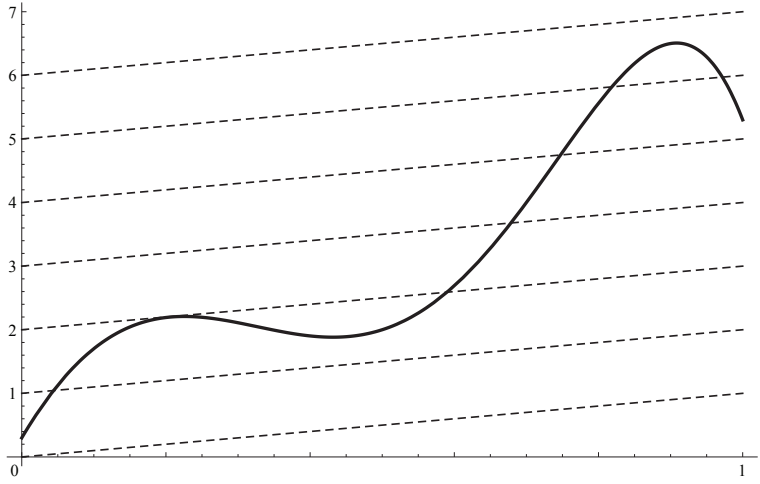}
   \end{overpic}
    \caption{The lines $l_k$ for $k\in\Z$ and an example of the graph of $G$.}
    \label{figPlotG}
\end{figure}

Depending on the orientation of $f|_\CC$ at the fixed point $y\in\CC$, we get one of the following cases:
\begin{enumerate}

\smallskip
\item If $f|_\CC$ preserves the orientation at $y$, then the graph of $G|_{(s,\psi^{-1}(y))}$ lies strictly between the lines $l_{k-1}$ and $l_k$, and the graph of $G|_{(\psi^{-1}(y),t)}$ lies strictly between the lines $l_k$ and $l_{k+1}$.

\smallskip
\item If $f|_\CC$ reverses the orientation at $y$, then the graph of $G|_{(s,\psi^{-1}(y))}$ lies strictly between the lines $l_k$ and $l_{k+1}$, and the graph of $G|_{(\psi^{-1}(y),t)}$ lies strictly between the lines $l_{k-1}$ and $l_k$.

\smallskip
\item If $f|_\CC$ neither preserves nor reverses the orientation at $y$, then the graph of $G|_{(s,t)\setminus\{\psi^{-1}(y)\}}$ either lies strictly between the lines $l_{k-1}$ and $l_k$ or lies strictly between the lines $l_k$ and $l_{k+1}$.

\end{enumerate}
Thus the number of fixed points of $f|_\CC$ where $f|_\CC$ preserves the orientation is exactly the number of intersections between the graph of $G$ and the lines $l_m$ with $m\in\Z$, where the graph of $G$ crosses the lines from below, and the number of fixed points of $f|_\CC$ where $f|_\CC$ reserves the orientation is exactly the number of intersections between the graph of $G$ and the lines $l_m$ with $m\in\Z$, where the graph of $G$ crosses the lines from above. Therefore the number of fixed points of $f|_\CC$ where $f|_\CC$ preserves the orientation minus the number of fixed points of $f|_\CC$ where $f|_\CC$ reverses the orientation is equal to $G(1)-G(0)-1=\deg (f|_\CC) -1$.
\end{proof}

For each $n\in\N$ and each expanding Thurston map $f\:S^2\rightarrow S^2$, we denote by
\begin{equation}
P_{n,f}=\{x\in S^2\,|\, f^n(x)=x, f^k(x)\neq x,k\in\{1,2,\dots,n-1\}\}
\end{equation}
the \defn{set of periodic points of $f$ with period $n$}, and by
\begin{equation}
p_{n,f}=\sum\limits_{x\in{P_{n,f}}} \deg_{f^n}(x), \qquad  \widetilde{p}_{n,f}= \card P_{n,f}
\end{equation}
the numbers of periodic points $x$ of $f$ with period $n$, counted with and without weight $\deg_{f^n}(x)$, respectively, at each $x$. In particular, $P_{1,f}$ is the set of fixed points of $f$ and $p_{1,f}=1+\deg f$ as we will see in the proof of Theorem~\ref{thmNoFixedPts} below. More generally, for all $m\in\N_0$ and $n\in \N$ with $m < n$, we denote by
\begin{equation}  \label{eqSetPrePeriodicPts}
S_{n}^m = \{ x\in S^2 \,|\,  f^m(x)=f^n(x)  \}
\end{equation}
the \defn{set of preperiodic points of $f$ with parameters $m,n$} and by
\begin{equation}   \label{eqNoPrePeriodicPts}
s_{n}^m = \sum\limits_{x\in S_{n}^m}  \deg_{f^n} (x),  \qquad   \widetilde{s}_{n}^m = \card S_{n}^m
\end{equation}
the numbers of preperiodic points of $f$ with parameters $m,n$, counted with and without weight $\deg_{f^n}(x)$, respectively, at each $x$. Note that in particular, for each $n\in\N$, $S_{n}^0 = P_{1,f^n}$ is the set of fixed points of $f^n$.

\begin{proof}[Proof of Theorem~\ref{thmNoFixedPts}]

The idea of the proof is to first prove the theorem for $F=f^n$ for sufficiently large $n$ so that we can assume the existence of some $F$-invariant Jordan curve containing $\post F$. This enables us to make use of the combinatorial information from the cell decompositions induced by $(F,\CC)$. Then we can generalize the conclusion to arbitrary expanding Thurston maps by an elementary number-theoretic argument.

We first prove the theorem for $F=f^n$ for $n \geq N(f)$ where $N(f)$ is a constant as given in Corollary~\ref{corCexists} depending only on $f$. Let $\CC$ be an $f^n$-invariant Jordan curve containing $\post f$ such that no $n$-tile in $\DD^n(f,\CC)$ joins opposite sides of $\CC$ as given in Corollary~\ref{corCexists}. So $\CC$ is an $F$-invariant Jordan curve containing $\post F$ such that no $1$-tile in $\DD^1(F,\CC)$ joins opposite sides of $\CC$. 

Unless otherwise stated, we consider the cell decompositions induced by $(F,\CC)$ in this proof. Let $w_w=\card \X^1_{ww} $ be the number of white $1$-tiles contained in the white $0$-tile, $b_w=\card \X^1_{bw} $ be the number of black $1$-tiles contained in the white $0$-tile, $w_b=\card \X^1_{wb} $ be the number of white $1$-tiles contained in the black $0$-tile, and $b_b=\card \X^1_{bb} $ be the number of black $1$-tiles contained in the black $0$-tile. Note that $w_w+w_b=b_w+b_b=\deg F$.

By Corollary~\ref{corNoFixedPtsUpperBound}, we know that fixed points of $F$ are isolated.

\smallskip

Note that 
\begin{equation}    \label{eqPfThmNoFixedPts}
w_w+b_b=\deg F +\deg (F|_\CC),
\end{equation}
which follows from the equation $w_w-b_w=\deg(F|_\CC)$ by Lemma~\ref{lmDeg_f_C}, and the equation $b_w + b_b=\deg F$.

\smallskip

We define sets
$$A=\{X \in \X^1_{ww}\,|\,\text{there exists $p\in\CC\cap X$ with } F(p)=p\},$$
$$B=\{X \in \X^1_{bw}\,|\,\text{there exists $p\in\CC\cap X$ with } F(p)=p\},$$
and let $a=\card A$, $b=\card B$.

\smallskip

We then claim that
\begin{equation}
a-b=\deg(F|_\CC)-1.
\end{equation}

In order to prove this claim, we will first prove that $a-b$ is equal to the number of fixed points of $F|_\CC$ where $F|_\CC$ preserves the orientation minus the number of fixed points of $F|_\CC$ where $F|_\CC$ reverses the orientation.

So let $p\in\CC$ be a fixed point of $F|_\CC$.

\begin{figure}
    \centering
    \begin{overpic}
    [width=6cm, 
    tics=20]{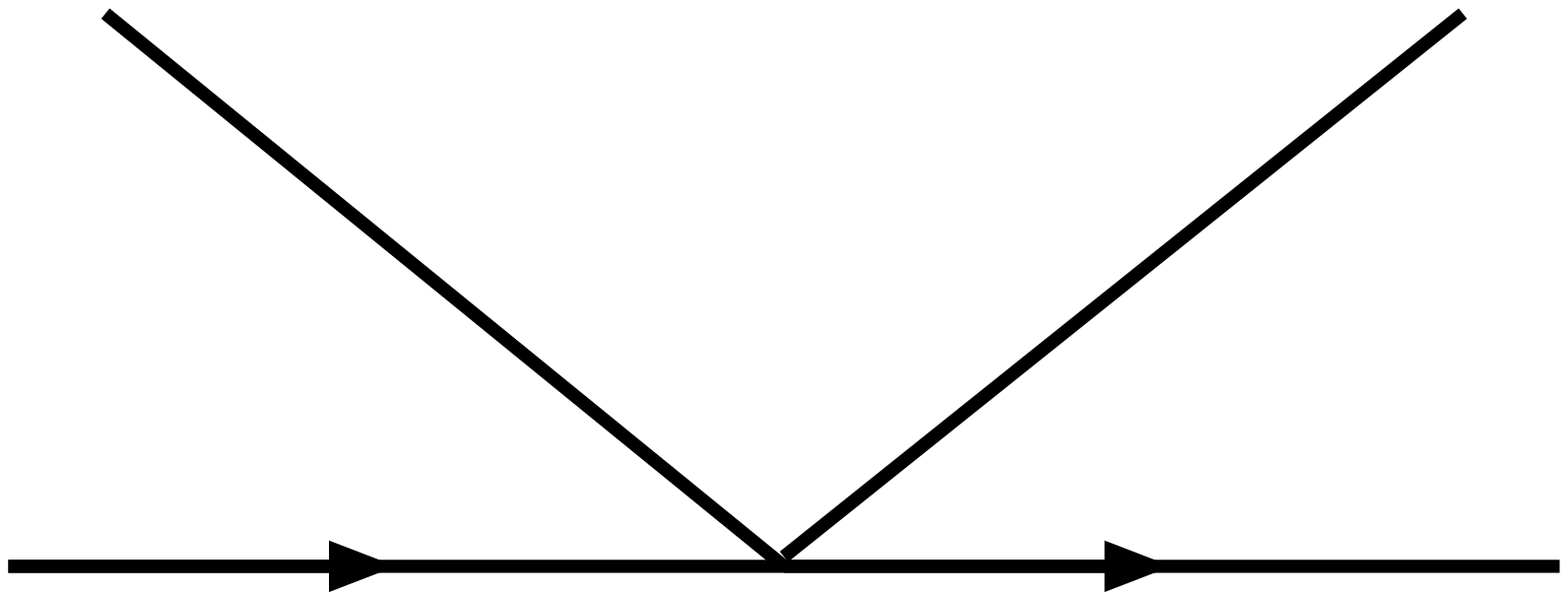}
    \put(75,40){$bw$}
    \put(25,14){$ww$}
    \put(120,14){$ww$}
    \put(25,-8){$e_1$}
    \put(120,-8){$e_2$}
    \put(80,-8){$p$}
    \end{overpic}
    \caption{Case (2)(a) where $F(e_1)\supseteq e_1$ and $F(e_2)\supseteq e_2$.}
    \label{figPlota}

    \centering
    \begin{overpic}
    [width=6cm, 
    tics=20]{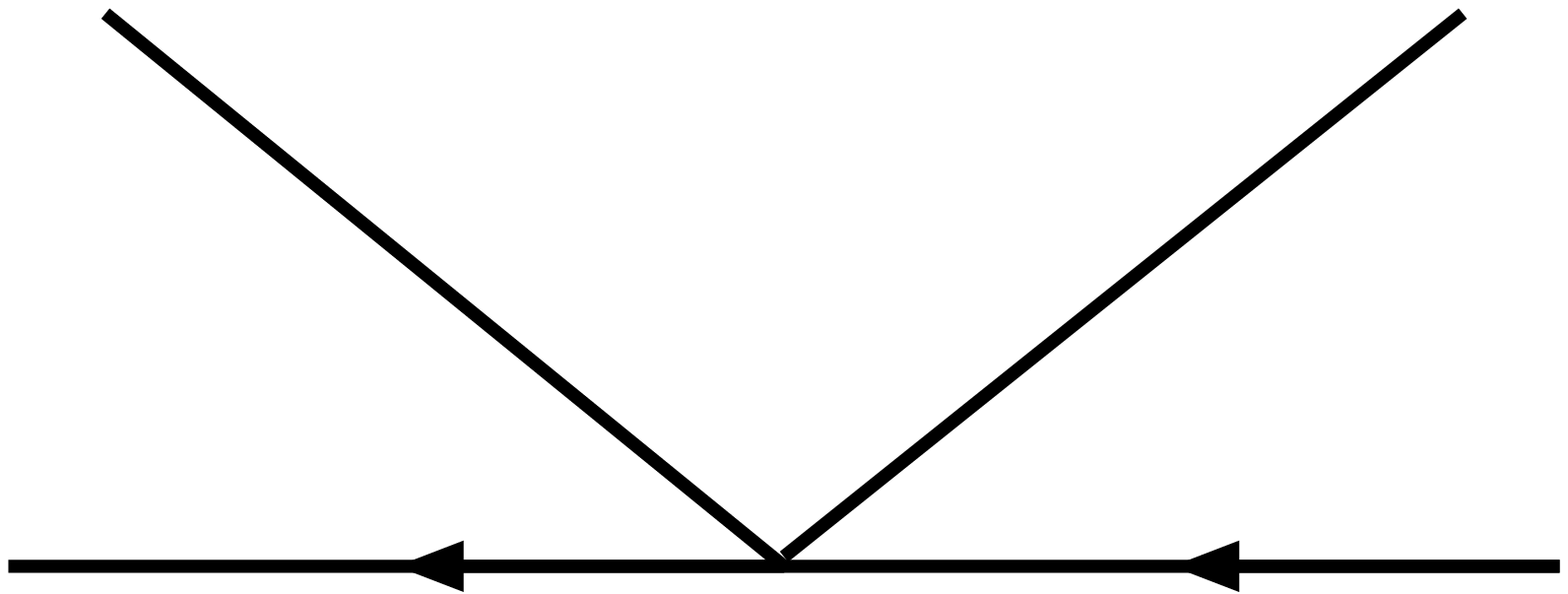}
    \put(75,40){$ww$}
    \put(25,14){$bw$}
    \put(120,14){$bw$}
    \put(25,-8){$e_1$}
    \put(120,-8){$e_2$}
    \put(80,-8){$p$}
    \end{overpic}
    \caption{Case (2)(b) where $F(e_1)\supseteq e_2$ and $F(e_2)\supseteq e_1$.}
    \label{figPlotb}

    \centering
    \begin{overpic}
    [width=6cm, 
    tics=20]{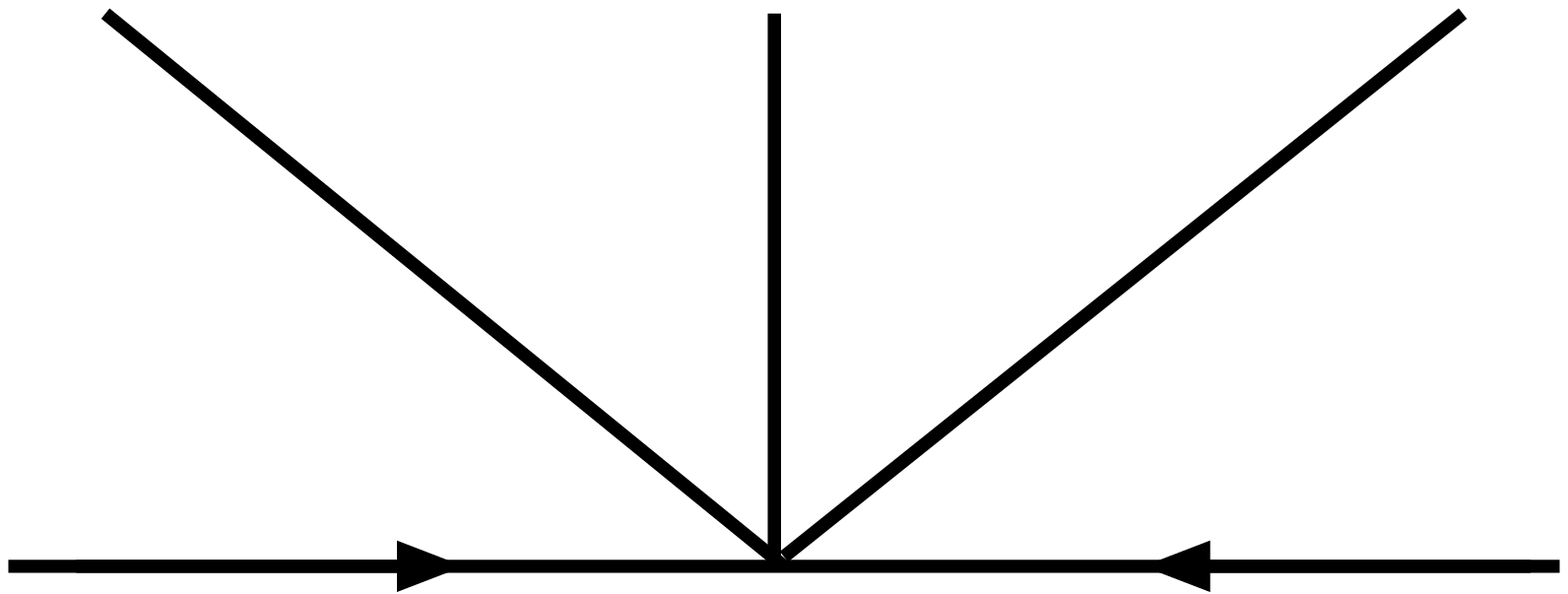}
    \put(55,40){$bw$}
    \put(95,40){$ww$}
    \put(25,14){$ww$}
    \put(120,14){$bw$}
    \put(25,-8){$e_1$}
    \put(120,-8){$e_2$}
    \put(80,-8){$p$}
    \end{overpic}
    \caption{Case (2)(c) where $F(e_1)=F(e_2)\supseteq e_1$.}
    \label{figPlotc}

    \centering
    \begin{overpic}
    [width=6cm, 
    tics=20]{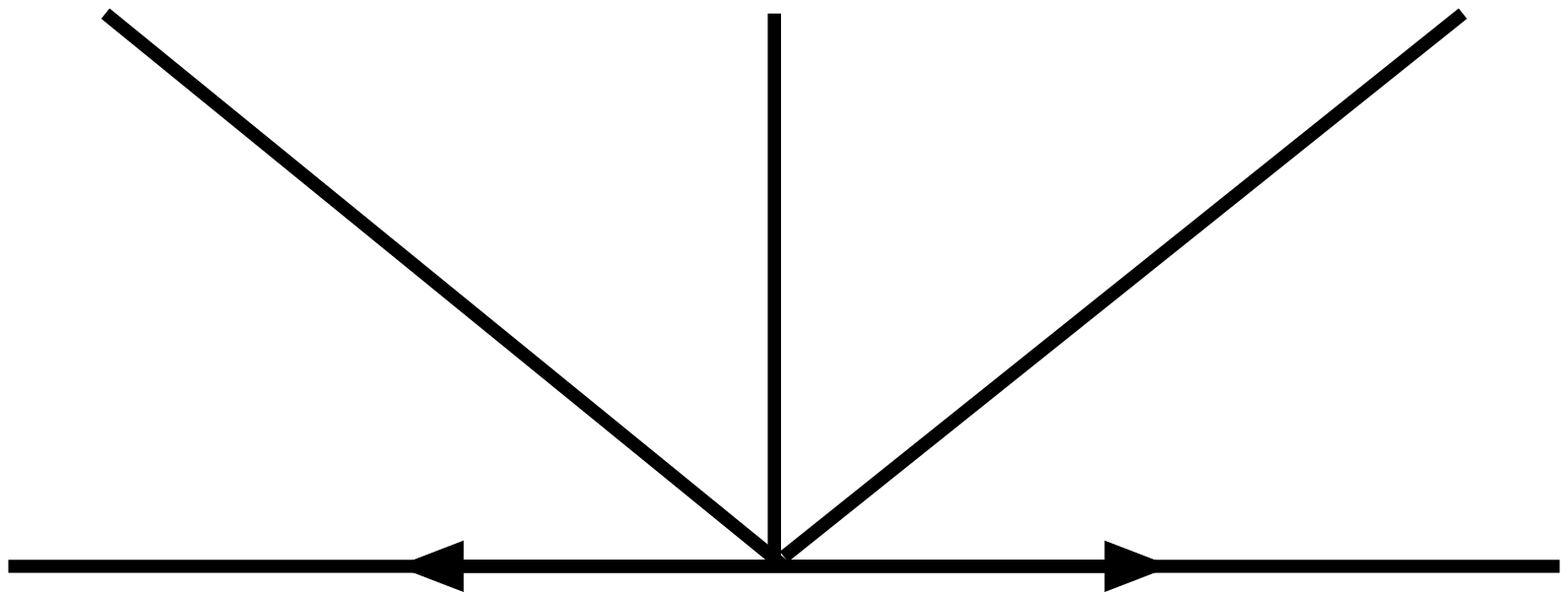}
    \put(55,40){$ww$}
    \put(95,40){$bw$}
    \put(25,14){$bw$}
    \put(120,14){$ww$}
    \put(25,-8){$e_1$}
    \put(120,-8){$e_2$}
    \put(80,-8){$p$}
    \end{overpic}
    \caption{Case (2)(d) where $F(e_1)=F(e_2)\supseteq e_2$.}
    \label{figPlotd}
\end{figure}

\begin{enumerate}

\smallskip
\item If $p$ is not a critical point of $F$, then either $F|_\CC$ preserves or reverses the orientation at $p$. In this case, the point $p$ is contained in exactly one white $1$-tile and one black $1$-tile.
\begin{enumerate}

\smallskip
\item If $F|_\CC$ preserves the orientation at $p$, then $p$ is contained in exactly one white $1$-tile that is contained in the white $0$-tile, and $p$ is not contained in any black $1$-tile that is contained in the while $0$-tile.

\smallskip
\item If $F|_\CC$ reverses the orientation at $p$, then $p$ is contained in exactly one black $1$-tile that is contained in the white $0$-tile, and $p$ is not contained in any white $1$-tile that is contained in the while $0$-tile.
\end{enumerate}

\smallskip
\item If $p$ is a critical point of $F$, then $p=F(p)\in\post f$ and so there are two distinct $1$-edges $e_1,e_2\subseteq\CC$ such that $\{p\}=e_1\cap e_2$. We refer to Figures \ref{figPlota} to \ref{figPlotd}.
\begin{enumerate}

\smallskip
\item If $e_1\subseteq F(e_1)$ and $e_2\subseteq F(e_2)$, then $p$ is contained in exactly $k$ white and $k-1$ black $1$-tiles that are contained in the white $0$-tile, for some $k\in\N$. Note that in this case $F|_\CC$ preserves the orientation at $p$.

\smallskip
\item If $e_2\subseteq F(e_1)$ and $e_1\subseteq F(e_2)$, then $p$ is contained in exactly $k-1$ white and $k$ black $1$-tiles that are contained in the white $0$-tile, for some $k\in\N$. Note that in this case $F|_\CC$ reverses the orientation at $p$.

\smallskip
\item If $e_1\subseteq F(e_1)= F(e_2)$, then $p$ is contained in exactly $k$ white and $k$ black $1$-tiles that are contained in the white $0$-tile, for some $k\in\N$. Note that in this case $F|_\CC$ neither preserves nor reverses the orientation at $p$.

\smallskip
\item If $e_2\subseteq F(e_1)= F(e_2)$, then $p$ is contained in exactly $k$ white and $k$ black $1$-tiles that are contained in the white $0$-tile, for some $k\in\N$. Note that in this case $F|_\CC$ neither preserves nor reverses the orientation at $p$.
\end{enumerate}
\end{enumerate}
It follows then that $a-b$ is equal to the number of fixed points of $F|_\CC$ where $F|_\CC$ preserves the orientation minus the number of fixed points of $F|_\CC$ where $F|_\CC$ reverses the orientation.

Then the claim follows from Lemma \ref{lmNoFixedPts_f_C}.

\smallskip

Next, we are going to prove that the number of fixed points of $F$, counted with weight given by the local degree, is equal to
\begin{equation} \label{eqDegF}
w_w+b_b-a+b,
\end{equation}
which, by (\ref{eqPfThmNoFixedPts}) and the claim above, is equal to 
\begin{equation*}
\deg F +\deg(F|_\CC)-(\deg(F|_\CC)-1)=1+\deg F.
\end{equation*}

Indeed, by Lemma~\ref{lmAtLeast1} and Lemma~\ref{lmAtMost1}, each $1$-tile that contributes in (\ref{eqDegF}), i.e., each $1$-tile in $\X^1_{ww}\cup\X^1_{bb}\cup B \cup A$, contains exactly one fixed point (not counted with weight) of $F$. On the other hand, each fixed point is contained in at least one of the $1$-tiles in $\X^1_{ww}\cup\X^1_{bb}\cup B \cup A$. Let $p$ be a fixed point of $F$, then one of the following happens:
\begin{enumerate}

\smallskip
\item If $p\notin\CC$, then $p$ is not contained in any $1$-edges $e$ since $F(e)\subseteq\CC$. So $p\in\inte X$ for some $X \in \X^1_{ww}\cup \X^1_{bb}\setminus \(A\cup B\)$, by Lemma~\ref{lmAtLeast1}. So each such $p$ contributes $1$ to (\ref{eqDegF}).

\smallskip
\item If $p\in\CC$ but $p\notin\crit F$, then $p$ is not a $1$-vertex, so either $p$ is contained in exactly two $1$-tiles $X\in\X^1_{ww}$ and $X'\in\X^1_{bb}$, or either $p$ is contained in exactly two $1$-tiles $X\in\X^1_{bw}$ and $X'\in\X^1_{wb}$. In either case, $p$ contributes $1$ to (\ref{eqDegF}).

\smallskip
\item If $p\in\CC$ and $p\in\crit F$, then $p$ is a $0$-vertex, so the part that $p$ contributes in (\ref{eqDegF}) counts the number of black $1$-tiles that contains $p$, which is exactly the weight $\deg_F(p)$ of $p$.
\end{enumerate}

Hence we have proved the theorem with $f$ replaced by $F=f^n$, for each $n \geq N(f)$ where $N(f)$ is the constant as given in Corollary~\ref{corCexists} depending only on $f$. We are now going to remove this restriction by an elementary number-theoretic argument.

Choose a prime $r\geq N(f)$. Note that the set of fixed points of $f^r$ can be decomposed into orbits under $f$ of length $r$ or 1, since $r$ is a prime. Let $p$ be a fixed point of $f^r$. By using the following formula derived from (\ref{eqLocalDegreeProduct}),
\begin{equation}
\deg_{f^r}(p) = \deg_f(p)\deg_f(f(p))\deg_f(f^2(p))\cdots \deg_f(f^{r-1}(p)),
\end{equation}
we can conclude that
\begin{enumerate}

\smallskip
\item[(1)] if $p\notin\crit(f^r)$, or equivalently, $\deg_{f^r}(p)=1$, and
\begin{enumerate}

\smallskip
\item[(i)] if $p$ is in an orbit of length $r$, then $p,f(p),\dots,f^{r-1}(p)\notin \crit f$, or equivalently, the local degrees of $f^r$ at these points are all 1;

\smallskip
\item[(ii)] if $p$ is in an orbit of length 1, then $p\notin\crit f$, or equivalently, $\deg_f(p)=1$;
\end{enumerate}

\smallskip
\item[(2)] if $p\in\crit(f^r)$, and
\begin{enumerate}

\smallskip
\item[(i)] if $p$ is in an orbit of length $r$, then all $p,f(p),\dots,f^{r-1}(p)$ are fixed points of $f^r$ with the same weight $\deg_{f^r}(p)=\deg_{f^r}(f^k(p))$ for each $k\in\N$;

\smallskip
\item[(ii)] if $p$ is in an orbit of length 1, then $p\in\crit f$ and the weight of $f^r$ at $p$ is $\deg_{f^r}(p)=(\deg_f (p))^r$.
\end{enumerate}
\end{enumerate}
Note that a fixed point $p\in S^2$ of $f^r$ is a fixed point of $f$ if and only if $p$ is in an orbit of length 1 under $f$. So by first summing the weight of the fixed points of $f^r$ in the same orbit then summing over all orbits and applying Fermat's Little Theorem, we can conclude that
\begin{align*}
p_{1,f^r} & = \sum\limits_{x\in P_{1,f^r}} \deg_{f^r}(x)  \\
          & = \sum\limits_{\text{(1)(i)}} r  + \sum\limits_{\text{(1)(ii)}} 1 + \sum\limits_{\text{(2)(i)}} r\deg_{f^r}(p) + \sum\limits_{\text{(2)(ii)}} (\deg_f(p))^r \\
          & \equiv \sum\limits_{\text{(1)(ii)}} 1 + \sum\limits_{\text{(2)(ii)}} \deg_f(p)\\
          & =  p_{1,f}  \pmod{r},
\end{align*}
where on the second line, the first sum ranges over all orbits in Case (1)(i), the second sum ranges over all orbits in Case (1)(ii), the third sum ranges over all orbits $\{p, f(p),\dots,f^{r-1}(p)\}$ in Case (2)(i), the last sum ranges over all orbits $\{p\}$ in Case (2)(ii). Thus by (\ref{eqDegreeProduct}) and Fermat's Little Theorem again, we have
\begin{align*}
0 = & \deg(f^r)+1-p_{1,f^r}  \\
  \equiv & (\deg f)^r+1-p_{1,f} \\
  \equiv & 1+\deg f-p_{1,f}  \pmod{r}.
\end{align*}
By choosing the prime $r$ larger than
$$\Abs{1+\deg f-p_{1,f}  },$$
we can conclude that
$$
p_{1,f}  = 1+\deg f.
$$
\end{proof}

In particular, we have the following corollary in which the weight for all points are trivial.

\begin{cor}
If $f$ is an expanding Thurston map with no critical fixed points, then there are exactly $1+\deg f$ distinct fixed points of $f$. Moreover, if $f$ is an expanding Thurston map with no periodic critical points, then there are exactly $1+(\deg f)^n$ distinct fixed points of $f^n$, for each $n\in\N$.
\end{cor}

\begin{proof}
The first statement follows immediately from Theorem \ref{thmNoFixedPts}.

To prove the second statement, we first recall that if $f$ is an expanding Thurston map, so is $f^n$ for each $n\in\N$. Next we note that for each fixed point $p\in S^2$ of $f^n$, $n\in\N$, we have $\deg_{f^n}(p)=1$. For otherwise, suppose $\deg_{f^n}(p)>1$ for some $n\in\N$, then
$$
1< \deg_{f^n}(p) = \deg_f(p)\deg_f(f(p))\deg_f(f^2(p))\cdots \deg_f(f^{n-1}(p)).
$$
Thus at least one of the points $p, f(p), f^2(p),\dots, f^{n-1}(p)$ is a periodic critical point of $f$, a contradiction. The second statement now follows.
\end{proof}

We recall the definition of $s_n^m$ in (\ref{eqSetPrePeriodicPts}) and (\ref{eqNoPrePeriodicPts}).

\begin{cor}  \label{corNoPrePeriodicPts}
Let $f$ be an expanding Thurston map. For each $m\in\N_0$ and $n\in\N$ with $m<n$, we have
\begin{equation}
s_{n}^m = (\deg f)^n + (\deg f)^m.
\end{equation}
\end{cor}

\begin{proof}
For all $m\in\N_0$ and $n\in\N$ with $m<n$, we have
\begin{align*}
s_{n}^m & = \sum\limits_{x\in S_{n}^m}  \deg_{f^n}(x) = \sum\limits_{y=f^{n-m}(y)} \sum\limits_{x\in f^{-m}(y)} \deg_{f^n}(x) \\
          & = \sum\limits_{y=f^{n-m}(y)}  \deg_{f^{n-m}}(y) \sum\limits_{x\in f^{-m}(y)} \deg_{f^m}(x) \\
          & = \( (\deg f)^{n-m} + 1 \) (\deg f)^m.
\end{align*}
The last equality follows from (\ref{eqDeg=SumLocalDegree}), (\ref{eqDegreeProduct}), and Theorem~\ref{thmNoFixedPts}.
\end{proof}

Finally, for expanding Thurston maps with no periodic critical points, we derive a formula for $p_{n,f}$, $n\in\N$, from Theorem~\ref{thmNoFixedPts} and the M\"obius inversion formula (see for example, \cite[Section~2.4]{Bak85}).

\begin{definition} \label{defMobiusFn}
The \defn{M\"obius function}, $\mu(u)$, is defined by
$$
\mu(n) =\begin{cases} 1 & \text{if } n=1; \\ (-1)^r & \text{if } n=p_1 p_2 \dots p_r, \text{ and $p_1,\dots,p_r$ are distinct primes}; \\ 0  & \text{otherwise.} \end{cases}
$$
\end{definition}

\begin{cor}
Let $f$ be an expanding Thurston map without any periodic critical points. Then for each $n\in\N$, we have
$$
p_{n,f}=\sum\limits_{d|n} \mu(d)p_{1,f^{n/d}} = \begin{cases} \sum\limits_{d|n} \mu(d) (\deg f)^{n/d}  & \text{if } n>1; \\ 1 + \deg f  & \text{if } n=1. \end{cases}
$$
\end{cor}

\begin{proof}
The first equality follows from the M\"obius inversion formula and the equation $p_{1,f^n}=\sum\limits_{d|n} p_{d,f}$, for $n\in\N$. The second equality follows from Theorem~\ref{thmNoFixedPts} and the following fact (see for example, \cite[Section~2.4]{Bak85}):
$$
\sum\limits_{d|n} \mu(d) =\begin{cases} 1 & \text{if } n=1, \\ 0  & \text{if } n>1. \end{cases}
$$
\end{proof}

\section{Equidistribution}  \label{sctEquidistribution}

In this section, we derive various equidistribution results as stated in Theorem~\ref{thmWeakConvPreImg}, Theorem~\ref{thmWeakConvPrePerPts}, and Corollary~\ref{corWeakConvPerPts}. We prove these results by first establishing a general statement in Theorem~\ref{thmWeakConv} on the convergence of the distributions of the white $n$-tiles in the tile decompositions discussed in Section~\ref{sctThurstonMap}, in the weak* topology, to the unique measure of maximal entropy of an expanding Thurston map.

Let us now review the concept of measure of maximal entropy for dynamical systems. Then we prove in Theorem~\ref{thmWeakConv} that the distributions of the points, each of which is located ``near'' its corresponding white $n$-tile where the correspondence is a bijection, converges in the weak* topology to $\mu_f$ as $n\longrightarrow+\infty$. Then Theorem~\ref{thmWeakConvPreImg} follows from Theorem~\ref{thmWeakConv}. Theorem~\ref{thmWeakConvPrePerPts} finally follows after we prove a technical bound in Lemma~\ref{lmCoverEdges} generalizing a corresponding lemma from \cite{BM10}. As a special case, we obtain Corollary~\ref{corWeakConvPerPts}.

We start with recalling concepts of entropy for dynamical systems. We follow closely the notation from \cite[Chapter 20]{BM10}.

Let $(X,d)$ be a compact metric space and $g\:X\rightarrow X$ a continuous map. For each $n\in\N$ and $x,y\in X$,
$$
d^n_g(x,y)=\operatorname{max}\{d(g^k(x),g^k(y))\,|\,k=0,\dots,n-1\}
$$
defines a metric on $X$. Let $D(g,\epsilon,n)$ be the minimum number of $\epsilon$-balls in $(X,d^n_g)$ whose union covers $X$.

One can show that the \defn{topological entropy} $h_{\operatorname{top}}(g)$ of $g$, defined as
$$
h_{\operatorname{top}}(g)=\lim\limits_{\epsilon\to 0}\lim\limits_{n\to+\infty} \frac1n \log(D(g,\epsilon,n)),
$$
is well-defined and independent of $d$ as long as the topology on $X$ defined by $d$ remains the same \cite[Proposition 3.1.2]{KH95}.

Let $\mu\in \MMM(X,g)$. Let $I, J$ be countable index sets. A \defn{measurable partition} $\xi$ for $(X,\mu)$ is a countable collection $\xi=\{A_i\,|\,i\in I\}$ of Borel sets with $\mu(A_i\cap A_j)=0$ for all $i,j\in I$ with $i\neq j$, and
$$
\mu \bigg( X\setminus\bigcup\limits_{i\in I}A_i \bigg) =0.
$$

Let $\xi=\{A_i\,|\,i\in I\}$ and $\eta=\{B_j\,|\,j\in J\}$ be measurable partitions of $(X,\mu)$. Then the \defn{common refinement} $\xi \vee \eta$ of $\xi$ and $\eta$ defined as
$$
\xi \vee \eta = \{A_i\cap B_j \,|\, i\in I, j\in J\}
$$
is also a measurable partition. Let $g^{-1}(\xi)=\{g^{-1}(A_i) \,|\,i\in I\}$, and define for each $n\in\N$,
$$
\xi^n_g=\xi\vee g^{-1}(\xi)\vee\cdots\vee g^{-(n-1)}(\xi).
$$
The \defn{entropy} of $\xi$ is
$$
H_{\mu}(\xi)= - \sum\limits_{i\in I} \mu(A_i)\log\(\mu (A_i)\),
$$
where $0\log 0$ is equal to $0$ by convention. One can show (see \cite[Chapter 4]{Wa82}) that if $H_{\mu}(\xi)<+\infty$, then the following limit exists
$$
h_{\mu}(g,\xi)=\lim\limits_{n\to+\infty} \frac{1}{n} H_{\mu}(\xi^n_g) \in[0,+\infty).
$$

Then we denote the \defn{measure-theoretic entropy} of $g$ for $\mu$ by
\begin{align*}
h_{\mu}(g)=\sup\{h_{\mu}(g,\xi)\,|\, & \xi \text{ is a measurable partition of }\\
                                  &(X,\mu) \text{ with } H_{\mu}(\xi)<+\infty\}.
\end{align*}

By the variational principle (see \cite[Theorem~8.6]{Wa82}), we have 
$$
h_{\operatorname{top}}(g)=\sup\{h_{\mu}(g)\,|\,\mu\in \MMM(X,g)\}.
$$
A measure $\mu$ that achieves the supreme above is called \defn{a measure of maximal entropy} of $g$.

If $f$ is an expanding Thurston map, then
\begin{equation}   \label{eqTopEntropy}
h_{\operatorname{top}} (f)=\log(\deg f),
\end{equation}
and there exists a unique measure of maximal entropy $\mu_f$ for $f$ (see \cite[Theorem~20.9]{BM10} and \cite[Section~3.4 and Section~3.5]{HP09}). Moreover, for each $n\in\N$, the unique measure of maximal entropy $\mu_{f^n}$ of the expanding Thurston map $f^n$ is equal to $\mu_f$ (see \cite[Theorem~20.7 and Theorem~20.9]{BM10}).

We recall that in a compact metric space $(X,d)$, a sequence of finite Borel measures $\mu_n$ converges in the weak$^*$ topology to a finite Borel measure $\mu$, or $\mu_n \stackrel{w^*}{\longrightarrow} \mu$, as $n\longrightarrow +\infty$ if and only if $\lim\limits_{n\to+\infty} \int\! u\,\mathrm{d}\mu_n = \int\! u\,\mathrm{d}\mu$ for each $u\in \CCC(X)$.

We need the following lemmas for weak$^*$ convergence.

\begin{lemma}  \label{lmPushforwardConv}
Let $X$ and $\widetilde{X}$ be two compact metric spaces and $\phi \:X \rightarrow \widetilde{X}$ a continuous map. Let $\mu$ and $\mu_i$, for $i\in\N$, be finite Borel measures on $X$. If
$$
\mu_i \stackrel{w^*}{\longrightarrow} \mu \text{ as } i\longrightarrow +\infty,
$$
then $\phi_*(\mu)$ and $\phi_*(\mu_i)$, $i\in\N$, are finite Borel measures on $\widetilde{X}$, and
$$
\phi_*(\mu_i) \stackrel{w^*}{\longrightarrow} \phi_*(\mu) \text{ as } i\longrightarrow +\infty.
$$
\end{lemma}

Recall for a continuous map $\phi\: X\rightarrow \widetilde X$ between two metric spaces and a Borel measure $\nu$ on $X$, the \emph{push-forward} $\phi_*(\nu)$ of $\nu$ by $\phi$ is defined to be the unique Borel measure that satisfies $(\phi_*(\nu))(B) = \nu\(\phi^{-1}(B)\)$ for each Borel set $B\subseteq \widetilde X$.

\begin{proof}
By the Riesz representation theorem (see for example, \cite[Chapter~7]{Fo99}), the lemma follows if we observe that for each $h\in C(X)$, we have
$$
\int_{\widetilde{X}} \! h \, \mathrm{d}\phi_*\mu_i=\int_X \! (h\circ \phi) \,\mathrm{d}\mu_i \stackrel{i\longrightarrow +\infty}{\longrightarrow} \int_{X} \! (h\circ \phi)\,\mathrm{d}\mu = \int_{\widetilde{X}} \!h\,\mathrm{d}\phi_*\mu.
$$
\end{proof}

\begin{lemma}  \label{lmConvexCombConv}
Let $(X,d)$ be a compact metric space, and $I$ be a finite set. Suppose that $\mu$ and $\mu_{i,n}$, for $i\in I$ and $n\in\N$, are finite Borel measures on $X$, and $w_{i,n}\in [0,+\infty)$, for $i\in I$ and $n\in\N$ such that
\begin{enumerate}
\smallskip

\item $\mu_{i,n} \stackrel{w^*}{\longrightarrow} \mu$, as $n\longrightarrow +\infty$, for each $i\in I$,

\smallskip

\item $\lim\limits_{n\to+\infty} \sum\limits_{i\in I} w_{i,n}  = r$ for some $r\in\R$.
\end{enumerate}
Then $\sum\limits_{i\in I} w_{i,n}\mu_{i,n} \stackrel{w^*}{\longrightarrow}  r\mu$ as $n\longrightarrow +\infty$.
\end{lemma}

\begin{proof}
For each $u\in\CCC(X)$ and each $n\in\N$,
\begin{align*}
     & \Abs{\int \! u \,\mathrm{d}\bigg(\sum\limits_{i\in I} w_{i,n}\mu_{i,n} \bigg)  -  r \int \! u \,\mathrm{d}\mu  } \\
\leq & \sum\limits_{i\in I} w_{i,n} \Abs{\int \! u \,\mathrm{d}\mu_{i,n} -  \int \! u \,\mathrm{d}\mu  }  +  \bigg\lvert r-\sum\limits_{i\in I} w_{i,n} \bigg\rvert \Norm{\mu}.
\end{align*}
Since $\mu_{i,n} \stackrel{w^*}{\longrightarrow} \mu$, as $n\longrightarrow +\infty$, for each $i\in I$, and $\lim\limits_{n\to+\infty} \sum\limits_{i\in I} w_{i,n}  = r$, we can conclude that the right-hand side of the inequality above tends to $0$ as $n\longrightarrow +\infty$.
\end{proof}

We record the following well-known lemma, sometimes known as the Portmanteau Theorem, and refer the reader to \cite[Theorem~2.1]{Bi99} for the proof.

\begin{lemma}   \label{lmPortmanteau}
Let $(X,d)$ be a compact metric space, and $\mu$ and $\mu_i$, for $i\in\N$, be Borel probability measures on $X$. Then the following are equivalent:
\begin{enumerate}
\smallskip

\item $\mu_i \stackrel{w^*}{\longrightarrow} \mu$ as $i\longrightarrow +\infty$;

\smallskip

\item $\limsup\limits_{i\to+\infty} \mu_i(F) \leq \mu(F)$ for each closed set $F\subseteq X$;

\smallskip

\item $\liminf\limits_{i\to+\infty} \mu_i(G) \geq \mu(G)$ for each open set $G\subseteq X$;

\smallskip

\item $\lim\limits_{i\to+\infty} \mu_i(B)= \mu(B)$ for each Borel set $B\subseteq X$ with $\mu(\partial B) = 0$.
\end{enumerate}

\end{lemma}

\begin{lemma}   \label{lmLocalPerturb}
Let $(X,d)$ be a compact metric space. Suppose that $A_i  \subseteq X$, for $i\in\N$, are finite subsets of $X$ with maps $\phi_i\:A_i\rightarrow X$ such that 
\begin{equation*}
\lim\limits_{i\to+\infty} \max\{d(x,\phi_i(x))\,|\, x\in A_i\}=0.
\end{equation*}
Let $m_i\:A_i\rightarrow \R$, for $i\in\N$, be functions that satisfy
\begin{equation*}
\sup_{i\in\N} \Norm{m_i}_1 = \sup_{i\in\N} \sum\limits_{x\in A_i}\abs{m_i(x)} <+\infty.
\end{equation*}
Define for each $i\in\N$,
\begin{equation*}
\mu_i=\sum\limits_{x\in A_i} m_i(x)\delta_x, \quad \widetilde{\mu}_i=\sum\limits_{x\in A_i} m_i(x)\delta_{\phi_i(x)}.
\end{equation*}
If
\begin{equation*}
\mu_i  \stackrel{w^*}{\longrightarrow} \mu \text{ as } i\longrightarrow +\infty,
\end{equation*}
for some finite Borel measure $\mu$ on $X$, then
\begin{equation*}
\widetilde\mu_i  \stackrel{w^*}{\longrightarrow} \mu  \text{ as } i\longrightarrow +\infty.
\end{equation*}
\end{lemma}

\begin{proof}
It suffices to prove that for each continuous function $g\in\CCC(X)$,
$$
\int \! g \, \mathrm{d} \mu_i - \int \! g \,\mathrm{d} \widetilde\mu_i \longrightarrow 0 \text{ as } i \longrightarrow +\infty.
$$
Indeed, $g$ is uniformly continuous, so for each $\epsilon > 0$, there exists $N\in\N$ such that for each $n>N$ and for each $x\in A_n$, we have $\Abs{g(x)-g(\phi_n(x))}<\epsilon$. Thus
$$
\Abs{\int \! g \, \mathrm{d} \mu_n - \int \! g \,\mathrm{d} \widetilde\mu_n}  \leq \sum\limits_{x\in A_n} \Abs{g(x)-g(\phi_n(x))}\Abs{m_n(x)} \leq \epsilon \sup_{n\in\N}\Norm{m_n}_1.
$$
\end{proof}

The following lemma is a reformulation of Lemma~20.2 in \cite{BM10}. We will later generalize it in Lemma \ref{lmCoverEdges}.

\begin{lemma}[M.~Bonk \& D.~Meyer, 2010]  \label{lmCoverEdgesBM}
Let $f$ be an expanding Thurston map, and $\CC\subseteq S^2$ be an $f^N$-invariant Jordan curve containing $\post{f}$ for some $N\in\N$. Then there exists a constant $L_0 \in [1,\deg{f})$ with the following property:

For each $m\in\N_0$ with $m \equiv 0 \pmod N$, there exists a constant $C_0>0$ such that for each $k\in\N_0$ with $k \equiv 0 \pmod N$ and each $m$-edge $e$, there exists a collection $M_0$ of $(m+k)$-tiles with $\card{M_0}\leq C_0 L_0^k$ and $e \subseteq \inte\Big( \bigcup\limits_{X\in M_0} X \Big)$.
\end{lemma}

Let $F$ be an expanding Thurston map with an $F$-invariant Jordan curve $\CC\subseteq S^2$ containing $\post{F}$. As before, we let $w_w=\card \X^1_{ww} $ denote the number of white $1$-tiles contained in the white $0$-tile, $b_w=\card \X^1_{bw} $ the number of black $1$-tiles contained in the white $0$-tile, $w_b=\card \X^1_{wb} $ the number of white $1$-tiles contained in the black $0$-tile, and $b_b=\card \X^1_{bb} $ the number of black $1$-tiles contained in the black $0$-tile. We define
\begin{equation} \label{eqDefwb}
w=\frac{b_w}{b_w+w_b}, \quad  b=\frac{w_b}{b_w+w_b}.
\end{equation}
Note that (see the discussion in \cite{BM10} proceeding Lemma~20.1 in Chapter 20) $b_w,w_b, w,b>0$, $w+b=1$, and
\begin{equation}  \label{eqWw-Bw<deg}
\abs{w_w-b_w} < \deg F.
\end{equation}

M.~Bonk and D.~Meyer give the following characterization of the unique measure of maximal entropy of $F$ (see \cite[Proposition~20.7 and Theorem~20.9]{BM10}):

\begin{theorem}[M.~Bonk \& D.~Meyer, 2010]    \label{thmBMCharactMOME}
Let $F$ be an expanding Thurston map with an $F$-invariant Jordan curve $\CC\subseteq S^2$. Then there is a unique measure of maximal entropy $\mu_F$ of $F$, which is characterized among all Borel probability measures by the following property:

for each $n\in\N_0$ and each $n$-tile $X^n\in\X^n(F,\CC)$, 
\begin{equation}
\mu(X^n) = \begin{cases} w(\deg F)^{-n} & \text{if } X^n \in \X^n_w(F,\CC), \\ b(\deg F)^{-n} & \text{if } X^n \in \X^n_b(F,\CC). \end{cases}
\end{equation}
\end{theorem}

We now state our first characterization of the measure of maximal entropy $\mu_f$ of an expanding Thurston map $f$.

\begin{theorem} \label{thmWeakConv}
Let $f$ be an expanding Thurston map with its measure of maximal entropy $\mu_f$. Let $\CC \subseteq S^2$ be an $f^n$-invariant Jordan curve containing $\post{f}$ for some $n\in\N$. Fix a visual metric $d$ for $f$. Consider any sequence of non-negative numbers $\{\alpha_i\}_{i\in\N_0}$ with $\lim\limits_{i\to +\infty} \alpha_i = 0$, and any sequence of functions $\{\beta_i\}_{i\in\N_0}$ with $\beta_i$ mapping each white $i$-tile $X^i\in\X^i_w(f,\CC)$ to a point $\beta_i(X^i) \in N_d^{\alpha_i}(X^i)$. Let
$$
\mu_i = \frac{1}{(\deg f)^i} \sum\limits_{X^i\in\X^i_w(f,\CC)} \delta_{\beta_i(X^i)}, \quad i \in \N_0.
$$
Then
$$
\mu_i \stackrel{w^*}{\longrightarrow} \mu_f \text{ as } i\longrightarrow +\infty.
$$
\end{theorem}

Recall that $N_d^{\alpha_i}(X^i)$ denotes the open $\alpha_i$-neighborhood of $X^i$ in $(S^2,d)$. This theorem says that a sequence of probability measures $\{\mu_i\}_{i\in\N}$, with $\mu_i$ assigning the same weight to a point near each white $i$-tile, converges in the weak$^*$ topology to the measure of maximal entropy. In some sense, it asserts the equidistribution of the white $i$-tiles with respect to the measure of maximal entropy.  

In order to prove the above theorem, we first prove a weaker version of it.

\begin{prop}   \label{propWeakConv}
Let $F$ be an expanding Thurston map with its measure of maximal entropy $\mu_F$ and an $F$-invariant Jordan curve $\CC\subseteq S^2$ containing $\post{F}$. Consider any sequence of functions $\{\beta_i\}_{i\in\N_0}$ with $\beta_i$ mapping each white $i$-tile $X^i\in\X^i_w(F,\CC)$ to a point $\beta_i(X^i) \in \inte X^i$ for each $i\in\N_0$. Let
$$
\mu_i = \frac{1}{(\deg F)^i} \sum\limits_{X^i\in\X^i_w(F,\CC)} \delta_{\beta_i(X^i)}, \quad i\in \N_0.
$$
Then
$$
\mu_i \stackrel{w^*}{\longrightarrow} \mu_F \text{ as } i\longrightarrow +\infty.
$$
\end{prop}

\begin{proof}
Note that $\card \X^i_w =(\deg F)^i$, so $\mu_i$ is a probability measure for each $i\in\N_0$. Thus by Alaoglu's theorem, it suffices to prove that for each Borel measure $\mu$ which is a subsequential limit of $\{\mu_i\}_{i\in\N_0}$ in the weak$^*$ topology, we have $\mu=\mu_F$.

Let $\{i_n\}_{n\in\N} \subseteq \N$ be an arbitrary strictly increasing sequence such that
$$
\mu_{i_n} \stackrel{w^*}{\longrightarrow} \mu \text{ as } n\longrightarrow +\infty,
$$
for some Borel measure $\mu$. Clearly $\mu$ is also a probability measure.

Recall the definitions of $w,b\in(0,1)$ and $w_w,b_w,w_b,b_b$ (see (\ref{eqDefwb})). For each $m,i\in\N_0$ with $0 \leq m \leq i$, each white $m$-tile $X^m_w\in\X^m_w$, and each black $m$-tile $X^m_b\in\X^m_b$, by the formulas in Lemma~20.1 in \cite{BM10}, we have
\begin{align}
\mu_i(X^m_w) =   & \frac{1}{(\deg F)^i} \card\{X^i\in\X^i_w \,|\,X^i\subseteq X^m_w\}  \notag \\
           =   & \frac{1}{(\deg F)^i}\(w(\deg F)^{i-m}+b(w_w-b_w)^{i-m}\)    \label{eqMu_iX^kW}
\end{align}
and similarly,
\begin{equation}   \label{eqMu_iX^kB}
\mu_i(X^m_b) =   \frac{1}{(\deg F)^i}\(w(\deg F)^{i-m}-b(w_w-b_w)^{i-m}\) 
\end{equation}

\smallskip

We claim that for each $m$-tile $X^m\in\X^m$ with $m\in\N_0$, we have $\mu(\partial X^m)=0$.

To establish the claim, it suffices to prove that $\mu(e)=0$ for each $m$-edge $e$ with $m\in\N_0$. Applying Lemma \ref{lmCoverEdgesBM} in the case $f=F$ and $n=1$, we get that there exists constants $1<L_0<\deg F$ and $C_0>0$ such that for each $k\in\N_0$, there is a collection $M_0^k$ of $(m+k)$-tiles with $\card M_0^k\leq C_0L_0^k$ such that $e$ is contained in the interior of the set $\bigcup\limits_{X\in M_0^k} X$. So by (\ref{eqWw-Bw<deg}), (\ref{eqMu_iX^kW}), (\ref{eqMu_iX^kB}), and Lemma~\ref{lmPortmanteau}, we get  
\begin{align*}
\mu(e)  & \leq \mu \bigg(\inte \big(\bigcup\limits_{X\in M_0^k} X\big)\bigg)  \\
        & \leq \limsup_{l\to +\infty} \mu_{m+k+l} \bigg(\inte \big(\bigcup\limits_{X\in M_0^k} X\big)\bigg)\\
        & \leq \limsup_{l\to +\infty} \sum\limits_{X\in M_0^k} \mu_{m+k+l}(X)\\
        & \leq \sum\limits_{X\in M_0^k} \limsup_{l\to +\infty} \mu_{m+k+l}(X) \\
        & \leq C_0 L_0^k\frac{w+b}{(\deg F)^{m+k}}.
\end{align*}
By letting $k\longrightarrow +\infty$, we get $\mu(e)=0$, proving the claim.

\smallskip

Thus by (\ref{eqWw-Bw<deg}), (\ref{eqMu_iX^kW}), (\ref{eqMu_iX^kB}), the claim, and Lemma~\ref{lmPortmanteau}, we can conclude that for each $m\in\N_0$, and each white $m$-tile $X^m_w\in\X^m_w$, each black $m$-tile $X^m_b\in\X^m_b$, we have that 
\begin{equation*}
\mu(X^m_w)  = \lim\limits_{n\to+\infty} \mu_{i_n}(X^m_w) = w(\deg F)^{-m},
\end{equation*} 
\begin{equation*}
\mu(X^m_b)  = \lim\limits_{n\to+\infty} \mu_{i_n}(X^m_b)=  b(\deg F)^{-m}.  
\end{equation*}   
By Theorem~\ref{thmBMCharactMOME}, therefore, the measure $\mu$ is equal to the unique measure of maximal entropy $\mu_F$ of $F$.
\end{proof}

As a consequence of the above proposition, we have

\begin{cor}  \label{corWeakConvPreimageInt}
Let $f$ be an expanding Thurston map with its measure of maximal entropy $\mu_f$. Let $\CC\subseteq S^2$ be an $f^n$-invariant Jordan curve containing $\post{f}$ for some $n\in\N$. Fix an arbitrary $p \in \inte X_w^0$ where $X_w^0$ is the white $0$-tile for $(f,\CC)$. Define, for $i\in\N$,
$$
\nu_i=\frac{1}{(\deg f)^i}\sum\limits_{q\in f^{-i}(p)}\delta_q.
$$
Then
$$
\nu_i \stackrel{w^*}{\longrightarrow} \mu_f \text{ as } i\longrightarrow +\infty.
$$
\end{cor}

\begin{proof}
First observe that since $p$ is contained in the interior of the white $0$-tile, each $q\in f^{-n}(p)$ is contained in the interior of one of the white $n$-tiles, and each white $n$-tile contains exactly one $q$ with $f^n(q)=p$. So by Proposition~\ref{propWeakConv},
\begin{equation}  \label{eqvni}
\nu_{ni} \stackrel{w^*}{\longrightarrow} \mu_{f^n} \text{ as } i\longrightarrow +\infty,
\end{equation}
where $\mu_{f^n}$ is the unique measure of maximal entropy of $f^n$, which is equal to $\mu_f$ (see \cite[Theorem~20.7 and Theorem~20.9]{BM10}).

Then note that for $k>1$,
\begin{equation} \label{eqfnu}
f_*\nu_k =\frac{1}{(\deg f)^k}\sum\limits_{q\in f^{-k}(p)} \delta_{f(q)}=\frac{1}{(\deg f)^{k-1}}\sum\limits_{q\in f^{-k+1}(p)} \delta_q =\nu_{k-1}.
\end{equation}
The second equality above follows from the fact that the number of preimages of each point in $f^{-k+1}(p)$ is exactly $\deg f$.

So by (\ref{eqvni}), (\ref{eqfnu}), Lemma~\ref{lmPushforwardConv}, and the fact that $\mu_f$ is invariant under pushforward of $f$ from Theorem~20.9 in \cite{BM10}, for each $k\in \{0,1,\dots,n-1\}$, we get
$$
\nu_{ni-k} = (f_*)^k \nu_{ni} \stackrel{w^*}{\longrightarrow} (f_*)^k \mu_f = \mu_f \text{ as } i\longrightarrow +\infty.
$$
Therefore
$$
\nu_i \stackrel{w^*}{\longrightarrow} \mu_f \text{ as } i\longrightarrow +\infty.
$$
\end{proof}

\begin{rems}    \label{rmEquidistributionTilesWhiteBlack}
We can replace ``white'' by ``black'', $\X^i_w$ by $\X^i_b$, and $X^0_w$ by $X^0_b$ in the statements of Theorem~\ref{thmWeakConv}, Proposition~\ref{propWeakConv}, and Corollary~\ref{corWeakConvPreimageInt}. The proofs are essentially the same.
\end{rems}

\begin{proof} [Proof of Theorem~\ref{thmWeakConv}]
Fix an arbitrary $p \in \inte X_w^0$ in the interior of the while $0$-tile $X_w^0$ for the cell decomposition induced by $(f,\CC)$.

As in the proof of Corollary~\ref{corWeakConvPreimageInt}, for each $i\in\N_0$, there is a bijective correspondence between points in $f^{-i}(p)$  and the set of white $i$-tiles, namely, each $q\in f^{-i} (p)$ corresponds to the unique white $i$-tile, denoted as $X_q$, containing $q$. Then we define functions $\phi_i\: f^{-i}(p)\rightarrow S^2$ by setting $\phi_i(q)=\beta_i(X_q)$.

For our fixed visual metric $d$, there exists $C \geq 1$ and $\Lambda>1$ such that for each $n\in\N_0$ and each $n$-tile $X^n\in\X^n$, $\diam_d(X^n) \leq C \Lambda^{-n}$ (see Lemma~\ref{lmBMCellSizeBounds}). So for each $i\in N_0$ and each $q\in f^{-i}(p)$, we have
$$
d(q,\phi_i(q)) \leq  d(\phi_i(q),X_q)+\diam_d(X_q) \leq \alpha_i + C\Lambda^{-i}.
$$
Thus $\lim\limits_{i\to+\infty} \max\{d(x,\phi_i(x))\,|\, x\in f^{-i}(p)\}=0$.

For $i\in\N_0$, define
$$
\widetilde\mu_i = \frac{1}{(\deg f)^i} \sum\limits_{q\in f^{-i}(p)} \delta_q.
$$
Note that for $i\in\N_0$,
$$
\mu_i = \frac{1}{(\deg f)^i} \sum\limits_{X^i\in\X^i_w(f,\CC)} \delta_{\beta_i(X^i)}  = \frac{1}{(\deg f)^i} \sum\limits_{q\in f^{-i}(p)} \delta_{\phi_i(q)}.
$$

Then by Corollary~\ref{corWeakConvPreimageInt},
$$
\widetilde\mu_i \stackrel{w^*}{\longrightarrow} \mu_f \text{ as } i\longrightarrow +\infty.
$$
Therefore, by Lemma~\ref{lmLocalPerturb} with $A_i=f^{-i}(p)$ and $m_i(x) = \frac{1}{(\deg f)^i}, i\in \N_0$, we can conclude that
$$
\mu_i \stackrel{w^*}{\longrightarrow} \mu_f \text{ as } i\longrightarrow +\infty.
$$
\end{proof}

We are now ready to prove the equidistribution of preimages of an arbitrary point with respect to the measure of maximal entropy $\mu_f$.

\begin{proof}[Proof of Theorem~\ref{thmWeakConvPreImg}]
By Theorem~1.2 in \cite{BM10} or Corollary~\ref{corCexists}, we can fix an $f^n$-invariant Jordan curve $\CC \subseteq S^2$ containing $\post{f}$ for some $n\in\N$. We consider the cell decompositions induced by $(f,\CC)$.

We first prove (\ref{eqWeakConvPreImgWithWeight}).

We assume that $p$ is contained in the (closed) white $0$-tile. The proof for the case when $p$ is contained in the black $0$-tile is exactly the same except that we need to use a version of Theorem~\ref{thmWeakConv} for black tiles instead of using Theorem~\ref{thmWeakConv} literally, see Remark~\ref{rmEquidistributionTilesWhiteBlack}.

Observe that for each $i\in\N_0$ and each $q\in f^{-i}(p)$, the number of white $i$-tiles that contains $q$ is exactly $\deg_{f^i}(q)$. On the other hand, each white $i$-tile contain exactly one point $q$ with $f^i(q)=p$. So we can define $\beta_i\:\X_w^i \rightarrow S^2$ by mapping a white $i$-tile to the point $q$ in it that satisfies $f^i(q)=p$. Define $\alpha_i\equiv 0$. Theorem~\ref{thmWeakConv} applies, and thus (\ref{eqWeakConvPreImgWithWeight}) is true.

Next, we prove (\ref{eqWeakConvPreImgWoWeight}). The proof breaks into three cases.

\smallskip

Case 1. Assume that $p \notin \post f$. Then $\deg_f(x)=1$ for all $x\in\bigcup\limits_{n=1}^{+\infty} f^{-n}(p)$. So $\widetilde{\nu}_i = \nu_i$ for each $i\in\N$. Then (\ref{eqWeakConvPreImgWoWeight}) follows from (\ref{eqWeakConvPreImgWithWeight}) in this case.

\smallskip

Case 2. Assume that $p\in \post f$ and $p$ is not periodic. Then there exists $N\in\N$ such that $f^{-N}(p) \cap \post f = \emptyset$. For otherwise, there exists a point $z\in\post f$ which belongs to $f^{-c}(p)$ for infinitely many distinct $c\in\N$. In particular, there exist two integers $a>b>0$ such that $z\in f^{-a}(p)\cap f^{-b}(p)$. Then $f^{a-b}(p)= p$, a contradiction. So $\deg_f(q)=1$ for each $q\in \bigcup\limits_{x\in f^{-N}(p)}\bigcup\limits_{i=1}^{+\infty} f^{-i}(x)$. Note that for each $x\notin \post f$ and each $i\in\N$, the number of preimages of $x$ under $f^i$ is exactly $(\deg f)^i$. Then for each $i\in\N$, $Z_{i+N}=Z_N (\deg f)^i$, and 
$$
\widetilde{\nu}_{i+N}  = \frac{1}{Z_{i+N}}  \sum\limits_{q\in f^{-(i+N)}(p)} \delta_q  = \frac{1}{Z_N}  \sum\limits_{x\in f^{-N}(p)} \bigg( \frac{1}{(\deg f)^i} \sum\limits_{q\in f^{-i}(x)}  \delta_q  \bigg).
$$
For each $x\in f^{-N}(p)$, by Case 1,
$$
 \frac{1}{(\deg f)^i} \sum\limits_{q\in f^{-i}(x)}  \delta_q  \stackrel{w^*}{\longrightarrow} \mu_f \text{ as } i\longrightarrow +\infty.
$$
Thus each term in the sequence $\{\widetilde{\nu}_{i+N}\}_{i\in\N}$ is a convex combination of the corresponding terms in sequences of measures, each of which converges to $\mu_f$ in the weak$^*$ topology. Hence by Lemma~\ref{lmConvexCombConv}, the sequence $\{\widetilde{\nu}_{i+N}\}_{i\in\N}$ also converges to $\mu_f$ in the weak$^*$ topology in this case.

\smallskip

Case 3. Assume that $p\in\post f$ and $p$ is periodic with period $k\in\N$. Let $l=\card(\post f)$. We first note that for each $m,N\in\N$, the inequality
$$
Z_{m+N} \geq (Z_m - l)(\deg f)^N,
$$
and equivalently,
$$
\frac{Z_m}{Z_{m+N}} \leq \frac{1}{(\deg f)^N} + \frac{l}{Z_{m+N}}
$$
hold, since there are at most $l$ points in $Z_m\cap \post f$. So by Lemma~\ref{lmPreImageDense}, for each $\epsilon > 0$ and each $N$ large enough such that ${1}/{(\deg f)^N} < {\epsilon}/{2}$ and ${l}/{Z_{m+N}} <  {\epsilon}/{2}$, we get ${Z_m}/{Z_{m+N}}  <\epsilon$ for each $m\in\N$. We fix $j\in\N$ large enough such that ${Z_{m-jk}}/{Z_m} < \epsilon$ for each $m>jk$. Observe that for each $m>jk$,
\begin{align}  \label{eqPfthmWeakConvPreImg}
\widetilde{\nu}_m = &  \frac{1}{Z_m} \sum\limits_{q\in f^{-m}(p)} \delta_q  \notag \\
                  = &  \frac{1}{Z_m} \bigg(\sum\limits_{q\in f^{-(m-jk)}(p)} \delta_q +  \sum\limits_{x\in f^{-jk}(p)\setminus\{p\}} \sum\limits_{q\in f^{-(m-jk)}(x)} \delta_q   \bigg)   \\
                  = &  \frac{Z_{m-jk}}{Z_m} \bigg( \frac{1}{Z_{m-jk}}  \sum\limits_{q\in f^{-(m-jk)}(p)} \delta_q      \bigg)  + \frac{ 1}{Z_m} \sum\limits_{x\in f^{-jk}(p)\setminus\{p\}} \notag  \\
                    & \card\(f^{-(m-jk)}(x)\) \bigg( \frac{1}{\card\(f^{-(m-jk)}(x)\)  }  \sum\limits_{q\in f^{-(m-jk)}(x)} \delta_q   \bigg).   \notag
\end{align}
Note that no point $x\in f^{-jk}(p) \setminus\{p\}$ is periodic. Indeed, if $x\in f^{-jk}(p) \setminus\{p\}$ were periodic, then $x\in \bigcup\limits_{i=0}^{k-1}f^i(p)$, and so $x$ would have period $k$ as well. Thus $x= f^{jk}(x)=p$, a contradiction. Hence by Case~1 and Case~2, for each $x\in  f^{-jk}(p) \setminus\{p\}$,
$$
\frac{1}{\card\(f^{-(m-jk)}(x)\)  }  \sum\limits_{q\in f^{-(m-jk)}(x)} \delta_q \stackrel{w^*}{\longrightarrow} \mu_f \text{ as } m\longrightarrow +\infty.
$$

Let $\mu\in\PPP(S^2)$ be an arbitrary subsequential limit of $\{\widetilde\nu_m\}_{m\in\N}$ in the weak$^*$ topology. For each strictly increasing sequence $\{m_i\}_{i\in\N}$ in $\N$ that satisfies
$$
\widetilde{\nu}_{m_i} \stackrel{w^*}{\longrightarrow} \mu \text{ as } i\longrightarrow +\infty,
$$
we can assume, due to Alaoglu's Theorem, by choosing a subsequence if necessary, that 
$$
\frac{Z_{m_i-jk}}{Z_{m_i}} \bigg( \frac{1}{Z_{m_i-jk}}  \sum\limits_{q\in f^{-(m_i-jk)}(p)} \delta_q      \bigg)   \stackrel{w^*}{\longrightarrow} \eta \text{ as } i\longrightarrow +\infty,
$$
for some Borel measure $\eta$ with total variation $\Norm{\eta} \leq \epsilon$. Observe that for each $i\in \N$,
$$
\frac{ 1}{Z_{m_i}} \sum\limits_{x\in f^{-jk}(p)\setminus\{p\}} \card\(f^{-(m_i-jk)}(x)\)   = 1- \frac{Z_{m_i-jk}}{Z_{m_i}},
$$
since $p\in f^{-jk}(p)$ and $\card\(f^{-(m_i-jk)}(p)\) = Z_{m_i-jk}$. By choosing a subsequence of $\{m_i\}_{i\in\N}$ if necessary, we can assume that there exists $r\in [0,\epsilon]$ such that
\begin{equation*}
\lim\limits_{i\to+\infty}  \frac{Z_{m_i-jk}}{Z_{m_i}} = r.
\end{equation*} 
So by taking the limits of both sides of (\ref{eqPfthmWeakConvPreImg})  in the weak$^*$ topology along the subsequence $\{m_i\}_{i\in\N}$, we get from Lemma~\ref{lmConvexCombConv} that $\mu = \eta + (1-r)\mu_f$. Thus
$$
\Norm{\mu-\mu_f} \leq  \Norm{\eta} + r \Norm{\mu_f} \leq 2 \epsilon.
$$
Since $\epsilon$ is arbitrary, we can conclude that $\mu= \mu_f$. We have proven in this case that each subsequential limit of $\{\widetilde{\nu}_m\}_{m\in\N}$ in the weak$^*$ topology is equal to $\mu_f$. Therefore (\ref{eqWeakConvPreImgWoWeight}) is true in this case.
\end{proof}

In order to prove Theorem~\ref{thmWeakConvPrePerPts}, we will need Lemma~\ref{lmCoverEdges} which is a generalization of Lemma \ref{lmCoverEdgesBM}.

\begin{lemma}     \label{lmSumLocalDegreeBound}
Let $f$ be an expanding Thurston map and $d=\deg f$. Then there exist constants $C>0$ and $\alpha\in (0,1]$ such that for each nonempty finite subset $M$ of $S^2$ and each $n\in\N$, we have
\begin{equation}  \label{eqSumLocalDegreeBound}
\frac{1}{d^n} \sum\limits_{x\in M}  \deg_{f^n}(x) \leq C  \max \bigg\{ \( \frac{\card M}{d^n} \)^\alpha, \frac{\card M}{d^n} \bigg\}.
\end{equation}
\end{lemma}

Note that when $\card M \leq d^n$, the right-hand side of (\ref{eqSumLocalDegreeBound}) becomes $C  \max  \( \frac{\card M}{d^n} \)^\alpha$.

\begin{proof}
Let $m=\card M$. Set
\begin{equation*}
D=\prod\limits_{x\in \crit f} \deg_f(x).
\end{equation*}

In order to establish the lemma, we consider the following three cases.

\smallskip

Case 1: Suppose that $f$ has no periodic critical points. Then since for each $x\in S^2$ and each $n\in\N$,
\begin{equation}  \label{eqDegFn}
\deg_{f^n}(x) = \deg_f(x)\deg_f(f(x))\cdots \deg_f(f^{n-1}(x)),
\end{equation}
it is clear that $\deg_{f^n}(x)\leq D$. So
\begin{equation*}
\frac{1}{d^n} \sum\limits_{x\in M}  \deg_{f^n}(x) \leq D  \frac{m}{d^n}.
\end{equation*}
Thus in this case, $C=D$ and $\alpha=1$.

\smallskip

Case 2: Suppose that $f$ has periodic critical points, but all periodic critical points are fixed points of $f$. 

Let $T_0 = \{ x\in\crit f \,|\, f(x)=x \}$ be the set of periodic critical points of $f$. Then define recursively for each $i\in\N$,
$$
T_i = f^{-1}(T_{i-1}) \setminus \bigcup\limits_{j=0}^{i-1} T_j.
$$
Define $T_{-1}=S^2\setminus \bigcup\limits_{j=0}^{+\infty} T_j$, and $\widetilde T_i = S^2\setminus \bigcup\limits_{j=0}^{i} T_j$ for each $i\in\N_0$. Set $t_0 = \card T_0$. Since $\crit f$ is a finite set, we have $1\leq t_0 < +\infty$. Then for each $i\in\N$, we have
$$
\card T_i \leq d^i t_0. 
$$

We note that if $\deg_f(x)=d$ for some $x\in T_0$, then $f^{-i}(x)=\{x\}$ for each $i\in\N$, contradicting Lemma~\ref{lmPreImageDense}. So $\deg_f(x)\leq d-1$ for each $x\in T_0$. Thus for each $x\in T_0$ and each $m\in\N$, we have
$$
\deg_{f^m} (x)\leq  (d-1)^m.
$$
Moreover, for each $i,m\in\N$ with $i<m$ and each $x\in T_i$, we get
\begin{align*}
\deg_{f^m}(x) & = \deg_f(x)\deg_f(f(x))\cdots\deg_f(f^{i-1}(x))\deg_{f^{m-i}}(f^i(x))  \\
              & \leq D (d-1)^{m-i}.
\end{align*}
Similarly, for each $i,m\in\N$ with $i \geq m$ and each $x\in \widetilde T_i$, we have
$$
\deg_{f^m} (x) \leq D.
$$

Thus for each $n\in\N$,
\begin{align*}
     & \frac{1}{d^n}\sum\limits_{x\in M} \deg_{f^n}(x) \\
=    & \frac{1}{d^n}\sum\limits_{j=-1}^{+\infty}  \sum\limits_{x\in M\cap T_j}  \deg_{f^{n}}(x)\\
\leq & \frac{1}{d^n} \bigg( \sum\limits_{j=0}^{n}  \sum\limits_{x\in M\cap T_j} D(d-1)^{n-j} + \sum\limits_{x\in M\cap \widetilde T_n} D \bigg).
\end{align*}
Note that the more points in $M$ lie in $T_j$ with $j\in [0,n]$ as small as possible, the larger the right-hand side of the last inequality is. So the right-hand side of the last inequality is
\begin{align*}
\leq & \frac{1}{d^n} \Bigg( \sum\limits_{j=0}^{\lceil \log_d \lceil\frac{m}{t_0}\rceil\rceil} (\card T_j)D(d-1)^{n-j}  +mD \Bigg)  \\
\leq & \frac{D t_0}{d^n}\sum\limits_{j=0}^{\lceil \log_d \lceil\frac{m}{t_0}\rceil\rceil}  d^j (d-1)^{n-j}  + \frac{mD}{d^n} \\
\leq & D t_0 \( \frac{d-1}{d} \)^n \sum\limits_{j=0}^{\lceil \log_d m \rceil} \( \frac{d}{d-1} \)^j  + \frac{mD}{d^n}\\
=    & D t_0 \( \frac{d-1}{d} \)^n \frac{\( \frac{d}{d-1} \)^{\lceil \log_d m \rceil+1}  -1 }{\frac{d}{d-1} -1 }  + \frac{mD}{d^n} \\
\leq & D t_0 \( \frac{d-1}{d} \)^n \( \frac{d}{d-1} \)^{2+ \log_d m } (d-1) + \frac{mD}{d^n}\\
\leq & D t_0 \frac{d^2}{d-1} \(  \( \frac{d-1}{d} \)^{n- \log_d m } +  \frac{m}{d^n} \) \\
=    & \frac12 E_f \( d^{(n- \log_d m ) \log_d \frac{d-1}{d}} +  \frac{m}{d^n} \) \\
\leq & E_f \max \Big\{\( \frac{m}{d^n} \)^{\log_d \frac{d}{d-1}}, \frac{m}{d^n} \Big\},
\end{align*}
where $E_f= 2 D t_0 \frac{d^2}{d-1}$ is a constant that only depends on $f$. Thus in this case, $C=E_f$ and $\alpha= \log_d \frac{d}{d-1}\in(0,1]$.

\smallskip

Case 3: Suppose that $f$ has periodic critical points that may not be fixed points of $f$. 

Set $\kappa$ to be the product of the periods of all periodic critical points of $f$. 

We claim that each periodic critical point of $f^\kappa$ is a fixed point of $f^\kappa$. Indeed, if $x$ is a periodic critical point of $f$ satisfying $f^{\kappa p}(x)=x$ for some $p\in \N$, then by (\ref{eqDegFn}), there exists an integer $i\in\{0,1,\dots,\kappa -1\}$ such that $f^i(x)\in \crit f$. Then $f^i(x)$ is a periodic critical point of $f$, so $f^\kappa(f^i(x)) = f^i(x)$. Thus
\begin{align*}
f^\kappa(x) & = f^{\kappa-i}(f^i(x)) = f^{\kappa-i + \kappa}(f^i(x)) = \dots \\
            & = f^{\kappa-i + (p-1)\kappa}(f^i(x)) = f^{\kappa p} (x) = x.
\end{align*}
The claim now follows.

Note that for each $n\in\N$,
\begin{equation*}
       \frac{1}{d^n}\sum\limits_{x\in M} \deg_{f^n}(x)  
 \leq  d^\kappa \frac{1}{d^{\kappa \lceil \frac{n}{\kappa}\rceil }}\sum\limits_{x\in M} \deg_{f^{\kappa \lceil \frac{n}{\kappa}\rceil }}(x).
\end{equation*}
Hence by applying Case 2 for $f^\kappa$, we get a constant $E_{f^\kappa}$ that depends only on $f$, such that the right-hand side of the above inequality is
\begin{align*}
\leq  &  d^\kappa E_{f^\kappa} \max \Bigg\{ \(\frac{m}{d^{\kappa \lceil \frac{n}{\kappa}\rceil }} \)^{\log_{d^\kappa}  \frac{d^\kappa}{d^\kappa -1} }  , \frac{m}{d^{\kappa \lceil \frac{n}{\kappa}\rceil }}  \Bigg\}   \\
\leq  &  d^\kappa E_{f^\kappa}  \max \Big\{ \(\frac{m}{d^n} \)^{\log_{d^\kappa}  \frac{d^\kappa}{d^\kappa -1} } ,  \frac{m}{d^n}  \Big\}.
\end{align*}
Thus in this case $C= d^\kappa E_{f^\kappa} $ and $\alpha = \log_{d^\kappa} \frac{d^\kappa}{d^\kappa -1} \in (0,1] $.
\end{proof}

Now we formulate a generalization of Lemma~\ref{lmCoverEdgesBM}.

\begin{lemma} \label{lmCoverEdges}
Let $f$ be an expanding Thurston map, and $\CC\subseteq S^2$ be an $f^N$-invariant Jordan curve containing $\post{f}$ for some $N\in\N$. Then there exists a constant $L \in [1, \deg{f})$ with the following property:

For each $m\in\N_0$, there exists a constant $D>0$ such that for each $k\in\N_0$ and each $m$-edge $e$, there exists a collection $M$ of $(m+k)$-tiles with $\card{M}\leq D L^k$ and $e \subseteq \inte \Big( \bigcup\limits_{X\in M} X\Big)$.
\end{lemma}

\begin{proof}
We denote $d=\deg f$, and consider the cell decompositions induced by $(f,\CC)$ in this proof.

\smallskip

Step 1: We first assume that for some $m\in \N$, there exist constants $L\in [1,d)$ and $D>0$ such that for each $k\in\N_0$ and each $m$-edge $e$, there exists a collection $M$ of $(m+k)$-tiles with $\card M \leq DL^k$ and $e\subseteq \inte \Big( \bigcup\limits_{X\in M} X\Big)$. Then by Proposition~\ref{propCellDecomp}(i), for each $(m-1)$-edge $e$, we can choose an $m$-edge $e'$ such that $f(e')=e$. For each $k\in\N_0$, there exists a collection $M'$ of $(m+k)$-tiles with $\card M' \leq DL^k$ and $e'\subseteq \inte \Big( \bigcup\limits_{X\in M'} X\Big)$. We set $M$ to be the collection $\{ f(X) \,|\, X\in M'\}$ of $(m-1+k)$-tiles. Then $\card M \leq \card M' \leq DL^k$ and $e\subseteq \inte \Big( \bigcup\limits_{X\in M} X\Big)$. Hence, it suffices to prove the lemma for ``each $m\in\N_0$ with $m\equiv 0 \pmod N$'' instead of ``each $m\in\N_0$''.

\smallskip

Step 2: We will prove the following statement by induction on $\kappa$:
\begin{itemize}

\smallskip
\item[] For each $\kappa \in \{0,1,\dots,N-1\}$, there exists a constant $L_\kappa \in [1, d)$ with the following property:

For each $m\in\N_0$ with $m \equiv 0 \pmod N$, there exists a constant $D_\kappa>0$ such that for each $k\in\N_0$ with $k \equiv \kappa \pmod N$ and each $m$-edge $e$, there exists a collection $M_{m,k,e}$ of $(m+k)$-tiles that satisfies $\card{M_{m,k,e}}\leq D_\kappa L_\kappa^k$ and $e \subseteq \inte \Big( \bigcup\limits_{X\in M_{m,k,e}} X\Big)$.
\smallskip

\end{itemize}

Lemma~\ref{lmCoverEdgesBM} gives the case for $\kappa=0$. For the induction step, we assuming the above statement for some $\kappa \in[0,N-1]$.

Let $i\in\N_0$ and $p\in S^2$ be an $i$-vertex. We define the \defn{$i$-flower} $W^i(p)$ as in \cite{BM10} by
$$
W^i(p) = \bigcup  \{ \inte c \,|\, c \in \DD^i, p\in c \}.
$$
Note that the number of $i$-tiles in $W^i(p)$ is $2\deg_{f^i}(p)$, i.e.,
\begin{equation}   \label{eqCardFlower}
\card \{X\in\X^i \,|\, p\in X   \}  =  2  \deg_{f^i}(p).
\end{equation}

By \cite[Lemma~7.11]{BM10}, there exists a constant $\beta\in\N$, which depends only on $f$ and $\CC$, such that for each $i\in\N$ and each $i$-tile $X\in\X^i$, $X$ can be covered by a union of at most $\beta$ $(i+1)$-flowers.

Fix an arbitrary $m\in\N_0$ with $m\equiv 0 \pmod N$, and fix an arbitrary $m$-edge $e$. 

By the induction hypothesis, there exist constants $D_\kappa >0$ and $L_\kappa \in [1,d)$ such that for each $k\in\N_0$ with $k\equiv \kappa +1 \pmod N$, there exists a collection $M_{m,k-1,e}$ of $(m+k-1)$-tiles with $\card M_{m,k-1,e} \leq D_\kappa L_\kappa^{k-1}$ and $e \subseteq \inte \Big(\bigcup\limits_{X\in M_{m,k-1,e}} X\Big)$. Each $X\in M_{m,k-1,e}$ can be covered by $\beta$ $(m+k)$-flowers $W^{m+k}(p)$. We can then construct a set $F\subseteq \V^{m+k}$ of $(m+k)$-vertices such that 
\begin{equation}   \label{eqPflmCoverEdges1}
\card F \leq \beta D_\kappa L_\kappa^{k-1}
\end{equation}
and
\begin{equation}
\bigcup\limits_{X\in M_{m,k-1,e}} X \subseteq \bigcup\limits_{p\in F} W^{m+k}(p).
\end{equation}
We define
\begin{equation}
M_{m,k,e}  = \{ X\in \X^{m+k} \,|\, X\cap F\neq \emptyset   \}.
\end{equation}
Then $e \subseteq \inte \Big( \bigcup\limits_{X\in M_{m,k,e}} X\Big)$, and by (\ref{eqCardFlower}),
\begin{equation}  \label{eqPflmCoverEdges2}
\card M_{m,k,e}  \leq \sum\limits_{p\in F} 2 \deg_{f^{m+k}}(p).
\end{equation}
 
Since $L_\kappa \in [1,d)$, there exists $K\in\N$, depending only on $f,\CC, m$, and $\kappa$, such that for each $i\geq K$, we have $\beta D_\kappa L_\kappa^{i-1} \leq d^{m+i}$.

Thus by (\ref{eqPflmCoverEdges1}), (\ref{eqPflmCoverEdges2}), and Lemma~\ref{lmSumLocalDegreeBound}, for each $k\geq K$ with $k\equiv \kappa + 1 \pmod N$, there exists constants $C>0$ and $\alpha\in(0,1]$, both of which depend only on $f$, such that
\begin{align}  \label{eqPflmCoverEdges3}
\card M_{m,k,e}  &\leq  2  \sum\limits_{p\in F} \deg_{f^{m+k}}(p) \notag\\
                 &\leq  2 C d^{(m+k)(1-\alpha)}  \( \beta D_\kappa L_\kappa^{k-1} \)^\alpha   \\
                 &   =  2 C d^{m(1-\alpha)} \beta^\alpha D_\kappa^\alpha L_\kappa^{-\alpha} \( d^{1-\alpha} L_\kappa^{\alpha} \)^k. \notag
\end{align}
Let $L_{\kappa +1} = d^{1-\alpha} L_\kappa^{\alpha}$. Since $L_\kappa \in [1,d)$, we get $L_{\kappa +1}\in [L_\kappa,d)\subseteq [1,d)$. Note that $L_{\kappa+1}$ only depends on $f,\CC$, and $\kappa$. We define
\begin{equation*}
\tau = \max \bigg\{ 2  \sum\limits_{p\in V} \deg_{f^{m+i}}(p)   \,|\, i \leq K, V\subseteq \V^{m+i}, \card V\leq  \beta D_\kappa L_\kappa^{k-1}  \bigg\}.
\end{equation*}
Since $\tau$ is the maximum over a finite set of numbers, $\tau< +\infty$. We set
\begin{equation} \label{eqPflmCoverEdges4}
D_{\kappa +1} = \max \{ \tau,  2 C d^{m(1-\alpha)} \beta^\alpha D_\kappa^\alpha L_\kappa^{-\alpha} \}.
\end{equation}
Then by (\ref{eqPflmCoverEdges2}), (\ref{eqPflmCoverEdges3}), and (\ref{eqPflmCoverEdges4}), we get that for each $k\in\N_0$ with $k\equiv \kappa+1 \pmod N$,
\begin{equation}
\card M_{m,k,e}  \leq \sum\limits_{p\in F} 2 \deg_{f^{m+k}}(p) \leq D_{\kappa+1} L_{\kappa +1}^k.
\end{equation}
We note that $\tau$ only depends on $f,\CC, m$, and $\kappa$, so $D_{\kappa+1}$ also only depends on $f,\CC,m$, and $\kappa$.

This completes the induction.

\smallskip

Step 3: Now we define
$$
L=\max\{L_\kappa \,|\, \kappa\in\{0,1,\dots,N-1\}   \}.
$$
For each fixed $m\in\N_0$ with $m\equiv 0 \pmod N$, we set
$$
D=\max\{D_\kappa \,|\, \kappa\in\{0,1,\dots,N-1\}   \},
$$
and for each given $k\in\N_0$ and $e\in\E^m$, let $M=M_{m,k,e}$. Then we have $\card M \leq DL^k$ and $e\subseteq \inte \Big(\bigcup\limits_{X\in M} X\Big)$. We note that here $L$ only depends on $f$ and $\CC$, and on the other hand, $D$ only depends on $f,\CC$, and $m$. The proof is now complete.
\end{proof}

\begin{rems}
It is also possible to prove the previous lemma by observing that $\CC$ equipped with the restriction of a visual metric $d$ for $f$ is a quasicircle (see \cite[Theorem~1.8]{BM10}), and $S^2$ equipped with $d$ is linearly locally connected (see \cite[Proposition~16.3]{BM10}). A metric space $X$, that is homeomorphic to the plane and with $\overline{X}$ linearly locally connected and $\partial X$ a Jordan curve, has the property that $\partial X$ is porous in $\overline{X}$ (see \cite[Theorem~IV.14]{Wi07}). Then we can mimic the original proof of Lemma~20.2 in \cite{BM10}. Our proof adopted above is more elementary and self-contained.
\end{rems}

We are finally ready to prove the equidistribution of the preperiodic points with respect to the measure of maximal entropy $\mu_f$. 

\begin{proof}[Proof of Theorem~\ref{thmWeakConvPrePerPts}]
Fix an arbitrary $N \geq N(f)$ where $N(f)$ is an constant as given in Corollary~\ref{corCexists} depending only on $f$. We also fix an $f^N$-invariant Jordan curve $\CC$ containing $\post f$ such that no $N$-tile in $\DD^N(f,\CC)$ joins opposite sides of $\CC$ as given in Corollary~\ref{corCexists}. In the proof below, we consider the cell decompositions $\DD^i(f,\CC), i\in\N_0,$ induced by $(f,\CC)$, and denote $d=\deg f$. 

Since $\xi_n^m$ and $\widetilde\xi_n^m$ are Borel probability measures for all $m\in\N_0$ and $n\in\N$ with $m<n$, by Alaoglu's Theorem, it suffices to prove that in the weak$^*$ topology, every convergent subsequence of $\{\xi_n^{m_n} \}_{n\in\N}$ and $\{\widetilde\xi_n^{m_n} \}_{n\in\N}$ converges to $\mu_f$.

\smallskip

Proof of (\ref{eqWeakConvPrePerPtsWithWeight}):

\smallskip

Let $\{ n_i \}_{i\in\N}$ be a strictly increasing sequence with 
$$
\xi_{n_i}^{m_{n_i}}   \stackrel{w^*}{\longrightarrow} \mu \text{ as } i\longrightarrow +\infty,
$$
for some measure $\mu$. 

\smallskip

Case 1 for (\ref{eqWeakConvPrePerPtsWithWeight}): We assume in this case that there is no constant $K\in\N$ such that for all $i\in\N$, $n_i-m_{n_i} \leq K$. Then by choosing a subsequence of $\{ n_i \}_{i\in\N}$ if necessary, we can assume that $n_i - m_{n_i}\longrightarrow +\infty$ as $i\longrightarrow +\infty$.

Here is the idea of the proof in this case. By the spirit of Lemma~\ref{lmAtLeast1} and Lemma~\ref{lmAtMost1}, there is an almost bijective correspondence between the fixed points of $f^{n-m_n}$ and the $(n-m_n)$-tiles containing such points. The correspondence is particularly nice away from $\CC$. Thus there is almost a bijective correspondence between the preperiodic points in $S_n^{m_n}$ and the $n$-tiles containing such points. So if we can control the behavior near $\CC$, then Theorem~\ref{thmWeakConv} applies and we finish the proof in this case. Now the control we need is provided by Lemma~\ref{lmCoverEdges}.

Now we start to implement this idea. So we fix a 0-edge $e_0 \subseteq \CC$. We observe that for each $i\in\N$, we can pair a white $i$-tile $X_w^i\in\X_w^i$ and a black $i$-tile $X_b^i\in\X_b^i$ whose intersection $X_w^i\cap X_b^i$ is an $i$-edge contained in $f^{-i}(e_0)$.  There are a total of $d^i$ such pairs and each $i$-tile is in exactly one such pair. We denote by $\PP_i$ the collection of the unions $X_w^i\cup X_b^i$ of such pairs, i.e.,
$$
\PP_i=\{X_w^i\cup X_b^i  \,|\, X_w^i\in\X_w^i,X_b^i\in\X_b^i, X_w^i\cap X_b^i\cap f^{-i}(e_0) \in \E^i \}.
$$
We denote $\PP'_i=\{A\in \PP_i \,|\, A\cap\CC = \emptyset\}$.

By Lemma~\ref{lmCoverEdges}, there exists $1 \leq L < d$ and $C>0$ such that for each $i\in\N$ there exists a collection $M$ of $i$-tiles with $\card M\leq CL^i$ such that $\CC$ is contained in the interior of the set $\bigcup\limits_{X\in M} X$. Note that $L$ and $C$ are constants independent of $i$. Observe that for each $A\in \PP_i$ that does not contain any $i$-tile in the collection $M$, we have $A \cap X \subseteq \partial \Big( \bigcup\limits_{X\in M} X\Big)$ for each $X\in M$, so $A \cap \inte \Big( \bigcup\limits_{X\in M} X \Big)= \emptyset$. Since the number of distinct $A\in \PP_i$ that contains an $i$-tile in $M$ is bounded above by $CL^i$, we get
\begin{equation}   \label{eqCardP'_i}
\card(\PP'_i) \geq d^i -CL^i.
\end{equation}

Note that for each $i\in\N$ and each $A\in\PP'_i$, either $A\subseteq X_w^0$ or $A\subseteq X_b^0$ where $X_w^0$ (resp.\ $X_b^0$) is the white (resp.\ black) $0$-tile for $(f,\CC)$. So by Proposition~\ref{propCellDecomp}(i) and Brouwer's Fixed Point Theorem, there is a map $\tau\:\PP'_i\rightarrow P_{1,f^i}$ from $\PP'_i$ to the set of fixed points of $f^i$ such that $\tau(A)\in A$. Note if a fixed point $x$ of $f^i$ has weight $\deg_{f^i}(x)>1$, then $x$ has to be contained in $\post f \subseteq \CC$. Thus $\deg_{f^i}(\tau(A))=1$ for all $A\in\PP'_i$. 

If for some $A\in\PP'_i$, the point $\tau(A)$ were on the boundaries of the two $i$-tiles whose union is $A$, then $\tau(A)$ would have to be contained in $\CC$ since the boundaries are mapped into $\CC$ under $f^i$. Thus for each $A\in\PP'_i$, the point $\tau(A)$ is contained in the interior of one of the two $i$-tiles whose union is $A$. Hence $\tau$ is injective. Moreover, 
\begin{equation}  \label{eqDeg=1AwayFromC}
\deg_{f^{i+j}}(x) = 1 \text{ for each } j\in\N_0 \text{ and each } x \in \bigcup\limits_{A\in\PP'_i} f^{-j}(\tau(A)).
\end{equation}

For each $i\in\N$, we choose a map $\beta_{n_i} \: \X_w^{n_i} \rightarrow S^2$ by letting $\beta_{n_i}(X)$ be the unique point in $f^{-m_{n_i}}(\tau(A))\cap B$ where $B\in\PP_{n_i}$ with $X\subseteq B$, if there exists $A\in \PP'_{n_i-m_{n_i}}$ with $f^{m_{n_i}}(X) \subseteq A$; and by letting $\beta_{n_i}(X)$ be an arbitrary point in $X$ if there exists no $A\in \PP'_{n_i-m_{n_i}}$ with $f^{m_{n_i}}(X) \subseteq A$.  

We fix a visual metric $d$ for $f$ with an expansion factor $\Lambda >1$. Note that $\Lambda$ can be chosen to depend only on $f$ and $d$. Then $\diam_d(A) < c \Lambda^{-i}$ for each $i\in\N$, where $c\geq 1$ is a constant depending only on $f$, $d$, and $\CC$ (See Lemma~\ref{lmBMCellSizeBounds}(ii)). Define $\alpha_n=c\Lambda^{-n}$ for each $n\in\N$. Thus $\alpha_{n_i}$ and $\beta_{n_i}$ satisfies the hypothesis in Theorem~\ref{thmWeakConv}. Define $\mu_{n_i}$ as in Theorem~\ref{thmWeakConv}. Then
\begin{equation}  \label{eqPfThmWeakConvPrePerPts1}
\mu_{n_i}   \stackrel{w^*}{\longrightarrow} \mu_f \text{ as } i\longrightarrow +\infty,
\end{equation}
by Theorem~\ref{thmWeakConv}.

We claim that the total variation $\Norm{\mu_{n_i} - \xi_{n_i}^{m_{n_i}}}$ of $\mu_{n_i} - \xi_{n_i}^{m_{n_i}}$ converges to $0$ as $i \longrightarrow +\infty$.

Assuming the claim, then by (\ref{eqPfThmWeakConvPrePerPts1}), we can conclude that (\ref{eqWeakConvPrePerPtsWithWeight}) holds in this case.

To prove the claim, by Corollary~\ref{corNoPrePeriodicPts}, we observe that for each $i\in\N$,
\begin{align} \label{eqPfThmWeakConvPrePerPts2}
\Norm{\mu_{n_i}-\xi_{n_i}^{m_{n_i}}}   \leq  & \norm{\mu_{n_i} - \frac{1}{d^{n_i-m_{n_i}}} \sum\limits_{A\in\PP'_{n_i-m_{n_i}}} \frac{1}{d^{m_{n_i}}} \sum\limits_{q\in f^{-m_{n_i}}(\tau(A))}  \delta_q  }    \notag \\
                                         + & \norm{\(\frac{1}{d^{n_i}} - \frac{1}{d^{n_i} + d^{m_{n_i}}} \) \sum\limits_{A\in\PP'_{n_i-m_{n_i}}}   \sum\limits_{q\in f^{-m_{n_i}}(\tau(A))}  \delta_q  }    \\
                                         + & \norm{\frac{1}{d^{n_i-m_{n_i}} + 1} \sum\limits_{A\in\PP'_{n_i-m_{n_i}}} \frac{1}{d^{m_{n_i}}} \sum\limits_{q\in f^{-m_{n_i}}(\tau(A))}  \delta_q   - \xi_{n_i}^{m_{n_i}}  }. \notag
\end{align}
In the first term on the right-hand side of (\ref{eqPfThmWeakConvPrePerPts2}), each $\delta_q$ in the summations cancels with the corresponding term in the definition of $\mu_{n_i}$. So the first term on the right-hand side of  (\ref{eqPfThmWeakConvPrePerPts2}) is equal to the difference of the total variations of the two measures, which by (\ref{eqCardP'_i}), is
$$
\leq 1- \frac{(d^{n_i-m_{n_i}}-CL^{n_i-m_{n_i}})d^{m_{n_i}}}{d^{n_i}}  = C \(\frac{L}{d} \)^{n_i-m_{n_i}}.
$$
In the second term on the right-hand side of (\ref{eqPfThmWeakConvPrePerPts2}), the total number of terms in the summations is bounded above by $d^{n_i}$. So the second term on the right-hand-side of (\ref{eqPfThmWeakConvPrePerPts2}) is
$$
\leq \Abs{\frac{1}{d^{n_i}}  - \frac{1}{d^{n_i} + d^{m_i}} }  d^{n_i}.
$$
In the third term on the right-hand side of (\ref{eqPfThmWeakConvPrePerPts2}), by (\ref{eqDeg=1AwayFromC}), $\deg_{f^{n_i}}(q) =1$ for each $A\in\PP'_{n_i-m_{n_i}}$ and each $q\in f^{-m_{n_i}}(\tau(A))$. So by (\ref{eqDistrPrePerPts}) and Corollary~\ref{corNoPrePeriodicPts}, each $\delta_q$ in the summations cancels with the corresponding $\delta_q$ in $\xi_{n_i}^{m_{n_i}}$. So the third term on the right-hand-side of (\ref{eqPfThmWeakConvPrePerPts2}) is equal to the difference of the total variations of the two measures, which by (\ref{eqCardP'_i}) and Corollary~\ref{corNoPrePeriodicPts}, is
$$
\leq 1- \frac{(d^{n_i-m_{n_i}}-CL^{n_i-m_{n_i}})d^{m_{n_i}} }{(d^{n_i-m_{n_i}} + 1)d^{m_{n_i}}}  = \frac{1+CL^{n_i-m_{n_i}}}{d^{n_i-m_{n_i}} + 1}.
$$
Since $n_i - m_{n_i}\longrightarrow +\infty$ as $i\longrightarrow +\infty$, each term on the right-hand-side of (\ref{eqPfThmWeakConvPrePerPts2}) converges to $0$ as $i \longrightarrow +\infty$. So 
$$
\Norm{\mu_{n_i} - \xi_{n_i}^{m_{n_i}}}  \longrightarrow 0 \text{ as } i\longrightarrow +\infty
$$
as claimed.

\smallskip

Case 2 for (\ref{eqWeakConvPrePerPtsWithWeight}): We assume in this case that there is a constant $K\in\N$ such that for all $i\in\N$, $n_i-m_{n_i} \leq K$. Then by choosing a subsequence of $\{ n_i \}_{i\in\N}$ if necessary, we can assume that there exists some constant $l\in[0,K]$ such that for all $i\in\N$, $n_i-m_{n_i} = l$. Note that in this case, $m_{n_i} \longrightarrow +\infty$ as $i \longrightarrow +\infty$.

Then by Corollary~\ref{corNoPrePeriodicPts} and Theorem~\ref{thmNoFixedPts},
\begin{align*}
\xi_{n_i}^{m_{n_i}} & = \frac{1}{d^{m_{n_i}}(d^l +1 )} \sum\limits_{x\in S_{n_i}^{m_{n_i}}}  \deg_{f^{n_i}} (x) \delta_x  \\
                  & = \frac{1}{d^l+1}  \sum\limits_{y=f^l(y)} \deg_{f^l}(y) \bigg( \frac{1}{d^{m_{n_i}}} \sum\limits_{x\in f^{-m_{n_i}}(y)} \deg_{f^{m_{n_i}}} (x) \delta_x   \bigg).
\end{align*}
By Theorem~\ref{thmWeakConvPreImg}, for each $y\in S^2$,
$$
\frac{1}{d^{m_{n_i}}} \sum\limits_{x\in f^{-m_{n_i}}(y)} \deg_{f^{m_{n_i}}} (x) \delta_x    \stackrel{w^*}{\longrightarrow} \mu_f \text{ as } i\longrightarrow +\infty.
$$
So each term in the sequence $\{ \xi_{n_i}^{m_{n_i}}  \}_{i\in\N}$ is a convex combination of the corresponding terms in sequences of measures, each of which converges in the weak$^*$ topology to $\mu_f$. Hence by Lemma~\ref{lmConvexCombConv}, $\{ \xi_{n_i}^{m_{n_i}}  \}_{i\in\N}$ also converges to $\mu_f$ in the weak$^*$ topology. It then follows that $\mu=\mu_f$. Thus (\ref{eqWeakConvPrePerPtsWithWeight}) follows in this case.

\smallskip

Proof of (\ref{eqWeakConvPrePerPtsWoWeight}):

\smallskip

Let $\{ n_i \}_{i\in\N}$ be a strictly increasing sequence with 
$$
\widetilde\xi_{n_i}^{m_{n_i}}   \stackrel{w^*}{\longrightarrow} \widetilde\mu \text{ as } i\longrightarrow +\infty,
$$
for some measure $\widetilde\mu$. 

\smallskip

Case 1 for (\ref{eqWeakConvPrePerPtsWoWeight}): We assume in this case that there is no constant $K\in\N$ such that for all $i\in\N$, $n_i-m_{n_i} \leq K$. Then by choosing a subsequence of $\{ n_i \}_{i\in\N}$ if necessary, we can assume that $n_i - m_{n_i}\longrightarrow +\infty$ as $i\longrightarrow +\infty$.

The idea of the proof in this case is similar to that of the proof of Case 1 for (\ref{eqWeakConvPrePerPtsWithWeight}).

We use the same notation as in the proof of Case 1 for (\ref{eqWeakConvPrePerPtsWithWeight}). Then (\ref{eqWeakConvPrePerPtsWoWeight}) follows in this case if we can prove that $\Norm{\mu_{n_i} - \widetilde\xi_{n_i}^{m_{n_i}}}$ converges to $0$ as $i \longrightarrow +\infty$.

As before, we observe that
\begin{align} \label{eqPfThmWeakConvPrePerPts3}
\Norm{\mu_{n_i}-\widetilde\xi_{n_i}^{m_{n_i}}}   \leq  & \norm{\mu_{n_i} - \frac{1}{d^{n_i-m_{n_i}}} \sum\limits_{A\in\PP'_{n_i-m_{n_i}}} \frac{1}{d^{m_{n_i}}} \sum\limits_{q\in f^{-m_{n_i}}(\tau(A))}  \delta_q  }    \notag \\
                                         + & \norm{\(\frac{1}{d^{n_i}} - \frac{1}{\widetilde s_{n_i}^{m_{n_i}}} \) \sum\limits_{A\in\PP'_{n_i-m_{n_i}}} \sum\limits_{q\in f^{-m_{n_i}}(\tau(A))}  \delta_q  }    \\
                                         + & \norm{\frac{1}{\widetilde s_{n_i}^{m_{n_i}}} \sum\limits_{A\in\PP'_{n_i-m_{n_i}}} \sum\limits_{q\in f^{-m_{n_i}}(\tau(A))}  \delta_q   - \widetilde\xi_{n_i}^{m_{n_i}}  }. \notag
\end{align}
As the first term on the right-hand side of (\ref{eqPfThmWeakConvPrePerPts2}) discussed before, the first term on the right-hand-side of (\ref{eqPfThmWeakConvPrePerPts3}) is
$$
\leq 1 - \frac{(d^{n_i-m_{n_i}}-CL^{n_i-m_{n_i}})d^{m_{n_i}}}{d^{n_i}}  = C \(\frac{L}{d} \)^{n_i-m_{n_i}}.
$$
In the second term on the right-hand side of (\ref{eqPfThmWeakConvPrePerPts3}), the total number of terms in the summations is bounded above by $d^{n_i}$. By (\ref{eqCardP'_i}), (\ref{eqDeg=1AwayFromC}), and Corollary~\ref{corNoPrePeriodicPts}, we have
\begin{align}   \label{eqPfThmWeakConvPrePerPts4}
      & d^{m_{n_i}}(d^{n_i-m_{n_i}} +1)  = s_{n_i}^{m_{n_i}} \geq  \widetilde s_{n_i}^{m_{n_i}} \\
 \geq & d^{m_{n_i}} \card(\PP'_{n_i-m_{n_i}}) \geq d^{m_{n_i}}(d^{n_i-m_{n_i}}-CL^{n_i-m_{n_i}}).   \notag
\end{align}
So the second term on the right-hand-side of (\ref{eqPfThmWeakConvPrePerPts3}) is
$$
\leq  \Abs{\frac{1}{d^{n_i}} - \frac{1}{\widetilde s_{n_i}^{m_{n_i}}} } d^{n_i} = \Abs{1-\frac{d^{n_i}}{\widetilde s_{n_i}^{m_{n_i}}}}   \leq \max\Big\{ \frac{1}{d^{n_i-m_{n_i}}}, \frac{CL^{n_i-m_{n_i}}}{d^{n_i-m_{n_i}}} \Big\}.
$$
In the third term on the right-hand side of (\ref{eqPfThmWeakConvPrePerPts3}), by (\ref{eqDeg=1AwayFromC}), $\deg_{f^{n_i}}(q) =1$ for each $A\in\PP'_{n_i-m_{n_i}}$ and each $q\in f^{-m_{n_i}}(\tau(A))$. So by (\ref{eqDistrPrePerPts}), each $\delta_q$ in the summations cancels with the corresponding $\delta_q$ in $\widetilde\xi_{n_i}^{m_{n_i}}$. So the third term on the right-hand-side of (\ref{eqPfThmWeakConvPrePerPts3}) is equal to the difference of the total variations of the two measures, which by (\ref{eqPfThmWeakConvPrePerPts4}) and (\ref{eqCardP'_i}), for $n_i-m_{n_i}$ large enough, is
$$
\leq \frac{d^{n_i} + d^{m_{n_i}} -(d^{n_i - m_{n_i}}-CL^{n_i-m_{n_i}})d^{m_{n_i}} }{\widetilde s_{n_i}^{m_{n_i}}} \leq \frac{1+CL^{n_i-m_{n_i}}}{d^{n_i-m_{n_i}} -CL^{n_i-m_{n_i}}}.
$$
Since $n_i - m_{n_i}\longrightarrow +\infty$ as $i\longrightarrow +\infty$, each term on the right-hand-side of (\ref{eqPfThmWeakConvPrePerPts3}) converges to $0$ as $i\longrightarrow+\infty$. So we can conclude that
$$
\Norm{\mu_{n_i} - \widetilde \xi_{n_i}^{m_{n_i}}}  \longrightarrow 0 \text{ as } i\longrightarrow +\infty.
$$
So $\widetilde\mu=\mu_f$. Thus (\ref{eqWeakConvPrePerPtsWoWeight}) follows in this case.

\smallskip

Case 2 for (\ref{eqWeakConvPrePerPtsWoWeight}): We assume in this case that there is a constant $K\in\N$ such that for all $i\in\N$, $n_i-m_{n_i} \leq K$. Then by choosing a subsequence of $\{ n_i \}_{i\in\N}$ if necessary, we can assume that there exists some constant $l\in[0,K]$ such that for all $i\in\N$, $n_i-m_{n_i} = l$. Note that in this case, $m_{n_i} \longrightarrow +\infty$ as $i \longrightarrow +\infty$.

Then for each $i\in\N$, we have
$$
\widetilde\xi_{n_i}^{m_{n_i}}   = \frac{1}{\widetilde s_{n_i}^{m_{n_i}}}  \sum\limits_{x\in S_{n_i}^{m_{n_i}}} \delta_x = \frac{1}{\widetilde s_{n_i}^{m_{n_i}}}  \sum\limits_{y=f^l(y)} Z_{m_{n_i},y} \bigg( \frac{1}{Z_{m_{n_i},y}} \sum\limits_{x\in f^{-m_{n_i}}(y)} \delta_x  \bigg),
$$
where $Z_{m,y} = \card \( f^{-m}(y) \)$ for each $y\in S^2$ and each $m\in\N_0$. Note that for each $i\in\N$, we have 
$$
\widetilde s_{n_i}^{m_{n_i}} = \sum\limits_{y=f^l(y)} Z_{m_{n_i},y}.
$$
Denote, for each $i\in\N$ and each $y\in S^2$, the Borel probability measure $\mu_{i,y}=  \frac{1}{Z_{m_{n_i},y}} \sum\limits_{x\in f^{-m_{n_i}}(y)} \delta_x$. Then by Theorem~\ref{thmWeakConvPreImg}, we have 
$$
\mu_{i,y}  \stackrel{w^*}{\longrightarrow} \mu_f \text{ as } i\longrightarrow +\infty.
$$
So each term in $\{ \widetilde \xi_{n_i}^{m_{n_i}}  \}_{i\in\N}$ is a convex combination of the corresponding terms in sequences of measures, each of which converges in the weak$^*$ topology to $\mu_f$. Hence by Lemma~\ref{lmConvexCombConv}, $\{ \widetilde \xi_{n_i}^{m_{n_i}}  \}_{i\in\N}$ also converges to $\mu_f$ in the weak$^*$ topology. It then follows that $\widetilde\mu=\mu_f$. Thus (\ref{eqWeakConvPrePerPtsWoWeight}) follows in this case.
\end{proof}

The proof of Theorem~\ref{thmWeakConvPrePerPts} also gives us the following corollary.

\begin{cor}   \label{corAsympRatioNoPrePerPts}
Let $f$ be an expanding Thurston map. If $\{m_n\}_{n\in\N}$ is a sequence in $\N_0$ such that $m_n <n$ for each $n\in\N$ and $\lim\limits_{n\to+\infty} n-m_n = +\infty$, then
\begin{equation}  \label{eqAsympRatioNoPrePerPts}
\lim\limits_{n\to+\infty}  \frac{\widetilde s_n^{m_n}}{s_n^{m_n}} = 1.
\end{equation}
\end{cor}

\begin{proof}
By the proof of Theorem~\ref{thmWeakConvPrePerPts}, especially (\ref{eqPfThmWeakConvPrePerPts4}), we get that for each $n\in\N$,
\begin{equation}
\frac{d^{n-m_n}-CL^{n-m_n}}{d^{n-m_n}+1}   \leq  \frac{\widetilde s_n^{m_n}}{s_n^{m_n}} \leq 1,
\end{equation}
where $d=\deg f$.
Then (\ref{eqAsympRatioNoPrePerPts}) follows from the fact that $1\leq L <d$ and the condition that $\lim\limits_{n\to+\infty} n-m_n = +\infty$.
\end{proof}

By (\ref{eqTopEntropy}), Theorem~\ref{thmNoFixedPts}, and Corollary~\ref{corAsympRatioNoPrePerPts} with $m_n=0$ for each $n\in\N$, we get the following corollary, which is an analog of the corresponding result for expansive homeomorphisms on compact metric spaces with the \emph{specification property} (see, for example, \cite[Theorem~18.5.5]{KH95}).

\begin{cor}   \label{corAsympRatioNoPerPts}
Let $f$ be an expanding Thurston map. Then for each constant $c\in (0,1)$, there exists a constant $N\in\N$ such that for each $n \geq N$,
\begin{align*}
c e^{n h_{\operatorname{top}}(f)} = c (\deg f)^n & < \card \{x\in S^2 \,|\, f^n(x)=x \} \\
                                                 & \leq \sum\limits_{x=f^n(x)}\deg_{f^n}(x) = (\deg f)^n +1 < \frac{1}{c} e^{n h_{\operatorname{top}}(f)}.
\end{align*}
In particular, 
$$
\lim\limits_{n\to +\infty}\frac{\card \{x\in S^2 \,|\, f^n(x)=x \}}{\exp\(n h_{\operatorname{top}}(f)\)} = \lim\limits_{n\to +\infty}\frac{\card \{x\in S^2 \,|\, f^n(x)=x \}}{(\deg f)^n } = 1.
$$
\end{cor}

Finally, we get the equidistribution of the periodic points with respect to the measure of maximal entropy $\mu_f$ as an immediate corollary.

\begin{proof}[Proof of Corollary~\ref{corWeakConvPerPts}]
We get (\ref{eqWeakConvPerPts1}) and (\ref{eqWeakConvPerPts2}) from Theorem~\ref{thmWeakConvPrePerPts} with $m_n=0$ for all $n\in\N$. Then (\ref{eqWeakConvPerPts3}) follows from (\ref{eqWeakConvPerPts2}) and Corollary~\ref{corAsympRatioNoPerPts}.
\end{proof}

\section{Expanding Thurston maps as factors of the left-shift}  \label{sctFactor}

M.~Bonk and D.~Meyer \cite{BM10} proved that for an expanding Thurs\-ton map $f$, the topological dynamical system $(S^2,f)$ is a factor of a certain classical topological dynamical system, namely, the left-shift on the one-sided infinite sequences of $\deg f$ symbols. The goal of this section is to generalize this result to the category of measure-preserving dynamical systems. The invariant measure for each measure-preserving dynamical system considered in this section is going to be the unique measure of maximal entropy of the corresponding system.

Let $X$ and $\widetilde{X}$ be topological spaces, and $f\:X\rightarrow X$ and $\widetilde{f}\:\widetilde{X}\rightarrow\widetilde{X}$ be continuous maps. We say that the topological dynamical system $(X,f)$ is a \defn{factor of the topological dynamical system} $(\widetilde{X},\widetilde{f})$ if there is a surjective continuous map $\varphi\:\widetilde X\rightarrow X$ such that $\varphi\circ\widetilde{f}=f\circ\varphi$. For measure-preserving dynamical systems $(X,g,\mu)$ and $(\widetilde{X},\widetilde{g},\widetilde{\mu})$ where $X$ and $\widetilde{X}$ are measure spaces, $g\:X\rightarrow X$ and $\widetilde{g}\:\widetilde{X}\rightarrow\widetilde{X}$ measurable maps, and $\mu\in \MMM(X,g)$ and $\widetilde{\mu}\in \MMM(\widetilde{X},\widetilde{g})$, we say that the measure-preserving dynamical system $(\widetilde{X},\widetilde{g},\widetilde{\mu})$ is a \defn{factor of the measure-preserving dynamical system} $(X,g,\mu)$ if there is a measurable map $\varphi\:\widetilde{X}\rightarrow X$ such that $\varphi\circ\widetilde{g}=g\circ\varphi$ and $\varphi_*\widetilde{\mu}=\mu$. Thus we get the following commutative diagram:
\begin{equation*}
    \xymatrix{
      \widetilde{X} \ar[r]^{\widetilde{f}} \ar@{->}[d]_\varphi &  \widetilde{X} \ar@{->}[d]^\varphi      \\
      X \ar[r]^{f} & X
    }
\end{equation*}

We recall a classical example of symbolic dynamical systems, namely $(J_k^\omega,\Sigma)$, where the \emph{alphabet} $J_k=\{0,1,\dots,k-1\}$ for some $k\in\N$, the \emph{set of infinite words} $J_k^\omega=\prod\limits_{i=1}^{+\infty} J_k$, and $\Sigma$ is the left-shift operator with
$$
\Sigma(i_1,i_2,\dots)=(i_2,i_3,\dots)
$$
for each $(i_i,i_2,\dots)\in J_k^\omega$. We equip $J_k^\omega$ with a metric $d$ such that the distance between two distinct infinite words $(i_1,i_2,\dots)$ and $(j_1,j_2,\dots)$ is $\frac{1}{m}$, where $m=\min\{n\in\N\,|\,i_n\neq j_n\}$.

Define the \emph{set of words of length $n$} as $J_k^n = \prod_{i=1}^n  J_k$, for $n\in\N$ and $J_k^0=\{\emptyset\}$ where $\emptyset$ is considered as a word of length $0$. Denote the \emph{set of finite words} by $J_k^*=\bigcup\limits_{n=0}^{+\infty} J_k^n$. Then the left-shift operator $\Sigma$ is defined on $J_k^*\setminus J_k^0$ naturally by 
$$
\Sigma(i_1,i_2,\dots,i_n)=(i_2,i_3,\dots,i_n).
$$

It is well-known that the dynamical system $(J_k^\omega,\Sigma)$ has a unique measure of maximal entropy $\mu_\Sigma$, which is characterized by the property that
$$
\mu_\Sigma \(C(j_1,j_2,\dots,j_n)\) = k^{-n},
$$
for $n\in\N$ and $j_1,j_2,\dots,j_n \in J_k$, where
\begin{equation}  \label{eqCylinder}
C(j_1,j_2,\dots,j_n) = \{ (i_1,i_2,\dots)\in J_k^\omega \,|\, i_1=j_1,i_2=j_2,\dots,i_n=j_n \}
\end{equation}
is the \defn{cylinder set} determined by $j_1,j_2,\dots,j_n$ (see for example, \cite[Section~4.4]{KH95}).

We will prove that for each expanding Thurston map $f$ with $\deg f=k$ and its measure of maximal entropy $\mu_f$, the measure-preserving dynamical system $(S^2,f,\mu_f)$ is a factor of the system $(J_k^\omega,\Sigma,\mu_\Sigma)$.

We now review a construction from \cite{BM10} for the convenience of the reader.

Let $f\:S^2\rightarrow S^2$ be an expanding Thurston map, and $\CC\subseteq S^2$ a Jordan curve with $\post f\subseteq \CC$. Consider the cell decompositions induced by the pair $(f,\CC)$. Let $k=\deg f$. Fix an arbitrary point $p\in \inte X_w^0$. Let $q_1,q_2,\dots,q_k$ be the distinct points in $f^{-1}(p)$. For $i=1,\dots,k$, we pick a continuous path $\alpha_i\:[0,1] \rightarrow S^2\setminus \post f$ with $\alpha_i(0)=p$ and $\alpha_i(1)=q_i$.

We construct $\psi \:J_k^* \rightarrow S^2$ inductively such that $\psi(I)\in f^{-n}(p)$, for each $n\in\N_0$ and $I\in J_k^n$, in the following way:

Define $\psi(\emptyset)=p$, and $\psi((i))=q_i$ for each $(i)\in J_k^1$. Suppose that $\psi$ has been defined for all $I\in \bigcup\limits_{j=0}^n  J_k^j$, where $n\in\N$. Now for each $(i_1,i_2,\dots,i_{n+1})\in J_k^{n+1}$, the point $\psi((i_1,i_2,\dots,i_{n}))\in f^{-n}(p)$ has already been defined. Since $f^n(\psi((i_1,i_2,\dots,i_{n})))=p$ and $f^n\:S^2\setminus f^{-n}(\post f) \rightarrow  S^2\setminus \post f$ is a covering map, the path $\alpha_{i_{n+1}}$ has a unique lift $\widetilde{\alpha}_{i_{n+1}}\:[0,1]\rightarrow S^2$ with $\widetilde{\alpha}_{i_{n+1}} (0) = \psi((i_1,i_2,\dots,i_n))$ and $f^n\circ \widetilde{\alpha}_{i_{n+1}} = \alpha_{i_{n+1}}$. We now define $\psi((i_1,i_2,\dots,i_{n+1}))=  \widetilde{\alpha}_{i_{n+1}}(1)$. Note that then
\begin{align*}
  &f^{n+1}(\psi((i_1,i_2,\dots,i_{n+1})))\\
= &f^{n+1}( \widetilde{\alpha}_{i_{n+1}}(1)) = f( \alpha_{i_{n+1}}(1)) = f(q_{i_{n+1}}) =p.
\end{align*}
Hence $ \psi((i_1,i_2,\dots,i_{n+1})) \in f^{-(n+1)}(p)$. This completes the inductive construction of $\psi$.

Note that $\psi\: J_k^* \rightarrow S^2$ induces a map $\widetilde\psi \: J_k^* \rightarrow \bigcup\limits_{n=0}^{+\infty} \X_w^n$ by mapping each $(i_1,i_2,\dots,i_n)\in J_k^n$ to the unique white $n$-tile $X_w^n \in \X_w^n$ containing $\psi((i_1,i_2,\dots,i_n)) \in f^{-n}(p)$. 

By the proof of Theorem~1.6 in Chapter 9 of \cite{BM10}, for each $n\in\N$, $\psi |_{J_k^n}\:J_k^n \rightarrow f^{-n}(p)$ is a bijection. Hence $\widetilde\psi|_{J^n_k} \: J_k^n \rightarrow \X_w^n$ for $n\in\N_0$, and $\widetilde\psi\: J_k^* \rightarrow \bigcup\limits_{n=0}^{+\infty} \X_w^n$ are also bijections. Moreover, by the proof of Theorem~1.6 in \cite{BM10}, we have that for each $(i_1,i_2,\dots)\in J_k^\omega$, $\{\psi((i_1,i_2,\dots,i_n))\}_{n\in\N}$ is a Cauchy sequence in $(S^2,d)$, for each visual metric $d$ for $f$. So as shown in the proof of Theorem~1.6 in \cite{BM10}, the map  $\varphi \:J_k^\omega \rightarrow S^2$ defined by
\begin{equation}
\varphi ((i_1,i_2,\dots))  = \lim\limits_{n\to+\infty} \psi((i_1,i_2,\dots,i_n))
\end{equation}
satisfies
\begin{enumerate}

\smallskip
\item $\varphi$ is continuous,

\smallskip
\item $f \circ \varphi = \varphi \circ \Sigma$,

\smallskip
\item $\varphi \:J_k^\omega \rightarrow S^2$ is surjective.
\end{enumerate}

So we can now reformulate Theorem~1.6 from \cite{BM10} in the following way.

\begin{theorem}[M.~Bonk \& D.~Meyer 2010] \label{thmBMfactor}  
Let $f\:S^2\rightarrow S^2$ be an expanding Thurston map with $\deg f = k$. Then $(S^2,f)$ is a factor of the topological dynamical system $(J_k^{\omega},\Sigma)$. More precisely, the surjective continuous map $\varphi\: J_k^\omega \rightarrow S^2$ defined above satisfies $f\circ \varphi =\varphi \circ \Sigma$.
\end{theorem}

We will strengthen Theorem~\ref{thmBMfactor} in the following theorem.

\begin{theorem} \label{thmLfactor}
Let $f\:S^2\rightarrow S^2$ be an expanding Thurston map with $\deg f = k$. Then $(S^2,f,\mu_f)$ is a factor of the measure-preserving dynamical system $(J_k^{\omega},\Sigma,\eta_\Sigma)$, where $\mu_f$ and $\eta_\Sigma$ are the unique measures of maximal entropy of $(S^2,f)$ and $(J_k^\omega,\Sigma)$, respectively. More precisely, the surjective continuous map $\varphi\: J_k^\omega \rightarrow S^2$ defined above satisfies $f\circ \varphi =\varphi \circ \Sigma$ and $\phi_*\eta_\Sigma=\mu_f$.
\end{theorem}

\begin{proof}
Let $\CC \subseteq S^2$ be a Jordan curve containing $\post f$. Let $d$ be a visual metric on $S^2$ for $f$ with an expansion factor $\Lambda>1$. Note that $\Lambda$ can be chosen to depend only on $f,d$, and $\CC$. Consider the cell decompositions induced by $(f,\CC)$.

By Theorem~\ref{thmBMfactor}, it suffices to prove that $\phi_*\nu=\mu_f$.

For each $n \in\N$, we fix a function $\widetilde\beta_n\: J_k^n\rightarrow J_k^\omega$ which maps each $(i_1,i_2,\dots,i_n)\in J_k^n$ to $(i_1,i_2,\dots,i_n,i_{n+1},\dots)\in J_k^\omega$, for some arbitrarily chosen $i_{n+1},i_{n+2},\dots \in J_k$ depending on $i_1,i_2,\dots, i_n$. In other words, $\widetilde\beta_n$ extends a finite word of length $n$ to an arbitrary infinite word. 

Define $\beta_n = \varphi \circ \widetilde\beta_n \circ \widetilde\psi^{-1}$, for each $n\in \N$, where $\widetilde\psi$ is defined earlier in this section.

\smallskip

We claim that the maps $\beta_n \:\X_w^n \rightarrow S^2$ with $n\in\N$ satisfy the hypothesis for $\beta_n$ in Theorem~\ref{thmWeakConv}, namely,
$$
\max \{ d(\beta_n(X_w^n), X_w^n)  \,|\, X_w^n \in \X_w^n    \} \longrightarrow 0 \text{ as } n\longrightarrow +\infty.
$$

Indeed, by the construction of $\varphi,\widetilde\beta_n$, $\psi$ and $\widetilde\psi$ above, we have that $\beta_n$ maps a white $n$-tile $X_w^n$ to the limit of a Cauchy sequence 
\begin{equation*}
\(\psi((j_1,j_2,\dots,j_m)) \)_{m\in\N}
\end{equation*}
such that $\psi((j_1,j_2,\dots,j_n))\in X_w^n$. Since for each $m\in\N$,  the points $\psi((j_1,j_2,\dots,j_m))$ and $\psi((j_1,j_2,\dots,j_{m+1}))$ are joined by a lift of one of the paths $\alpha_1,\alpha_2,\dots,\alpha_k$ (defined above) by $f^m$, by Lemma~8.11 in \cite{BM10}, we have that 
$$
d\(\psi((j_1,j_2,\dots,j_m)),\psi((j_1,j_2,\dots,j_{m+1}))\) \leq C \Lambda^{-m},
$$
for all $m\in\N$, where $C>0$ and $\Lambda >1$ are constants depending only on $f,\CC$, and $d$, in particular, independent of $m$ and $(j_1,j_2,\dots)\in J_k^\omega$. So $d(\beta_n(X^n_w),X^n_w) \leq C\frac{\Lambda^n}{1-\Lambda}$ for each $n\in\N$ and each $X^n_w\in\X^n_w$. The above claim follows.

\smallskip

For $i\in\N$, define
$$
\eta_i= \frac{1}{k^i} \sum\limits_{I\in J_k^i} \delta_{\widetilde\beta_i(I)}.
$$
Observe that for all $n\in\N$ and $m\in\N$ with $m\geq n$, and each $(i_1,i_2,\dots,i_n)\in J_k^n$, we have
\begin{equation*}
\mu_m(C(i_1,i_2,\dots,i_n)) = \mu_\Sigma(C(i_1,i_2,\dots,i_n)),
\end{equation*}
where $C(i_1,i_2,\dots,i_n)$ is defined in (\ref{eqCylinder}). So by the uniform continuity of each continuous function on $J_k^\omega$, it is easy to see that
\begin{equation} \label{eqConvLS}
\eta_i \stackrel{w^*}{\longrightarrow} \eta_\Sigma \text{ as } i\longrightarrow +\infty.
\end{equation}

Note that since $\widetilde\psi|_{J^n_k} \: J_k^n \rightarrow \X_w^n$ is a bijection for each $n\in\N_0$, we have for each $i\in\N$,
\begin{equation*}
\varphi_* \eta_i  
= \frac{1}{k^i} \sum\limits_{I\in J_k^i} \delta_{\varphi \circ\widetilde\beta_i(I)}
= \frac{1}{k^i} \sum\limits_{X^i\in \X_w^i} \delta_{\varphi \circ \widetilde\beta_i \circ \widetilde\psi^{-1}(X^i)}
= \frac{1}{k^i} \sum\limits_{X^i\in \X_w^i} \delta_{\beta_i(X^i)}.
\end{equation*}

Hence, by Theorem~\ref{thmWeakConv},
\begin{equation} \label{eqConvPushForwd}
\varphi_*\eta_i \stackrel{w^*}{\longrightarrow} \mu_f \text{ as } i\longrightarrow +\infty.
\end{equation}

Therefore, by (\ref{eqConvLS}), (\ref{eqConvPushForwd}), and Lemma~\ref{lmPushforwardConv}, we can conclude that $\phi_*\eta_\Sigma=\mu_f$.
\end{proof}

\section{A random iteration algorithm for producing the measure of maximal entropy}  \label{sctIteration}
In this section, we follow the idea of \cite{HT03} to prove that for each $p\in S^2$, the measure of maximal entropy $\mu_f$ of an expanding Thurston map $f$ is almost surely the limit of
$$
\frac1n \sum\limits_{i=0}^{n-1} \delta_{q_i}
$$
as $n \longrightarrow +\infty$ in the weak* topology, where $q_0=p$, and $q_i$ is one of the points $x$ in $f^{-1}(q_{i-1})$, chosen with probability $\frac{\deg_f(x)}{\deg f}$, for each $i\in\N$.

To give a more precise formulation, we will use the language of Markov process from the probability theory (see, for example, \cite{Du10} for an introduction).

\smallskip

Let $(X,d)$ be a compact metric space. Equip the space $\PPP(X)$ of probability measures with the weak$^*$ topology. A continuous map $X\rightarrow \PPP(X)$ assigning to each $x\in X$ a measure $\mu_x$ defines a \defn{random walk} on $X$. We define the corresponding \defn{Markov operator} $Q\:\CCC(X)\rightarrow\CCC(X)$ by
\begin{equation}
Q\phi(x)=\int \! \phi(y) \, \mathrm{d} \mu_x(y).
\end{equation}
Let $Q^*$ be the adjoint operator of $Q$, i.e., for each $\phi \in\CCC(X)$ and $\rho\in\PPP(X)$,
\begin{equation}
\int \! Q\phi \,\mathrm{d} \rho = \int \! \phi \, \mathrm{d} (Q^*\rho) .
\end{equation}

Consider a stochastic process $(\Omega,\mathcal{F},P)$, where
\begin{enumerate}

\smallskip
\item $\Omega = \{(\omega_0,\omega_1,\dots) \,|\, \omega_i\in X, i\in\N_0  \} =\prod\limits_{i=0}^{+\infty} X$, equipped with the product topology,

\smallskip
\item $\mathcal{F}$ is the Borel $\sigma$-algebra on $\Omega$,

\smallskip
\item $P\in\PPP(\Omega)$.

\end{enumerate}
This process is a \defn{Markov process with transition probabilities $\{\mu_x\}_{x\in X}$} if
\begin{equation}
P\{\omega_{n+1} \in A \,|\, \omega_0=z_0,\omega_1=z_1, \dots, \omega_n=z_n \} = \mu_{z_n}(A)
\end{equation}
for all $n\in\N_0$, Borel subsets $A\subseteq X$, and $z_0,z_1,\dots,z_n\in X$.

The transition probabilities $\{\mu_x\}_{x\in X}$ are determined by the operator $Q$ and so we can speak of a \defn{Markov process determined by $Q$}.

\smallskip

Let $f\:S^2\rightarrow S^2$ be an expanding Thurston map with $\deg f = k$. The continuous map $S^2 \rightarrow \PPP(S^2)$ assigning to each $x\in S^2$ the probability measure
\begin{equation}
\mu_x= \frac1k\sum\limits_{y\in f^{-1}(x)}  \deg_f(y) \delta_y
\end{equation}
induces the Markov operator $Q\:\CCC(S^2)\rightarrow\CCC(S^2)$ which satisfies
\begin{equation}  \label{eqQ}
Q\phi(x)=\frac1k\sum\limits_{y\in f^{-1}(x)}  \deg_f(y) \phi(y)
\end{equation}
for all $\phi\in\CCC(S^2)$ and $x\in S^2$. To show $Q$ is well-defined, we need to prove $Q\phi(x)$ is continuous in $x\in S^2$ for each $\phi\in \CCC(S^2)$. Indeed, by fixing an arbitrary Jordan curve $\CC\subseteq S^2$ containing $\post f$, we know for each $x$ in the white $0$-tile $X^0_w$, 
$$
Q\phi(x)=\frac1k\sum\limits_{X\in X^1_w}\phi(y_X), 
$$
where $y_X$ is the unique point contained in the white $1$-tile $X$ with the property that $f(y_X)=x$. If we move $x$ around continuously within $X^0_w$, then each $y_X$ moves around continuously within $X$. Thus $Q\phi(x)$ restricted to $X^0_w$ is continuous in $x$. Similarly, $Q\phi(x)$ restricted to $X^0_b$ is also continuous. Hence $Q\phi(x)$ is continuous in $x\in S^2$.

Fix an arbitrary $z\in S^2$. Then there exists a unique Markov process $(\Omega,\mathcal{F},P_z)$ determined by $Q$ with
\begin{enumerate}

\smallskip
\item $\Omega = \prod\limits_{i=0}^{+\infty} S^2$, equipped with the product topology,

\smallskip
\item $\mathcal{F}$ being the Borel $\sigma$-algebra on $\Omega$,

\smallskip
\item $P_z$ being a Borel probability measure on $\Omega$ satisfying
\begin{equation*}
P_z\{\omega_{n+1} \in A \,|\, \omega_0=z,\omega_1=z_1, \dots, \omega_n=z_n \} = \mu_{\omega_n}(A)
\end{equation*}
for all $n\in\N$, Borel subset $A\subseteq S^2$, and $z_1,z_2,\dots,z_n \in S^2$.
\end{enumerate}
The existence and uniqueness of $P_z$ follows from \cite[Theorem~1.4.2]{Lo77}. Since the Markov process $(\Omega,\mathcal{F},P_z)$ is determined by $f$ as well, we will also call $(\Omega,\mathcal{F},P_z)$ the \defn{Markov process determined by $f$}.

Now we can formulate our main theorem for this section.

\begin{theorem} \label{thmRandomIntConv}
Let $f\:S^2\rightarrow S^2$ be an expanding Thurston map with its measure of maximal entropy $\mu_f$. Let $(\Omega,\mathcal{F},P_z)$ be the  Markov process determined by $f$. Then for each $z\in S^2$, we have that $P_z$-almost surely,
\begin{equation}
\frac1n\sum\limits_{j=0}^{n-1} \delta_{\omega_j} \stackrel{w^*}{\longrightarrow} \mu_f \text{ as } n\longrightarrow +\infty.
\end{equation}
\end{theorem}

In other words, if we fix a point $z\in S^2$ and set it as the first point in an infinite sequence, and choose each of the following points randomly according to the Markov process determined by $f$, then $P_z$-almost surely, the probability measure equally distributed on the first $n$ points in the sequence converges in the weak$^*$ topology to $\mu_f$ as $n\longrightarrow +\infty$. 

In order to prove Theorem~\ref{thmRandomIntConv}, we need a theorem by H.~Furstenberg and Y.~Kifer from \cite{FK83}.

\begin{theorem}[H.~Furstenberg \& Y.~Kifer 1983]  \label{thmFK83}
Let $\Omega=\{\omega_n \in X\,|\, n\in\N_0\}$ be the Markov process determined by the operator $Q$. Assume that there exists a unique Borel probability measure $\mu$ that is invariant under the adjoint operator $Q^*$ on $\PPP(X)$. Then for each $\omega_0 \in X$, we have that $P_{\omega_0}$-almost surely,
\begin{equation}
\frac1n\sum\limits_{j=0}^{n-1} \delta_{\omega_j} \stackrel{w^*}{\longrightarrow} \mu \text{ as } n\longrightarrow +\infty.
\end{equation}
\end{theorem}

Theorem~\ref{thmRandomIntConv} follows immediately from Theorem~\ref{thmFK83} and the following lemma.

\begin{lemma}   \label{lmQInvMeasure}
Let $f\:S^2\rightarrow S^2$ be an expanding Thurston map. Then the unique measure of maximal entropy $\mu_f$ for $f$ is the only measure that is invariant under the adjoint operator $Q^*\:\PPP(S^2)\rightarrow\PPP(S^2)$ of $Q\:\CCC(S^2)\rightarrow\CCC(S^2)$, where $Q$ is defined in (\ref{eqQ}).
\end{lemma}

\begin{proof}
Let $k=\deg f$. Fix a Jordan curve $\CC\subseteq S^2$ with $\post f \subseteq \CC$. Let $d$ be a visual metric on $S^2$ for $f$ with an expansion factor $\Lambda>1$. Note that $\Lambda$ can be chosen to depend only on $f$ and $d$. Consider the cell decompositions induced by $(f,\CC)$.

Recall $\nu_n$ defined in (\ref{eqDistrPreImg}) for a fixed $p\in S^2$. Observe that $Q^*\nu_n=\nu_{n+1}$ for all $n\in\N_0$. By Theorem~\ref{thmWeakConvPreImg},
\begin{align*}
\int \! \varphi \, \mathrm{d} (Q^*\mu_f) & =  \int \! Q \varphi \, \mathrm{d} \mu_f = \lim\limits_{n\to+\infty} \int \! Q\varphi \, \mathrm{d} \nu_n = \lim\limits_{n\to+\infty} \int \!  \varphi \, \mathrm{d} Q^* \nu_n  \\
                                     &= \lim\limits_{n\to+\infty} \int \!  \varphi \, \mathrm{d} \nu_{n+1} = \int \!  \varphi \, \mathrm{d} \mu_f .
\end{align*}
Thus $Q^*\mu_f=\mu_f$, and so $\mu_f$ is indeed invariant under $Q^*$.

By (\ref{eqQ}), for each $x\in S^2$, each $n\in\N$ and each $\varphi\in\CCC(S^2)$, we have
\begin{equation}   \label{eqQnPhi}
Q^n\varphi(x)= \frac{1}{k^n}\sum\limits_{y\in f^{-n}(x)} \deg_{f^n}(y) \varphi(y).
\end{equation}
So by Theorem~\ref{thmWeakConvPreImg}, we get
\begin{equation}  \label{eqQConvPtw}
Q^n\varphi(x) -\int \! \varphi \, \mathrm{d} \mu_f \longrightarrow \ 0 \text{ as } n\longrightarrow +\infty.
\end{equation}

\smallskip

We claim that the convergence in (\ref{eqQConvPtw}) is uniform in $x$.

To prove the claim, we first assume that $x$ is in the (closed) white $0$-tile $X_w^0$. If we move $x$ around continuously within $X_w^0$, then each preimage of $x$ under $f^n$ moves around continuously within one of the white $n$-tiles $X_w^n\in\X_w^n$, for each $n\in\N$. By Lemma~\ref{lmBMCellSizeBounds}, there exists a constant $C\geq 1$ depending only on $f,\CC$, and $d$ such that $\diam_d(X^n_w) \leq C\Lambda^{-n}$ for each $n\in\N$ and each $X_w^n\in\X_w^n$. Then by the uniform continuity of $\varphi$ and (\ref{eqQnPhi}), we have that $Q^n\varphi(x)$ converges uniformly to $\int \! \varphi \, \mathrm{d} \mu_f$  over $X_w^0$ as $n \longrightarrow+\infty$. Similarly, we have that the convergence in (\ref{eqQConvPtw}) is uniform over the black $0$-tile $X_b^0$. Hence, the convergence in (\ref{eqQConvPtw}) is uniform over $S^2$. The claim is proved.

\smallskip

Suppose that $\mu\in\PPP(S^2)$ satisfies $Q^*\mu = \mu$. Then for each $\varphi\in\CCC(S^2)$, by the claim above, we have
\begin{align*}
  \int \! \varphi(x) \, \mathrm{d} \mu(x) = & \lim\limits_{n\to+\infty}   \int \! \varphi(x) \, \mathrm{d} (Q^*)^n\mu(x) \\
   = & \lim\limits_{n\to+\infty}   \int \! Q^n\varphi(x) \, \mathrm{d} \mu(x)\\
   = & \int \! \varphi(x) \, \mathrm{d} \mu_f(x).
\end{align*}
Hence $\mu=\mu_f$.
\end{proof}

As a special case of \cite[Theorem~3.4.11]{HP09}, the next corollary follows immediately from the uniform convergence in (\ref{eqQConvPtw}):

\begin{cor}   \label{corConvPullbackMeasure}
Let $f\:S^2\rightarrow S^2$ be an expanding Thurston map with its measure of maximal entropy $\mu_f$. Then for each Borel probability measure $\mu$ on $S^2$, we have
\begin{equation}   \label{eqConvPullbackMeasure}
\(Q^*\)^n \mu \stackrel{w^*}{\longrightarrow} \mu_f \text{ as } n\longrightarrow +\infty,
\end{equation}
where $Q^*\:\PPP(S^2)\rightarrow\PPP(S^2)$ is the adjoint operator of $Q\:\CCC(S^2)\rightarrow\CCC(S^2)$ defined in (\ref{eqQ}).
\end{cor}

\begin{proof}
By the claim proved in the proof of Lemma~\ref{lmQInvMeasure}, the convergence in (\ref{eqQConvPtw}) is uniform in $x$ for each $\varphi\in\CCC(S^2)$. Thus by integrating (\ref{eqQConvPtw}) over $S^2$ with respect to $\mu$, we get
\begin{equation*} 
\int \! Q^n\varphi \, \mathrm{d} \mu  \longrightarrow \ \int \! \varphi \, \mathrm{d} \mu_f  \text{ as } n\longrightarrow +\infty,
\end{equation*}
from which (\ref{eqConvPullbackMeasure}) follows.
\end{proof}

\begin{rem}
The operator $Q$ as defined in (\ref{eqQ}) is actually the Ruelle operator for an expanding Thurston map, in the special case when the potential is identically $0$. For some background of the Ruelle operator and the thermodynamical formalism, see, for example, \cite{Ru89, PU10}, and in the context of expanding Thurston maps, see \cite{Li13}. More generally, we prove in \cite{Li13} that for each expanding Thurston map and each H\"{o}lder continuous (with respect to any visual metric) potential $\phi$, there exists a unique equilibrium state, which is exact, and in particular, mixing. As a generalization of the measure of maximal entropy, an equilibrium state is an invariant probability measure that maximizes the pressure, which in turn is a generalization of the topological entropy. Moreover, we prove in \cite{Li13} that the equilibrium state is the unique probability measure invariant under the adjoint of the Ruelle operator $\mathcal{R}_{\widetilde\phi}$ corresponding to the H\"{o}lder continuous potential $\widetilde\phi$ determined by $\phi$. In the case when $\phi=0$, we have $\widetilde\phi=0$. Thus Lemma~\ref{lmQInvMeasure} follows from the more general result in \cite{Li13}. The direct proof of Lemma~\ref{lmQInvMeasure} we included above is much simpler though.
\end{rem}

\end{document}